\documentclass[12pt,reqno,a4paper]{amsart}
\usepackage{hyperref}
\usepackage{enumerate}
\usepackage[left=2.5cm, right=2.5cm, top= 3cm, bottom=5cm]{geometry}
\usepackage{amsmath,amssymb,amsthm,amsfonts}

\usepackage{soul}

\usepackage{xcolor}
\definecolor{myred}{rgb}{205, 2, 27}

\usepackage{tikz}
\usetikzlibrary{positioning}
\usepackage{verbatim}
\usetikzlibrary{arrows,shapes}
\usetikzlibrary{positioning, 
	quotes}
\usepackage{graphicx}
\usepackage{tikz}
\usetikzlibrary{graphs, graphs.standard}

\pgfdeclarelayer{background}
\pgfsetlayers{background,main}

\tikzstyle{vertex}=[circle,fill=black!25,minimum size=14pt,inner sep=0pt]
\tikzstyle{selected vertex} = [vertex, fill=red!24]
\tikzstyle{edge} = [draw,thick,-]
\tikzstyle{weight} = [font=\small]
\tikzstyle{selected edge} = [draw,line width=5pt,-,red!50]
\tikzstyle{ignored edge} = [draw,line width=5pt,-,black!20]
\newtheorem{theorem}{\bfseries Theorem}[section]
\newtheorem{proposition}[theorem]{\bfseries Proposition}
\newtheorem{lemma}[theorem]{\bfseries Lemma}
\newtheorem{claim}[theorem]{\bfseries Claim}

\theoremstyle{definition}  %plain, definition, remark
\newtheorem{remark}[theorem]{\bfseries Remark}
\newtheorem{definition}[theorem]{\bfseries Definition}

\newtheorem{example}[theorem]{\bfseries Example}
\newtheorem{conjecture}{\bfseries Conjecture}[section]

\newtheorem{manualtheoreminner}{Theorem}
\newenvironment{manualtheorem}[1]{%
	\renewcommand\themanualtheoreminner{#1}%
	\manualtheoreminner
}{\endmanualtheoreminner}

\usetikzlibrary{shapes}
\usetikzlibrary{plotmarks}
\usetikzlibrary{arrows}
\usetikzlibrary{positioning}
\title{\textbf{On the Cost Number of Graphs with Determining Number Two} \\ \vspace{10mm} \large{Alexa Gopaulsingh$^*$,  Zalán Molnár,   Amitayu Banerjee }}

\usepackage{graphicx}				% Use pdf, png, jpg, or eps§ with pdflatex; use eps in DVI mode
\usepackage{yfonts}
\usepackage{tikz}
\usepackage{caption}
\usepackage{amsmath,amsfonts,amssymb,amsthm,epsfig,epstopdf,titling,url,array,mathtools}
\usepackage{sidecap}
\usepackage{float}
\usepackage[utf8]{inputenc}
\usepackage{t1enc}
\usepackage{subcaption}
\usetikzlibrary{positioning,arrows,calc}								% TeX will automatically convert eps --> pdf in pdflatex		
\usepackage{wrapfig, framed, caption}
\usepackage[titletoc]{appendix}
\usepackage{enumitem}
\usepackage{faktor}
\usetikzlibrary{decorations.pathmorphing}

\def\>{\ensuremath\rangle}
\def\<{\ensuremath\langle}

\makeatletter
\@namedef{subjclassname@2020}{\textup {2020} Mathematics Subject Classification}
\makeatother

\begin{document}
	
	\maketitle

	\vspace{-1.cm}
	\begin{center}
		\noindent \textbf{Abstract}
	\end{center}
	{\small  
	A \textit{distinguishing vertex coloring} of a graph $G$ is a vertex coloring  such that only the identity automorphism of $G$ preserves the coloring. 
	A graph is \textit{$2$-distinguishable} if it admits a distinguishing vertex coloring with two colors, and its \textit{cost} $\rho(G)$ is the minimum size of a color class in such a coloring. 
	The determining number of a graph $G$, denoted by $Det(G)$, is the minimum size of a subset $S\subseteq V(G)$ such that only the trivial automorphism fixes every element of $S$ pointwise.
    Boutin (J. Combin. Math. Combin. Comput. 85: 161-171, 2013) asked if $\rho(G)$ and $Det(G)$ can be arbitrarily far apart.
    While the case for $Det(G) = 1$ is trivial, the answer remained unknown for $Det(G) \ge 2$.
    In this manuscript, we show that if $Det(G)=2$ then not only is $\rho(G)$  bounded, but in fact $\rho(G) \leq 4$. This is the first resolution of Boutin's question for any nontrivial fixed determining number. Moreover, for every fixed  $Det(G)= n$, we construct examples giving a lower bound on any possible upper bound for $\rho(G)$ in terms of $n$.
    
}
	
	\let\thefootnote\relax\footnotetext{\phantom{aa} {2020} Mathematics Subject Classification 05C15 05C25\\ \phantom{aa} *Corresponding author}
	%\textbf{Relation between the determining number and the cost number.}
	\section{Introduction}
All graphs considered in this paper are finite, simple, and undirected.
The notion of distinguishing colorings was introduced in 1977 by L\'{a}szl\'{o} Babai under the name of asymmetric colorings  \cite{Babai1977}. In 1996, Albertson and Collins \cite{AC1996} introduced the {\em distinguishing number} of a graph $G$, denoted by $D(G)$, as the least integer $d$ such that $G$ admits a vertex coloring $f$
such that only the identity automorphism of $G$ preserves $f$.
A graph that can be distinguished by a coloring using $d$ colors is said to be {\em $d$-distinguishable}. The problem of computing the distinguishing numbers of various graph families has since been extensively investigated in several studies, including those by Albertson and Boutin \cite{R1}, Imrich and Klav\v{z}ar \cite{IK2006}, Tymoczko \cite{R3}, and Kalinowski and Pil\'{s}niak \cite{D2a}.
The {\em cost number} of a $2$-distinguishable graph, denoted by $\rho(G)$, is the smallest size of a color class in a $2$-distinguishing coloring of $G$. Extensive research has established that wide classes of graphs are 2-distinguishable, including:
\begin{itemize}
    \item Hypercubes $Q_n$ for all $n \geq 4$ \cite{BC2004};
    \item Cartesian powers $G^n$ of a connected graph $G \neq K_2, K_3$ for all $n \geq 2$ \cite{IK2006};
    \item Kneser graphs $K_{n,k}$ for all $n \geq 6$ and $k \geq 2$ \cite{AB2007};
    \item 3-connected planar graphs, excluding seven small exceptions \cite{FNT2008}.
\end{itemize}

A fundamental concept for analyzing both distinguishing and cost numbers is the {\em determining number} of a graph $G$, denoted by $Det(G)$. 
This is the minimum size of a set of nodes $S$,  which is such that if $\varphi$ is a non-trivial automorphism of $G$, then $\varphi$ moves at least one node of $S$.
This concept was also introduced in the literature as the {\em fixing number} of a set; see \cite{Fixing1}, \cite{Fixing2}. Alikhani and Soltani in \cite{AS2021} studied some properties of the cost numbers of the graphs using their determining numbers.
It is known that for a fixed distinguishing number, one can construct graphs with arbitrarily large determining numbers. For graphs with $Det(G) = 1$, it is trivial to show that $\rho(G) = 1$. However, the behavior of $\rho(G)$ when $\text{Det}(G) \geq 2$ is unknown. In 2013, Boutin \cite{Bou2013} posed the following question:
\\		

   \textbf{Open question}: Find graphs for which $\rho(G)$ is arbitrarily larger than $Det(G)$. 
\subsection{Main Result}
In Theorem \ref{main_thm} in Section 3, we show that if $Det(G) = 2$, then  $\rho(G)$ is bounded and $\rho(G)\leq 4$. This is the first resolution of Boutin’s question for any nontrivial fixed
determining number.
In Section 2, we develop a helpful tool for this by considering the action of the automorphism group on the nodes of a graph (here we obtain Lemma \ref{L2} which is needed in our main proof). This will lead to the terminology of the \textit{automorphism representation} of a graph which encodes this action. It is important to note that two graphs having isomorphic automorphism groups is not enough to conclude that they have the same distinguishing number (and cost), an observation also noted by Tymoczko \cite{R3}. Theorem \ref{T1} and the remark following it show that equivalent automorphism representations give the same distinguishing number and, for 2-distinguishable graphs, the same cost number.  In Section 4, we also consider graphs with larger determining numbers. Here, we construct a family of 2-distinguishable graphs, $H_n$ with $Det(H_n) =n$ and $\rho(H_n)=n + f(n)$ for every integer $n \geq 1$, where $f(n)$ denotes the minimum total
number of $1$'s among any $n$ distinct binary strings of  length ${\lceil log_2{n} \rceil}$. Consequently, if the cost number of 2-distinguishable graphs is bounded in terms of their determining number, then any such bound must be at least $n + f(n)$ when $Det(G)=n$.  We conjecture that this lower bound is also a universal upper bound on $\rho(G)$ for any 2-distinguishable graph $G$ with determining number $n$ and that the graphs $H_n$ (cf. Example \ref{E4.1}) serve as the extremal examples demonstrating that this upper bound is sharp.

%Thus, while our main result gives a small upper bound for $Det(G)=2$, the examples in Section 4 show that substantially larger costs are possible for larger determining numbers.

    \section{Automorphism Representations}
  \noindent 
  In this section, we define the notion of two graphs being distinguishably equivalent, which will be used to establish our main result.
  
		\begin{definition} \label{D2.1}
			Let $G = (V(G), E(G))$ be a graph. A \textbf{labeling }of $G$ is an injective map, $l:V(G)\rightarrow S$, naming the nodes of $G$. 
			For each automorphism $\alpha\in Aut(G)$  expressed as a product of disjoint cycles, say $(v_{11} v_{12}...v_{1k_{1}})...(v_{m1} v_{m2}...v_{mk_{n}})$,  let $l(\alpha)$ denote $$(l(v_{11}) l(v_{12})...l(v_{1k_{1}}))...(l(v_{m1}) l(v_{m2})...l(v_{mk_{n}})).$$ 
			We define $Aut(G, V(G)_{l}):=\{l(\alpha):\alpha\in Aut(G)\}$ as \textbf{the automorphism representation of $G$ under the labeling $l$}.\\ 

			\noindent Given a coloring $c:V(G)\rightarrow C$ and a labeling $l:V(G)\rightarrow S$, let $ c_l:l(V(G))\rightarrow C $ be the induced coloring defined by $c_l(l(v))=c(v)$. When no confusion can arise, we use $c$ in place of $c_l$.
			%Note that any permutation can be represented as a product of disjoint cycles and we will always represent them this way throughout this paper unless stated otherwise.
			%if a permutation is expressed as a product of disjoint cycles, then we say that this representation is in a standard form. All representations in this paper are taken to be in standard form.
		\end{definition}
		
The notion of an automorphism representation considered here is closely related to the standard notion of a permutation representation (or permutation action) of a group on a set; in our setting, it is the natural action of $\operatorname{Aut}(G)$ on $V(G)$. Equivalent permutation actions are commonly considered up to permutation isomorphism, or equivalently, up to conjugacy in a symmetric group. We adopt the terminology \emph{automorphism representation} to formulate this viewpoint specifically for graphs.

		\begin{definition} \label{D2}

        We define relations on the classes of graphs and labeled representations: $G_1 \equiv G_2$ and  $ Aut(G_1, V(G_1)_{l_1}) \equiv Aut(G_2, V(G_2)_{l_2})$ if and only if  there exist labelings $l_1: V(G_1) \rightarrow S$ and $l_2: V(G_2) \rightarrow S$ such that  $l_1(V(G_1)) = l_2(V(G_2))$ and $ Aut(G_1, V(G_1)_{l_1}) = Aut(G_2, V(G_2)_{l_2})$. As these are equivalence relations, we call such $G_1$ and  $G_2$  \textbf{distinguishably equivalent graphs}  and the equivalence class of $Aut(G, V(G)_{l})$  is called the \textbf{automorphism representation} of $G$, and is denoted by $Aut(G, V(G))$.
		\end{definition}

		\begin{example} 
        \label{E1} Consider the  graphs $G_1$ and $G_2$ below. These are distinguishably equivalent graphs as witnessed by the following labelings and the automorphism representations corresponding to these:\\
		\end{example}

        \vspace{-3mm}
		
		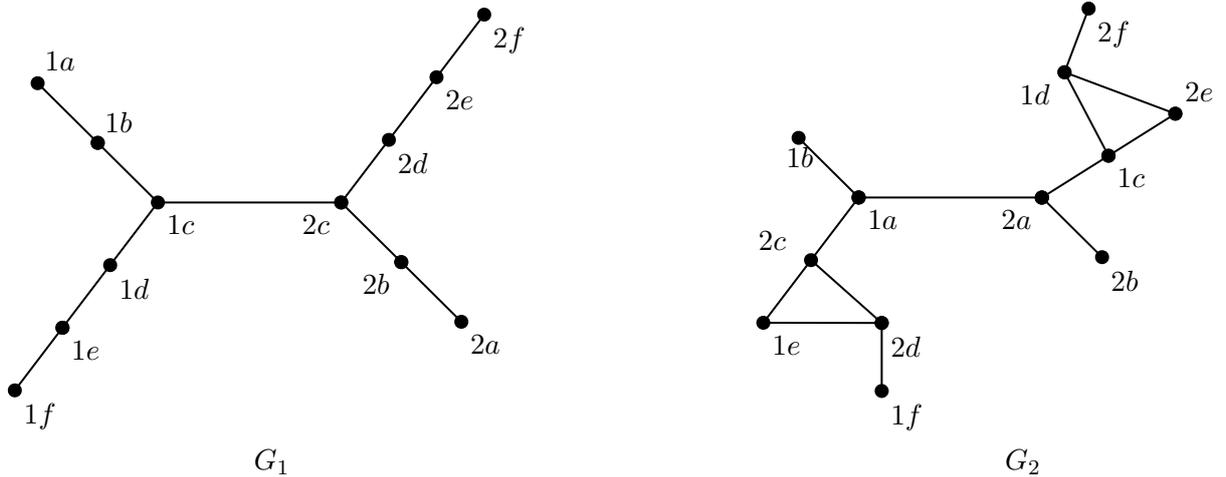
\begin{figure}[H]
		    \centering
		   	\vspace{-7mm}
\begin{center}
	\tikzset{every picture/.style={line width=0.75pt}} %set default line width to 0.75pt        
	
	\begin{tikzpicture}[x=0.75pt,y=0.75pt,yscale=-1.5,xscale=1.5]
		
		%Straight Lines [id:da42748347119362173] 
		\draw    (169,118) -- (230,118) ;
		\draw [shift={(230,118)}, rotate = 0] [color={rgb, 255:red, 0; green, 0; blue, 0 }  ][fill={rgb, 255:red, 0; green, 0; blue, 0 }  ][line width=0.75]      (0, 0) circle [x radius= 1.34, y radius= 1.34]   ;
		\draw [shift={(169,118)}, rotate = 0] [color={rgb, 255:red, 0; green, 0; blue, 0 }  ][fill={rgb, 255:red, 0; green, 0; blue, 0 }  ][line width=0.75]      (0, 0) circle [x radius= 1.34, y radius= 1.34]   ;
		%Straight Lines [id:da3311479467874803] 
		\draw    (153.13,139) -- (169,118) ;
		\draw [shift={(169,118)}, rotate = 307.07] [color={rgb, 255:red, 0; green, 0; blue, 0 }  ][fill={rgb, 255:red, 0; green, 0; blue, 0 }  ][line width=0.75]      (0, 0) circle [x radius= 1.34, y radius= 1.34]   ;
		\draw [shift={(153.13,139)}, rotate = 307.07] [color={rgb, 255:red, 0; green, 0; blue, 0 }  ][fill={rgb, 255:red, 0; green, 0; blue, 0 }  ][line width=0.75]      (0, 0) circle [x radius= 1.34, y radius= 1.34]   ;
		%Straight Lines [id:da7089004292619665] 
		\draw    (137.27,160) -- (153.13,139) ;
		\draw [shift={(153.13,139)}, rotate = 307.07] [color={rgb, 255:red, 0; green, 0; blue, 0 }  ][fill={rgb, 255:red, 0; green, 0; blue, 0 }  ][line width=0.75]      (0, 0) circle [x radius= 2.01, y radius= 2.01]   ;
		\draw [shift={(137.27,160)}, rotate = 307.07] [color={rgb, 255:red, 0; green, 0; blue, 0 }  ][fill={rgb, 255:red, 0; green, 0; blue, 0 }  ][line width=0.75]      (0, 0) circle [x radius= 2.01, y radius= 2.01]   ;
		%Straight Lines [id:da82439653116512] 
		\draw    (121.4,181) -- (137.27,160) ;
		\draw [shift={(137.27,160)}, rotate = 307.07] [color={rgb, 255:red, 0; green, 0; blue, 0 }  ][fill={rgb, 255:red, 0; green, 0; blue, 0 }  ][line width=0.75]      (0, 0) circle [x radius= 2.01, y radius= 2.01]   ;
		\draw [shift={(121.4,181)}, rotate = 307.07] [color={rgb, 255:red, 0; green, 0; blue, 0 }  ][fill={rgb, 255:red, 0; green, 0; blue, 0 }  ][line width=0.75]      (0, 0) circle [x radius= 2.01, y radius= 2.01]   ;
		%Straight Lines [id:da20330123192411254] 
		\draw    (149,98) -- (169,118) ;
		\draw [shift={(169,118)}, rotate = 45] [color={rgb, 255:red, 0; green, 0; blue, 0 }  ][fill={rgb, 255:red, 0; green, 0; blue, 0 }  ][line width=0.75]      (0, 0) circle [x radius= 2.01, y radius= 2.01]   ;
		\draw [shift={(149,98)}, rotate = 45] [color={rgb, 255:red, 0; green, 0; blue, 0 }  ][fill={rgb, 255:red, 0; green, 0; blue, 0 }  ][line width=0.75]      (0, 0) circle [x radius= 2.01, y radius= 2.01]   ;
		%Straight Lines [id:da7945842813988075] 
		\draw    (129,78) -- (149,98) ;
		\draw [shift={(149,98)}, rotate = 45] [color={rgb, 255:red, 0; green, 0; blue, 0 }  ][fill={rgb, 255:red, 0; green, 0; blue, 0 }  ][line width=0.75]      (0, 0) circle [x radius= 2.01, y radius= 2.01]   ;
		\draw [shift={(129,78)}, rotate = 45] [color={rgb, 255:red, 0; green, 0; blue, 0 }  ][fill={rgb, 255:red, 0; green, 0; blue, 0 }  ][line width=0.75]      (0, 0) circle [x radius= 2.01, y radius= 2.01]   ;
		%Straight Lines [id:da7877631333798232] 
		\draw    (250,138) -- (270,158) ;
		\draw [shift={(270,158)}, rotate = 45] [color={rgb, 255:red, 0; green, 0; blue, 0 }  ][fill={rgb, 255:red, 0; green, 0; blue, 0 }  ][line width=0.75]      (0, 0) circle [x radius= 2.01, y radius= 2.01]   ;
		\draw [shift={(250,138)}, rotate = 45] [color={rgb, 255:red, 0; green, 0; blue, 0 }  ][fill={rgb, 255:red, 0; green, 0; blue, 0 }  ][line width=0.75]      (0, 0) circle [x radius= 2.01, y radius= 2.01]   ;
		%Straight Lines [id:da702093731440804] 
		\draw    (230,118) -- (250,138) ;
		\draw [shift={(250,138)}, rotate = 45] [color={rgb, 255:red, 0; green, 0; blue, 0 }  ][fill={rgb, 255:red, 0; green, 0; blue, 0 }  ][line width=0.75]      (0, 0) circle [x radius= 2.01, y radius= 2.01]   ;
		\draw [shift={(230,118)}, rotate = 45] [color={rgb, 255:red, 0; green, 0; blue, 0 }  ][fill={rgb, 255:red, 0; green, 0; blue, 0 }  ][line width=0.75]      (0, 0) circle [x radius= 2.01, y radius= 2.01]   ;
		%Straight Lines [id:da1298751137444638] 
		\draw    (261.73,76) -- (277.6,55) ;
		\draw [shift={(277.6,55)}, rotate = 307.07] [color={rgb, 255:red, 0; green, 0; blue, 0 }  ][fill={rgb, 255:red, 0; green, 0; blue, 0 }  ][line width=0.75]      (0, 0) circle [x radius= 2.01, y radius= 2.01]   ;
		\draw [shift={(261.73,76)}, rotate = 307.07] [color={rgb, 255:red, 0; green, 0; blue, 0 }  ][fill={rgb, 255:red, 0; green, 0; blue, 0 }  ][line width=0.75]      (0, 0) circle [x radius= 2.01, y radius= 2.01]   ;
		%Straight Lines [id:da46740301394075634] 
		\draw    (245.87,97) -- (261.73,76) ;
		\draw [shift={(261.73,76)}, rotate = 307.07] [color={rgb, 255:red, 0; green, 0; blue, 0 }  ][fill={rgb, 255:red, 0; green, 0; blue, 0 }  ][line width=0.75]      (0, 0) circle [x radius= 1.34, y radius= 1.34]   ;
		\draw [shift={(245.87,97)}, rotate = 307.07] [color={rgb, 255:red, 0; green, 0; blue, 0 }  ][fill={rgb, 255:red, 0; green, 0; blue, 0 }  ][line width=0.75]      (0, 0) circle [x radius= 1.34, y radius= 1.34]   ;
		%Straight Lines [id:da006515066696156202] 
		\draw    (230,118) -- (245.87,97) ;
		\draw [shift={(245.87,97)}, rotate = 307.07] [color={rgb, 255:red, 0; green, 0; blue, 0 }  ][fill={rgb, 255:red, 0; green, 0; blue, 0 }  ][line width=0.75]      (0, 0) circle [x radius= 2.01, y radius= 2.01]   ;
		\draw [shift={(230,118)}, rotate = 307.07] [color={rgb, 255:red, 0; green, 0; blue, 0 }  ][fill={rgb, 255:red, 0; green, 0; blue, 0 }  ][line width=0.75]      (0, 0) circle [x radius= 2.01, y radius= 2.01]   ;

		% Text Node
		\draw (130.43,66.54) node [anchor=north west][inner sep=0.75pt]  [font=\small]  {$1a$};
		% Text Node
		\draw (150.43,86.54) node [anchor=north west][inner sep=0.75pt]  [font=\small]  {$1b$};
		% Text Node
		\draw (171,121.4) node [anchor=north west][inner sep=0.75pt]  [font=\small]  {$1c$};
		% Text Node
		\draw (155.13,142.4) node [anchor=north west][inner sep=0.75pt]  [font=\small]  {$1d$};
		% Text Node
		\draw (139.27,163.4) node [anchor=north west][inner sep=0.75pt]  [font=\small]  {$1e$};
		% Text Node
		\draw (123.4,184.4) node [anchor=north west][inner sep=0.75pt]  [font=\small]  {$1f$};
		
		\draw (200, 200) node [anchor=north west][inner sep=0.75pt]  [font=\small]  {$G_1$};
		
		\draw (450, 200) node [anchor=north west][inner sep=0.75pt]  [font=\small]  {$G_2$};

		% Text Node
		\draw (236.1,140.9) node [anchor=north west][inner sep=0.75pt]  [font=\small]  {$2b$};
		% Text Node
		\draw (272,161.4) node [anchor=north west][inner sep=0.75pt]  [font=\small]  {$2a$};
		% Text Node
		\draw (216.1,121.4) node [anchor=north west][inner sep=0.75pt]  [font=\small]  {$2c$};
		% Text Node
		\draw (247.87,100.4) node [anchor=north west][inner sep=0.75pt]  [font=\small]  {$2d$};
		% Text Node
		\draw (263.73,79.4) node [anchor=north west][inner sep=0.75pt]  [font=\small]  {$2e$};
		% Text Node
		\draw (279.6,58.4) node [anchor=north west][inner sep=0.75pt]  [font=\small]  {$2f$};
		% Text Node
		\draw (377.1,98.4) node [anchor=north west][inner sep=0.75pt]  [font=\small]  {$1b$};
		% Text Node
		
			\hspace{-.4cm}
		
		%Straight Lines [id:da27192199017335295] 
		\draw    (412.33,116.33) -- (473.33,116.33) ;
		\draw [shift={(473.33,116.33)}, rotate = 0] [color={rgb, 255:red, 0; green, 0; blue, 0 }  ][fill={rgb, 255:red, 0; green, 0; blue, 0 }  ][line width=0.75]      (0, 0) circle [x radius= 2.01, y radius= 2.01]   ;
		\draw [shift={(412.33,116.33)}, rotate = 0] [color={rgb, 255:red, 0; green, 0; blue, 0 }  ][fill={rgb, 255:red, 0; green, 0; blue, 0 }  ][line width=0.75]      (0, 0) circle [x radius= 2.01, y radius= 2.01]   ;
		%Straight Lines [id:da5730202129317246] 
		\draw    (396.47,137.33) -- (412.33,116.33) ;
		\draw [shift={(412.33,116.33)}, rotate = 307.07] [color={rgb, 255:red, 0; green, 0; blue, 0 }  ][fill={rgb, 255:red, 0; green, 0; blue, 0 }  ][line width=0.75]      (0, 0) circle [x radius= 2.01, y radius= 2.01]   ;
		\draw [shift={(396.47,137.33)}, rotate = 307.07] [color={rgb, 255:red, 0; green, 0; blue, 0 }  ][fill={rgb, 255:red, 0; green, 0; blue, 0 }  ][line width=0.75]      (0, 0) circle [x radius= 2.01, y radius= 2.01]   ;
		%Straight Lines [id:da027175404992263452] 
		\draw    (380.6,158.33) -- (396.47,137.33) ;
		\draw [shift={(396.47,137.33)}, rotate = 307.07] [color={rgb, 255:red, 0; green, 0; blue, 0 }  ][fill={rgb, 255:red, 0; green, 0; blue, 0 }  ][line width=0.75]      (0, 0) circle [x radius= 1.34, y radius= 1.34]   ;
		\draw [shift={(380.6,158.33)}, rotate = 307.07] [color={rgb, 255:red, 0; green, 0; blue, 0 }  ][fill={rgb, 255:red, 0; green, 0; blue, 0 }  ][line width=0.75]      (0, 0) circle [x radius= 1.34, y radius= 1.34]   ;
		%Straight Lines [id:da9208814458449974] 
		\draw    (392.33,96.33) -- (412.33,116.33) ;
		\draw [shift={(412.33,116.33)}, rotate = 45] [color={rgb, 255:red, 0; green, 0; blue, 0 }  ][fill={rgb, 255:red, 0; green, 0; blue, 0 }  ][line width=0.75]      (0, 0) circle [x radius= 2.01, y radius= 2.01]   ;
		\draw [shift={(392.33,96.33)}, rotate = 45] [color={rgb, 255:red, 0; green, 0; blue, 0 }  ][fill={rgb, 255:red, 0; green, 0; blue, 0 }  ][line width=0.75]      (0, 0) circle [x radius= 2.01, y radius= 2.01]   ;
		%Straight Lines [id:da3057682193287987] 
		\draw    (473.33,116.33) -- (493.33,136.33) ;
		\draw [shift={(493.33,136.33)}, rotate = 45] [color={rgb, 255:red, 0; green, 0; blue, 0 }  ][fill={rgb, 255:red, 0; green, 0; blue, 0 }  ][line width=0.75]      (0, 0) circle [x radius= 2.01, y radius= 2.01]   ;
		\draw [shift={(473.33,116.33)}, rotate = 45] [color={rgb, 255:red, 0; green, 0; blue, 0 }  ][fill={rgb, 255:red, 0; green, 0; blue, 0 }  ][line width=0.75]      (0, 0) circle [x radius= 2.01, y radius= 2.01]   ;
		%Straight Lines [id:da37089103564734827] 
		\draw    (396.47,137.33) -- (420,158.4) ;
		\draw [shift={(420,158.4)}, rotate = 41.83] [color={rgb, 255:red, 0; green, 0; blue, 0 }  ][fill={rgb, 255:red, 0; green, 0; blue, 0 }  ][line width=0.75]      (0, 0) circle [x radius= 1.34, y radius= 1.34]   ;
		\draw [shift={(396.47,137.33)}, rotate = 41.83] [color={rgb, 255:red, 0; green, 0; blue, 0 }  ][fill={rgb, 255:red, 0; green, 0; blue, 0 }  ][line width=0.75]      (0, 0) circle [x radius= 1.34, y radius= 1.34]   ;
		%Straight Lines [id:da26981729267376564] 
		\draw    (380.6,158.33) -- (420,158.4) ;
		\draw [shift={(420,158.4)}, rotate = 0.1] [color={rgb, 255:red, 0; green, 0; blue, 0 }  ][fill={rgb, 255:red, 0; green, 0; blue, 0 }  ][line width=0.75]      (0, 0) circle [x radius= 2.01, y radius= 2.01]   ;
		\draw [shift={(380.6,158.33)}, rotate = 0.1] [color={rgb, 255:red, 0; green, 0; blue, 0 }  ][fill={rgb, 255:red, 0; green, 0; blue, 0 }  ][line width=0.75]      (0, 0) circle [x radius= 2.01, y radius= 2.01]   ;
		%Straight Lines [id:da4846315764509541] 
		\draw    (420,158.4) -- (420,181.2) ;
		\draw [shift={(420,181.2)}, rotate = 90] [color={rgb, 255:red, 0; green, 0; blue, 0 }  ][fill={rgb, 255:red, 0; green, 0; blue, 0 }  ][line width=0.75]      (0, 0) circle [x radius= 2.01, y radius= 2.01]   ;
		\draw [shift={(420,158.4)}, rotate = 90] [color={rgb, 255:red, 0; green, 0; blue, 0 }  ][fill={rgb, 255:red, 0; green, 0; blue, 0 }  ][line width=0.75]      (0, 0) circle [x radius= 2.01, y radius= 2.01]   ;
		%Straight Lines [id:da5267789399507867] 
		\draw    (495.46,102.32) -- (473.23,116.4) ;
		\draw [shift={(473.23,116.4)}, rotate = 147.65] [color={rgb, 255:red, 0; green, 0; blue, 0 }  ][fill={rgb, 255:red, 0; green, 0; blue, 0 }  ][line width=0.75]      (0, 0) circle [x radius= 2.01, y radius= 2.01]   ;
		\draw [shift={(495.46,102.32)}, rotate = 147.65] [color={rgb, 255:red, 0; green, 0; blue, 0 }  ][fill={rgb, 255:red, 0; green, 0; blue, 0 }  ][line width=0.75]      (0, 0) circle [x radius= 2.01, y radius= 2.01]   ;
		%Straight Lines [id:da38700782207128914] 
		\draw    (517.7,88.23) -- (495.46,102.32) ;
		\draw [shift={(495.46,102.32)}, rotate = 147.65] [color={rgb, 255:red, 0; green, 0; blue, 0 }  ][fill={rgb, 255:red, 0; green, 0; blue, 0 }  ][line width=0.75]      (0, 0) circle [x radius= 2.01, y radius= 2.01]   ;
		\draw [shift={(517.7,88.23)}, rotate = 147.65] [color={rgb, 255:red, 0; green, 0; blue, 0 }  ][fill={rgb, 255:red, 0; green, 0; blue, 0 }  ][line width=0.75]      (0, 0) circle [x radius= 2.01, y radius= 2.01]   ;
		%Straight Lines [id:da9720044271620558] 
		\draw    (495.46,102.32) -- (480.84,74.32) ;
		\draw [shift={(480.84,74.32)}, rotate = 242.41] [color={rgb, 255:red, 0; green, 0; blue, 0 }  ][fill={rgb, 255:red, 0; green, 0; blue, 0 }  ][line width=0.75]      (0, 0) circle [x radius= 1.34, y radius= 1.34]   ;
		\draw [shift={(495.46,102.32)}, rotate = 242.41] [color={rgb, 255:red, 0; green, 0; blue, 0 }  ][fill={rgb, 255:red, 0; green, 0; blue, 0 }  ][line width=0.75]      (0, 0) circle [x radius= 1.34, y radius= 1.34]   ;
		%Straight Lines [id:da4016755805710217] 
		\draw    (517.7,88.23) -- (480.84,74.32) ;
		\draw [shift={(480.84,74.32)}, rotate = 200.68] [color={rgb, 255:red, 0; green, 0; blue, 0 }  ][fill={rgb, 255:red, 0; green, 0; blue, 0 }  ][line width=0.75]      (0, 0) circle [x radius= 2.01, y radius= 2.01]   ;
		\draw [shift={(517.7,88.23)}, rotate = 200.68] [color={rgb, 255:red, 0; green, 0; blue, 0 }  ][fill={rgb, 255:red, 0; green, 0; blue, 0 }  ][line width=0.75]      (0, 0) circle [x radius= 2.01, y radius= 2.01]   ;
		%Straight Lines [id:da8794625397518283] 
		\draw    (480.84,74.32) -- (488.85,52.98) ;
		\draw [shift={(488.85,52.98)}, rotate = 290.58] [color={rgb, 255:red, 0; green, 0; blue, 0 }  ][fill={rgb, 255:red, 0; green, 0; blue, 0 }  ][line width=0.75]      (0, 0) circle [x radius= 2.01, y radius= 2.01]   ;
		\draw [shift={(480.84,74.32)}, rotate = 290.58] [color={rgb, 255:red, 0; green, 0; blue, 0 }  ][fill={rgb, 255:red, 0; green, 0; blue, 0 }  ][line width=0.75]      (0, 0) circle [x radius= 2.01, y radius= 2.01]   ;

		\draw (414.33,119.73) node [anchor=north west][inner sep=0.75pt]  [font=\small]  {$1a$};
		% Text Node
		\draw (378,126.73) node [anchor=north west][inner sep=0.75pt]  [font=\small]  {$2c$};
		% Text Node
		\draw (422,161.8) node [anchor=north west][inner sep=0.75pt]  [font=\small]  {$2d$};
		% Text Node
		\draw (422,184.6) node [anchor=north west][inner sep=0.75pt]  [font=\small]  {$1f$};
		% Text Node
		\draw (382.6,161.73) node [anchor=north west][inner sep=0.75pt]  [font=\small]  {$1e$};
		% Text Node
		\draw (458.83,119.73) node [anchor=north west][inner sep=0.75pt]  [font=\small]  {$2a$};
		% Text Node
		\draw (495.33,139.73) node [anchor=north west][inner sep=0.75pt]  [font=\small]  {$2b$};
		% Text Node
		\draw (497.46,105.72) node [anchor=north west][inner sep=0.75pt]  [font=\small]  {$1c$};
		% Text Node
		\draw (464.96,77.22) node [anchor=north west][inner sep=0.75pt]  [font=\small]  {$1d$};
		% Text Node
		\draw (519.96,76.72) node [anchor=north west][inner sep=0.75pt]  [font=\small]  {$2e$};
		% Text Node
		\draw (490.85,56.38) node [anchor=north west][inner sep=0.75pt]  [font=\small]  {$2f$};

	\end{tikzpicture}
\end{center}

 \caption{Figure 1. Two Distinguishably Equivalent Graphs.}

 \vspace{-1mm}
		    \label{Fig_1}
		\end{figure}

\noindent The automorphism representations for the graphs in Figure 1 are;
		\vspace{3mm}

\noindent $Aut(G_1, V(G_1)_{l_1})= Aut(G_2, V(G_2)_{l_2})= \{e,(1a,2a)(1b,2b)(1c,2c)(1d,2d)(1e,2e)(1f,2f)\}$

		\vspace{3mm}
\noindent Notice that if we were to switch the labels of $1a$ and $2c$ in $G_2$ leaving the remaining labels the same to form the labeling $l_3$, then $Aut(G_1, V(G_1)_{l_1}) \not= Aut(G_2, V(G_2)_{l_3})$. Thus, distinguishably equivalent graphs need not be isomorphic, asymmetric, or complements of one another.

%There are, in fact, examples for every possible finite automorphism group. Let $\Gamma$ be a finite group, and by Frucht’s theorem let \(G_{\Gamma}\) be a connected graph such that
%$$ \operatorname{Aut}(G_{\Gamma})\cong\Gamma $$

%and $|V(G_{\Gamma})|>8$. For $n\geq 9$, there exist pairs of nonisomorphic asymmetric trees on $n$ vertices. Let
%%$$ n=|V(G_{\Gamma})|+1 $$

%and choose such a pair $T_1,T_2$. Every tree is Cartesian-prime, and since
%$$ |V(T_i)|>|V(G_{\Gamma})|, $$

%neither $T_i$ is a Cartesian factor of $G_{\Gamma}$. Thus, $G_{\Gamma}$ and $T_i$ are relatively prime. Sabidussi’s Cartesian-product theorem therefore gives
%$$ \operatorname{Aut}(G_{\Gamma}\square T_i) \cong \operatorname{Aut}(G_{\Gamma})\times\operatorname{Aut}(T_i) \cong\Gamma $$

%for $i\in\{1,2\}$. Moreover, every automorphism of $G_{\Gamma}\square T_i$ has the form
%$$ (v,t)\longmapsto(\varphi(v),t), \qquad \varphi\in\operatorname{Aut}(G_{\Gamma}), $$

%%because $T_i$ is asymmetric. Since $T_1$ and $T_2$ have the same number of vertices, these two actions have the same permutation representation. Finally, the uniqueness of Cartesian-prime factorization and $T_1\not\cong T_2$ imply that
%$$ G_{\Gamma}\square T_1\not\cong G_{\Gamma}\square T_2. $$

%Hence, $G_{\Gamma}\square T_1$ and $G_{\Gamma}\square T_2$ are nonisomorphic graphs with the same automorphism representation.\\

The following proposition shows that while two graphs  having the automorphism representation implies that they have the same automorphism group, the converse is not true. That is, graphs can have the same automorphism group but unequivalent automorphism representations.

		\begin{proposition}\label{P1} The following holds: 

        \begin{enumerate}
            \item[(i)] $Aut(G_1, V(G_1)) = Aut(G_2, V(G_2)) \Rightarrow Aut(G_1) \cong Aut (G_2)$,
             \item[(ii)] $Aut(G_1, V(G_1)) = Aut(G_2, V(G_2)) \Rightarrow |V(G_1)| = |V(G_2)|$,
            \item[(iii)] (a) $Aut(G_1) \cong Aut (G_2) \not\Rightarrow Aut(G_1, V(G_1)) = Aut(G_2, V(G_2))$, 
            \item[ ] (b) $|V(G_1)| = |V(G_2)| \not \Rightarrow Aut(G_1, V(G_1)) = Aut(G_2, V(G_2))$.
        \end{enumerate}
		\end{proposition}
		
			\begin{proof} (i) Suppose that $Aut(G_1, V(G_1)) = Aut(G_2, V(G_2)).$ Then, there exist some $ l_1: V(G_1) \rightarrow S$ and $l_2: V(G_2) \rightarrow S$ such that $ Aut(G_1, V(G_1)_{l_1}) = Aut(G_2, V(G_2)_{l_2})$. Now $\pi: Aut(G_1) \rightarrow Aut (G_2)$ defined by $\pi (\alpha) = \beta$	iff $l_1(\alpha) = l_2(\beta)$, is clearly an isomorphism from $Aut(G_1)$ to $Aut(G_2)$. \\
            
			\noindent (ii) Now $|V(G_1)| = |l_1(V(G_1))| = |l_2(V(G_2))|= |V(G_2)|$. The first and last equalities hold because $l_1$ and $l_2$ are injective labelings and the middle equality holds by Definition \ref{D2}.
			
			\noindent (iii)(a) and (b) See Example \ref{Ex_1}.
\end{proof}

		\begin{definition}
			Let $c: V(G) \rightarrow C$ be a coloring of $G$. Using the convention following Definition \ref{D2.1}, we also denote the
induced coloring $c_l:l(V(G))\rightarrow C$ by $c$. An  {automorphism }$\varphi \in Aut(G, V(G)_l)$ is \textbf{non-monochromatically colored by \textit{c}}  or \textbf{broken by \textit{c}} if at least one of the cycles in its disjoint cycle representation, is non-monochromatically colored. 
			In particular, let  $$\varphi = (n_{11}, n_{12}\dots n_{1k_1})(n_{21}, n_{22} \dots n_{2k_2}) \dots (n_{m1}, n_{m2} \dots n_{mk_m}) \in Aut(G, V(G)_l),$$ where $n_{ij}$ labels the nodes of $G$ for some labeling $l: V(G) \rightarrow S$ of $G$. Then, $\varphi$ is broken by $c$   if there is a cycle, $  (n_{i1}, n_{i2}\dots n_{ik_i})$, such that $c(n_{ij}) \not = c(n_{i j+1})$, for some $j$  or $c(n_{ik_i}) \not = c(n_{i 1})$ and $k_i \geq 2$.   
		\end{definition}
		
		%	\begin{definition}
			
	%		An\textbf{ automorphism }$\varphi \in Aut(G)$ is said to be broken by a coloring $c: V(G) \rightarrow C$, if $\varphi$'s representation in  $Aut(G, V(G)_{l})$ is broken by $c$ for any labelling of the nodes $l$.
	%		An\textbf{ automorphism representation of a graph}, $Aut(G, V(G))$,  is said to be \textbf{non-monochromatically colored} by a coloring $c: V(G) \rightarrow C$ iff for any labelling $l: V(G) \rightarrow S$, every non-identity $\varphi \in Aut(G, V(G)_l)$ is broken or  non-monochromatically colored.
	%	\end{definition}

		\begin{proposition}\label{P3}
	
	Let $l_1: V(G) \rightarrow S_1$ and $l_2: V(G) \rightarrow S_2$ be two labelings of $G$. Let $c: V(G) \rightarrow C$ be a coloring of $G$, and let $\varphi \in Aut(G)$. Then $l_1(\varphi) \in Aut(G, V(G)_{l_1})$ is broken by $c$ if and only if $l_2(\varphi) \in Aut(G, V(G)_{l_2})$ is broken by $c$.
\end{proposition}

\begin{proof}
	This follows immediately because relabeling a graph preserves the cycle structure of its automorphisms. A coloring breaks an automorphism if and only if at least two nodes in one of its cycles have different colors. This condition is independent of how the nodes are labeled.
\end{proof}		
		\begin{remark} Proposition \ref{P3} shows that, for a single graph, whether a coloring breaks an automorphism is independent of the labeling. \textit{However, to establish that two graphs are distinguishably equivalent, one must find labelings for which their automorphism representations are equal (see Example \ref{E1}).}
	\end{remark}

		\begin{lemma}\label{L1}
			A graph $G$ is distinguished by a coloring 	$c: V(G) \rightarrow C$, iff for any labeling, $l:V(G) \rightarrow S$,   $Aut(G, V(G)_l) \setminus \{e \}$  is non-monochromatically colored by $c$.
		\end{lemma}
		
			\begin{proof}
			Suppose that $G$ is distinguished by a coloring $c: V(G) \rightarrow C$. Let $\alpha$ be such that $\alpha \not = e$ and $\alpha = (n_{11}, n_{12}\dots n_{1k_1})(n_{21}, n_{22} \dots n_{2k_2}) \dots (n_{m1}, n_{m2} \dots n_{mk_m}) \in Aut(G, V(G)_l)$, where $n_{ij} = l(v_i)$ for some labeling $l: V(G) \rightarrow S$ of $G$. Since $c$ is a distinguishing coloring, $\alpha$ does not preserve the color of the nodes. Therefore,  there must exist a cycle which is non-monochromatically colored. So, we get that $Aut(G, V(G)_l)$ is non-monochromatically colored.

			Conversely, suppose that every non-identity permutation is broken by any labeling. Therefore, for any non-identity $$\alpha = (n_{11}, n_{12}\dots n_{1k_1})(n_{21}, n_{22} \dots n_{2k_2}) \dots (n_{m1}, n_{m2} \dots n_{mk_m}) \in  Aut(G, V(G)_l),$$ where $n_{ij} = l(v_i)$ for some labeling $l: V(G) \rightarrow S$ of $G$, at least one cycle of size $\geq 2$ must be non-monochromatically colored.  Consequently, $\alpha$ does not respect the coloring, so $c$ is distinguishing.
		\end{proof}	

        \noindent In the following example, we will see a case of two graphs with isomorphic automorphism groups but different automorphism representations.
			
	\begin{example}\label{Ex_1}		
		Let $G_1$ and  $G_2$ be the graphs with labelings  from Figure \ref{Fig_2}. Then, $Aut(G_1) \cong Aut (G_2) \cong S_3$,  but  $D(G_1) = 3$ and $D(G_2)= 2$, see the coloring in Figure 2 for an illustration.  However, as $|V(G_1)| \not =|V(G_2)|$, using Proposition \ref{P1} (ii), we get $Aut(G_1, V(G_1)) \not = Aut(G_2, V(G_2))$.

\begin{center}

    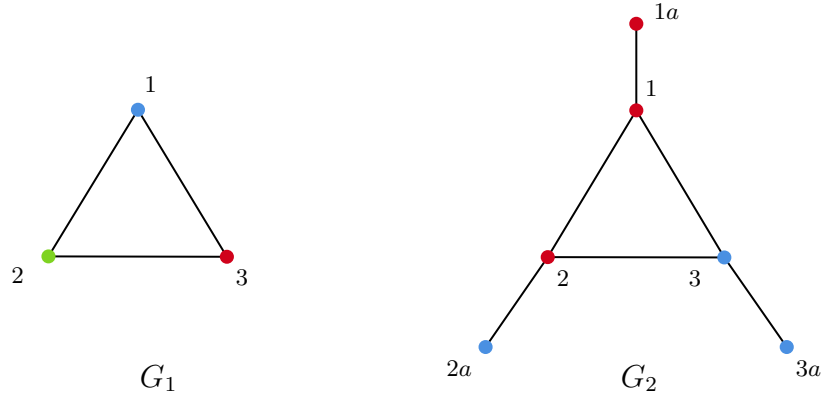
\begin{figure}[H]
        \centering
       
\tikzset{every picture/.style={line width=0.75pt}} %set default line width to 0.75pt        
\hspace{-2cm}
\begin{tikzpicture}[x=0.75pt,y=0.75pt,yscale=-1.5,xscale=1.5]
	
	%Shape: Free Drawing [id:dp11433321222189918] 
	\draw  [color={rgb, 255:red, 255; green, 255; blue, 255 }  ,draw opacity=1 ][line width=3] [line join = round][line cap = round] (115.5,187.25) .. controls (115.5,187.25) and (115.5,187.25) .. (115.5,187.25) ;
	%Shape: Free Drawing [id:dp9517428634778557] 
	\draw  [color={rgb, 255:red, 255; green, 255; blue, 255 }  ,draw opacity=1 ][line width=3] [line join = round][line cap = round] (121.5,266.75) .. controls (121.5,266.75) and (121.5,266.75) .. (121.5,266.75) ;
	%Shape: Triangle [id:dp49391816995134663] 
	\draw   (250.32,70.43) -- (279.91,119.73) -- (220.54,119.61) -- cycle ;
	%Straight Lines [id:da8605492387201223] 
	\draw [color={rgb, 255:red, 208; green, 2; blue, 27 }  ,draw opacity=1 ]   (279.91,119.73) ;
	\draw [shift={(279.91,119.73)}, rotate = 0] [color={rgb, 255:red, 208; green, 2; blue, 27 }  ,draw opacity=1 ][fill={rgb, 255:red, 208; green, 2; blue, 27 }  ,fill opacity=1 ][line width=0.75]      (0, 0) circle [x radius= 2.01, y radius= 2.01]   ;
	%Straight Lines [id:da1228412244385999] 
	\draw [color={rgb, 255:red, 126; green, 211; blue, 33 }  ,draw opacity=1 ]   (220.54,119.61) ;
	\draw [shift={(220.54,119.61)}, rotate = 0] [color={rgb, 255:red, 126; green, 211; blue, 33 }  ,draw opacity=1 ][fill={rgb, 255:red, 126; green, 211; blue, 33 }  ,fill opacity=1 ][line width=0.75]      (0, 0) circle [x radius= 2.01, y radius= 2.01]   ;
	%Shape: Triangle [id:dp43968507319402583] 
	\draw   (416.16,70.64) -- (445.45,119.96) -- (386.69,119.85) -- cycle ;
	%Straight Lines [id:da6382979825410895] 
	\draw    (416.16,41.6) -- (416.16,70.3) ;
	\draw [shift={(416.16,70.64)}, rotate = 90] [color={rgb, 255:red, 0; green, 0; blue, 0 }  ][line width=0.75]      (0, 0) circle [x radius= 1.34, y radius= 1.34]   ;
	\draw [shift={(416.16,41.6)}, rotate = 90] [color={rgb, 255:red, 0; green, 0; blue, 0 }  ][fill={rgb, 255:red, 0; green, 0; blue, 0 }  ][line width=0.75]      (0, 0) circle [x radius= 1.34, y radius= 1.34]   ;
	%Straight Lines [id:da677160624050352] 
	\draw    (445.45,119.96) -- (466.22,150) ;
	%Straight Lines [id:da47493837914183734] 
	\draw [fill={rgb, 255:red, 208; green, 2; blue, 27 }  ,fill opacity=1 ]   (386.69,119.85) -- (366,150) ;
	%Straight Lines [id:da5073985780281287] 
	\draw [color={rgb, 255:red, 74; green, 144; blue, 226 }  ,draw opacity=1 ]   (250.32,70.43) ;
	\draw [shift={(250.32,70.43)}, rotate = 0] [color={rgb, 255:red, 74; green, 144; blue, 226 }  ,draw opacity=1 ][fill={rgb, 255:red, 74; green, 144; blue, 226 }  ,fill opacity=1 ][line width=0.75]      (0, 0) circle [x radius= 2.01, y radius= 2.01]   ;
	%Straight Lines [id:da8272903703644627] 
	\draw [color={rgb, 255:red, 74; green, 144; blue, 226 }  ,draw opacity=1 ]   (366,150) ;
	\draw [shift={(366,150)}, rotate = 0] [color={rgb, 255:red, 74; green, 144; blue, 226 }  ,draw opacity=1 ][fill={rgb, 255:red, 74; green, 144; blue, 226 }  ,fill opacity=1 ][line width=0.75]      (0, 0) circle [x radius= 2.01, y radius= 2.01]   ;
	%Straight Lines [id:da7110999424301279] 
	\draw [color={rgb, 255:red, 208; green, 2; blue, 27 }  ,draw opacity=1 ]   (386.69,119.85) ;
	\draw [shift={(386.69,119.85)}, rotate = 0] [color={rgb, 255:red, 208; green, 2; blue, 27 }  ,draw opacity=1 ][fill={rgb, 255:red, 208; green, 2; blue, 27 }  ,fill opacity=1 ][line width=0.75]      (0, 0) circle [x radius= 2.01, y radius= 2.01]   ;
	%Straight Lines [id:da4060394456519225] 
	\draw [color={rgb, 255:red, 74; green, 144; blue, 226 }  ,draw opacity=1 ]   (466.22,150) ;
	\draw [shift={(466.22,150)}, rotate = 0] [color={rgb, 255:red, 74; green, 144; blue, 226 }  ,draw opacity=1 ][fill={rgb, 255:red, 74; green, 144; blue, 226 }  ,fill opacity=1 ][line width=0.75]      (0, 0) circle [x radius= 2.01, y radius= 2.01]   ;
	%Straight Lines [id:da17026368179351237] 
	\draw [color={rgb, 255:red, 74; green, 144; blue, 226 }  ,draw opacity=1 ]   (445.45,119.96) ;
	\draw [shift={(445.45,119.96)}, rotate = 0] [color={rgb, 255:red, 74; green, 144; blue, 226 }  ,draw opacity=1 ][fill={rgb, 255:red, 74; green, 144; blue, 226 }  ,fill opacity=1 ][line width=0.75]      (0, 0) circle [x radius= 2.01, y radius= 2.01]   ;
	%Straight Lines [id:da3064025313064114] 
	\draw [color={rgb, 255:red, 208; green, 2; blue, 27 }  ,draw opacity=1 ]   (416.16,70.64) ;
	\draw [shift={(416.16,70.64)}, rotate = 0] [color={rgb, 255:red, 208; green, 2; blue, 27 }  ,draw opacity=1 ][fill={rgb, 255:red, 208; green, 2; blue, 27 }  ,fill opacity=1 ][line width=0.75]      (0, 0) circle [x radius= 2.01, y radius= 2.01]   ;
	%Straight Lines [id:da5319373510125436] 
	\draw [color={rgb, 255:red, 208; green, 2; blue, 27 }  ,draw opacity=1 ]   (416.16,41.6) ;
	\draw [shift={(416.16,41.6)}, rotate = 0] [color={rgb, 255:red, 208; green, 2; blue, 27 }  ,draw opacity=1 ][fill={rgb, 255:red, 208; green, 2; blue, 27 }  ,fill opacity=1 ][line width=0.75]      (0, 0) circle [x radius= 2.01, y radius= 2.01]   ;
	
	% Text Node
	\draw (352,153.4) node [anchor=north west][inner sep=0.75pt]  [font=\scriptsize]  {$2a$};
	% Text Node
	\draw (207.21,122.35) node [anchor=north west][inner sep=0.75pt]  [font=\scriptsize]  {$2$};
	% Text Node
	\draw (468.22,153.4) node [anchor=north west][inner sep=0.75pt]  [font=\scriptsize]  {$3a$};
	% Text Node
	\draw (251.5,58.4) node [anchor=north west][inner sep=0.75pt]  [font=\scriptsize]  {$1$};
	% Text Node
	\draw (281.91,123.13) node [anchor=north west][inner sep=0.75pt]  [font=\scriptsize]  {$3$};
	% Text Node
	\draw (418.16,59.52) node [anchor=north west][inner sep=0.75pt]  [font=\scriptsize]  {$1$};
	
	\draw (410,155) node [anchor=north west][inner sep=0.75pt]  []  {$G_2$};
	
		\draw (250,155) node [anchor=north west][inner sep=0.75pt]  []  {$G_1$};
	% Text Node
	\draw (421,33.73) node [anchor=north west][inner sep=0.75pt]  [font=\scriptsize]  {$1a$};
	% Text Node
	\draw (388.69,123.25) node [anchor=north west][inner sep=0.75pt]  [font=\scriptsize]  {$2$};
	% Text Node
	\draw (432.67,123.4) node [anchor=north west][inner sep=0.75pt]  [font=\scriptsize]  {$3$};
\end{tikzpicture}
\vspace{-4cm}

\caption{Figure 2. Two Graphs with Isomorphic Automorphism Groups but Unequal \phantom{aaaaaaaaaaaaaaaa} Automorphism Representations. }
        \label{Fig_2}
        
\begin{alignat*}{2}
&Aut(G_1, V(G_1)l_1) =\ && \big\{e, (\textcolor{blue}{1},\textcolor{green}{2}), (\textcolor{blue}{1},\textcolor{red}{3}), (\textcolor{green}{2},\textcolor{red}{3}), (\textcolor{blue}{1},\textcolor{red}{3},\textcolor{green}{2}), (\textcolor{blue}{1},\textcolor{green}{2},\textcolor{red}{3})\big\}\\
 &Aut(G_2, V(G_2)l_2) =\ && \big\{e, (\textcolor{red}{1},\textcolor{red}{2})(\textcolor{red}{1a}, \textcolor{blue}{2a}), (\textcolor{red}{1},\textcolor{blue}{3})(\textcolor{red}{1a},\textcolor{blue}{3a}), (\textcolor{red}{2},\textcolor{blue}{3})(\textcolor{blue}{2a},\textcolor{blue}{3a}),\\ &  &&(\textcolor{red}{1},\textcolor{red}{2},\textcolor{blue}{3})(\textcolor{red}{1a},\textcolor{blue}{2a},\textcolor{blue}{3a}), (\textcolor{red}{1},\textcolor{blue}{3},\textcolor{red}{2})(\textcolor{red}{1a},\textcolor{blue}{3a},\textcolor{blue}{2a})\big\}.
\end{alignat*}

    \end{figure}
\end{center}
\vspace{-8mm}
\noindent Notice that  a distinguishing coloring induces a non-monochromatic coloring of all of the non-identity permutations in a labeled graph's automorphism representation. Conversely, any non-monochromatic coloring of a labeled graph's automorphism representation induces a distinguishing coloring. 
	\end{example}

   \noindent  While the previous example shows that two graphs having \textit{isomorphic automorphism groups is not enough}  to conclude that their distinguishing numbers are equal, the following result shows that two graphs having \textit{equal automorphism representations is enough} to force that their distinguishing numbers are equal.

		\begin{theorem} \label{T1} Let $G_1$ and $G_2$ be two graphs. Then, 
			$$Aut(G_1, V(G_1)) = Aut(G_2, V(G_2)) \implies D(G_1) = D(G_2).$$
		\end{theorem}
		
			\begin{proof}
			Suppose that $Aut(G_1, V(G_1)) = Aut(G_2, V(G_2))$ and that $D(G_1) = k_1$. Therefore, there exist labelings $l_1:V(G_1) \rightarrow S $ and $l_2: V(G_2) \rightarrow S$ such that $Aut(G_1, V(G_1)_{l_1}) = Aut(G_2, V(G_2)_{l_2})$. These labelings induce a bijection $f= l_1^{-1} \circ l_2: V(G_2) \rightarrow V(G_1). $ Also, there exists a coloring, $c_1: V(G_1) \rightarrow C$ using $k_1$ colors which distinguishes $G_1$. The corresponding coloring used to color $G_2$ is $c_2: V(G_2) \rightarrow C$ defined by $c_2 (v) = c_1(f(v))$.

			 By Lemma \ref{L1}, $c_1$ breaks all of the non-identity permutations in $Aut(G_1, V(G_1)_{l_1})$ and so $c_2$ breaks all of the non-identity permutations in $Aut(G_2, V(G_2)_{l_2})$ as $Aut(G_1, V(G_1)_{l_1}) = Aut(G_2, V(G_2)_{l_2})$. Using Lemma \ref{L1} again, we get that this coloring is a distinguishing coloring for $G_2$. So $D(G_2) \leq k_1 = D(G_1)$. 
		Interchanging the roles of $G_1$ and $G_2$, we get that $D(G_1) \leq D(G_2)$. Hence, $D(G_1) = D(G_2)$.
			\end{proof}

			\noindent \textbf{Remark}: The reasoning in the proof of Theorem \ref{T1} also
shows that if $G_1$ and $G_2$ are 2-distinguishable graphs, then $$Aut(G_1, V(G_1)) = Aut(G_2, V(G_2)) \implies \rho(G_1) = \rho(G_2).$$

	%	\begin{remark}
		    
	%Using Proposition \ref{P3} and Theorem \ref{T1}, we see that in order to distinguish a graph, it is enough to give it an arbitrary labelling $l$ and find a non-monochromatic coloring of $Aut(G, V(G)_{l})$.
	%		\end{remark} 
		
		%	\begin{remark}  While Theorem \ref{T1} states that the distinguishing numbers of $G_1$ and $G_2$ are equal if $Aut(G_1, V(G_1)) = Aut(G_2, V(G_2))$, note that is not true if simply $Aut(G_1) \cong Aut(G_2)$, see again Figure \ref{Fig_2}. 
            	%\end{remark}

		\noindent  Now, if $H$ is an induced subgraph of $G$, then a distinguishing coloring of $G$ when restricted to $H$ need not be a distinguishing coloring of $H$. However, when the nodes of $H$  have the same neighbours in $G \setminus H$, then this holds:

		\begin{lemma}\label{L2}
			Let $G$ be a graph and $H$ an induced subgraph of $G$ such that for any $h_1, \ h_2 \in H$, $N(h_1) \setminus V(H) = N(h_2) \setminus V(H)$.  If a coloring $k:V(G) \rightarrow C$ is a distinguishing coloring of $G$, then  $k\restriction V(H)$ is a distinguishing coloring of $H$, where $N(h_i) = \{z: \{h_i, z\} \in E(G)\}$.
		\end{lemma}
		
		\begin{proof}
			Let $\varphi$ be a non-identity automorphism of $H$. Let  $\varphi_G: V(G) \rightarrow V(G)$ be defined by $\varphi_G(v)= \varphi(v) \ \text{if} \ v \in V(H)$ and 	 $\varphi_G(v)= v \ \text{if} \ v \in V(G) \setminus V(H)$.
			
			\begin{claim}
			    
			 $\varphi_G$ is an automorphism of $G$.\end{claim}
			
			\begin{proof}[Proof of the claim:]
			     To see this, note that there are 3 types of edges and non-edges in $G$: (i) Edges and non-edges in $H$, (ii) Edges and non-edges in $G \setminus H$ and (iii) Edges and non-edges between $H$ and $G$. The first two types of edge and non-edges  are clearly preserved by $\varphi_G$. The third type is preserved since $N(h_1) \setminus V(H) = N(h_2) \setminus V(H)$ for any $h_1, \ h_2 \in V(H)$,  implies that for  $v \in V(G \setminus H)$, $\{h_1, v\} \in E(G)$ iff $ \{h_2, v\} \in E(G)$. Overall, since $\varphi_G$ preserves all edges and non-edges of $G$, it is an automorphism of $G$. \end{proof}

            \noindent Let $l: V(G) \rightarrow S$ be an injective map labeling the nodes of $G$ and let $l_H$ be this map restricted to the nodes of $H$. Let $l(V(G) \setminus V(H)) = \{v_1, \dots, v_k\}$ be the $l$-labeling of the nodes in $G \setminus H$. Since  $\varphi \in Aut(H, V(H)_l)$, thus $ \varphi_G =  (v_1)\dots(v_k)\varphi \in Aut(G, V(G)_l)$ by the previous claim. By Lemma \ref{L1}, as $k$ is a distinguishing coloring of $G$, it must non-monochromatically color $\varphi_G$. Observe that, as $\varphi_G$ has the same disjoint cycle representation as $\varphi$ apart from single cycles,  $k$ non-monochromatically colors $\varphi_G$ iff $k$ non-monochromatically color $\varphi$. Since  $k$ is a distinguishing coloring, it  non-monochromatically colors all elements in $Aut(G, V(G)_l) \setminus \{e\} $ by Lemma \ref{L1}. In particular,  $k$ breaks all the  colors  in $\{\varphi_G:  \varphi \in Aut(H, V(H)_{l_H})\setminus\{e\}  \}$. Hence,  $k$ non-monochromatically  colors all elements in $Aut(H, V(H)_{l_H}) \setminus \{e\}$ and we get that $k\restriction V(H)$ is a distinguishing coloring of $H$ by Lemma \ref{L1}.
		\end{proof}

		\section{Main Theorem}
\noindent Let us introduce the concepts required to prove our main theorem.   
\begin{definition}\label{D5}
			Let $G= (V(G), E(G))$ be a graph. A subset $S \subseteq V(G)$ is said to be a  \textbf{determining set} for $G$ if
			whenever $\varphi \in Aut(G)$ so that $\varphi(x) = x$ for all $x 
			\in S$, then $\varphi$ is the identity. The \textbf{determining number}, $Det(G)$, of $G$ is the minimum size of a determining set for $G$. 	Let $G$ be a 2-distinguishable
			graph. Call a color class in a 2-distinguishing coloring of $G$ a \textbf{distinguishing class} and the size of the smaller color class, \textbf{the cost of the coloring}. Then, the minimum size of a distinguishing class of $G$
			is called the \textbf{cost number} of $G$ and is denoted by $\rho(G)$.
		\end{definition}

		\noindent Our main result is to answer the question of Boutin mentioned in the introduction when $Det(G)=2$. In particular, we prove the following result:
	\begin{manualtheorem}{\ref{main_thm}}
		\textit{Let $G$ be a graph with $Det(G)= 2,  D(G) = 2$, then $2\leq\rho(G) \leq4$.}
	\end{manualtheorem}
        
		\noindent  We will make use of the following lemma from  \cite{Bou2013}:
  
			\begin{lemma}\label{L3} (Lemma 1 of \cite{Bou2013})
			A subset of vertices $S$ is a distinguishing class for $G$ iff $S$ is a determining set for $G$ with the property that every automorphism that fixes $S$ setwise, also fixes it pointwise. That is, if $\varphi \in Aut(G)$ and $\varphi(S) = S$, then $\varphi(s) = s$, for  $s\in S$. 
		\end{lemma}

\noindent\textbf{Notation.} Throughout the following, we  consider a fixed labeling of the nodes of $G$. As this context is fixed, we will simplify the notation and when a permutation belongs to $Aut(G, V(G)_l)$ we will state that it belongs to $Aut(G)$. If the images of a permutation are partially known, it will be written in a form denoting those images but with a variable for the remaining images. For example, if  for some $\varphi\in Aut(G)$ it is known that $\varphi(x) = y$ and $\varphi(y) = x$, but the other images are unknown, then we may write $\varphi = (xy) \alpha$, where $\alpha$ represents the remaining disjoint cycles of the permutation. If a permutation is written without a variable encoding its extension, e.g. $(xy)$, then the rest of the images are assumed to be single cycles. Also, in this paper, we follow the convention of multiplying the cycles from the \textit{rightmost cycle going towards the left} (treating them like functions). Lastly, we always consider permutations in disjoint cycle form.\\

\noindent Before proving Theorem \ref{main_thm} we will need several technical propositions and lemmas on graphs with determining number 2.

	\begin{proposition}\label{P4}
		Let $G$ be a graph. If $\{x, y\} \subseteq V(G)$ is a determining set of $G$, then there is no $\varphi\in Aut(G)\setminus \{e\} $ such that $ \varphi = (x)(y) \alpha$, for any $\alpha$, where $e$ is the identity.
		
	\end{proposition}

	\begin{proof}
		
		If such a $\varphi$ exists, then $\varphi(\{x, y\}) = \{x, y\}$, but $\varphi(x) = x$ and $\varphi(y) = y$. Since $\varphi \not = e$, this contradicts that $\{x, y\}$ is a determining set of $G$.	
	\end{proof}

	\begin{proposition}	\label{P5}
		Let $G$ be a graph and $\{x,y\}\subseteq V(G)$ be a determining set of $G$. Then
	\begin{alignat*}{4}
	(i)\ (xy) \alpha \in Aut(G) \text{ for some } \alpha\ \   &   \implies \alpha = (v_{11} v_{12})(v_{21} v_{22})\dots(v_{m1} v_{m2})(v_1) \dots(v_n), \\
&\text{ for some $m, n \geq 0$,   $v_{ij}=v_{kl}$ implies $i=k$,   $j=l$}.\\
(ii)\	(xv)(y) \alpha \in Aut(G)  \text{ for some } &\alpha  \implies \alpha = (v_{11} v_{12})(v_{21} v_{22})\dots(v_{m1} v_{m2})(v_1) \dots(v_n)   ,\\
&\text{ for some $m, n \geq 0$,   $v_{ij}=v_{kl}$ implies $i=k$, $j=l$}.\\
(iii)\	(yv)(x) \alpha \in Aut(G) \text{ for some } &\alpha  \implies \alpha = (v_{11} v_{12})(v_{21} v_{22})\dots(v_{m1} v_{m2})(v_1) \dots(v_n),\\
&\text{ for some $m, n \geq 0$,   $v_{ij}=v_{kl}$ implies $i=k$, $j=l$.}
	\end{alignat*}
	
\noindent That is, the remaining images of permutations containing $(xy), \ (xv)(y)$ or $(yv)(x)$ must be 2-cycles or fixed points.
	\end{proposition}
	
	\begin{proof}
		(i)	Suppose to get a contradiction that $(xy) \alpha \in Aut(G)$, where $\alpha$ is a product of disjoint cycles and $\alpha = (v_1v_2 \dots v_k) \beta$, where $k \geq 3$ and $\beta$ is the remaining disjoint cycles of $\alpha$. Then $((xy)\alpha)^2 = (xy)^2 (\alpha)^2 = (x)(y) ((v_1v_2 \dots v_k))^2 \beta^2 \not= e$, since $(v_1v_2 \dots v_k)^2 \not =e$ as $k \geq 3$. Then $((xy)\alpha)^2 $ is a non-identity permutation  fixing $x$ and $y$, which contradicts Proposition \ref{P4}.\\
		
		\noindent (ii)-(iii) Similar to part (i).
	\end{proof}
	
	\begin{proposition}\label{P6}
		Let $G$ be a graph. If $\{x, y\} \subseteq V(G)$ is a determining set of $G$ then,\\
		
		     $(xy)\alpha_1, \ (xy)\alpha_2 \in Aut(G)$ for some $\alpha_1,\alpha_2$ $\implies \alpha_1 = \alpha_2$.\\

		\noindent That is, we cannot have two different automorphisms extending $(xy)$.
		
	\end{proposition}	
	
	\begin{proof}
		By Proposition \ref{P5} (i), if $(xy)\alpha_1, (xy)\alpha_2 \in Aut(G)$, then $\alpha_1, \alpha_2$ are  products of disjoint 2-cycles and single cycles. Suppose to get a contradiction that $\alpha_1 \not= \alpha_2$.
        Let $\sigma_1 = (xy) \alpha_1$ and $\sigma_2 = (xy) \alpha_2$. Then, $\sigma_1 \sigma_2^{-1} = (x)(y) \gamma \not= e$. This contradicts Proposition \ref{P4}.
	\end{proof}

	\begin{proposition}\label{P7}
		Let $G$ be a graph and suppose that $\{x, y\} \subseteq V(G)$ is a determining set of $G$ and  $(xy)(d_1d_2)\alpha \in Aut(G)$, for some $\alpha$  (with $d_1 \not= d_2$).\\
		
		\noindent If 	$(xd_1)(y)\beta_1 \in Aut(G)$ for some $\beta_1$ $ \implies (xyd_1)\beta_2, (xd_1y)\beta_3 \not\in Aut(G)$ for any $\beta_2,\beta_3.$

	\end{proposition}		
	
	\begin{proof}
		Suppose to get a contradiction that $(xd_1)(y)\beta_1\in Aut(G)$ and \begin{align}
		(xyd_1)\beta_2\in Aut(G) \text{ or } \  (xd_1y)\beta_3 \in Aut(G),
		\end{align} for some $\beta_2,\beta_3$. First, we show that $(xyd_1)\beta_2\not\in Aut(G)$. Otherwise $(xd_1)(y)\beta_1 \cdot (xyd_1)\beta_2 = (xy) (d_1) \beta_1 \beta_2 \in Aut(G).$ Recall that, $(xy)(d_1d_2)\alpha \in Aut(G)$.  Proposition \ref{P6} implies that $(d_1) \beta_1 \beta_2 = (d_1d_2)\alpha$. However, since $\alpha, \ \beta_1, \ \beta_2$ are disjoint products of permutations not involving $d_1$,  this equality cannot hold as the image of $d_1$ is different on both sides. Thus $(xyd_1)\beta_2\not\in Aut(G)$.
		Next, consider if $(xd_1y)\beta_3 \in Aut (G)$. However, this also leads us to a contradiction, since  $((xd_1y)\beta_3)^2 = (xyd_1)(\beta_3)^2$ and taking $\beta_2 = (\beta_3)^2$, implies that $(xyd_1)\beta_2 \in Aut(G)$, for some $\beta_2$. Hence, $(xd_1y)\beta_3 \not\in Aut(G)$.
	\end{proof}

    \begin{proposition}\label{P8}
		Let $G$ be a graph and suppose that $\{x, y\} \subseteq V(G)$ is a determining set of $G$,  $(xy)(d_1d_2)\alpha \in Aut(G)$, for some $\alpha$  (with $d_1 \not= d_2$).\\
		
		\noindent If $(xd_1)(y)\beta_1 \in Aut(G)$, for some $\beta_1$  $\implies (yd_1)(x) \beta_2 \not\in Aut(G),$ for any $\beta_2$.
	\end{proposition}			
	
	\begin{proof}
		Suppose to get a contradiction that $(xy)(d_1d_2)\alpha,  (xd_1)(y)\beta_1, (yd_1)(x) \beta_2 \in Aut(G)$, for some $\alpha,\beta_1$ and $\beta_2$. Then, $ (xd_1)(y)\beta_1 \cdot (yd_1)(x) \beta_2 = (xd_1y) \beta_1\beta_2 \in Aut(G)$  contradicting to Proposition \ref{P7}.
	\end{proof}

	\begin{proposition}\label{P9}
			Let $G$ be a graph and suppose that $\{x, y\} \subseteq V(G)$ is a determining set of $G$,  $(xy)(d_1d_2)(d_3d_4)\alpha \in Aut(G)$, for some $\alpha$  and $d_i\neq d_j$ for $i\not= j\in \{1,2,3,4\}$.
		
		\begin{alignat*}{4}
			&	\text{If }   (xd_i)(y)\beta_1 \in Aut(G), \text{ for some } \beta_1&& \implies  &&&(i) \ (xd_id_j)(y)\beta_2, (xd_jd_i)(y)\beta_3 \not\in Aut(G), \\
			 & && &&&  (ii)\ (yd_id_j)(x)\beta_4, (yd_jd_i)(x)\beta_5 \not\in Aut(G),
		\end{alignat*}
		\noindent for any $\beta_l$, where $2\leq l\leq 5$.

	\end{proposition}	
	
	\begin{proof}
\noindent (i)		Suppose to get a contradiction that $(xd_i)(y)\beta_1\in Aut(G)$ and 
\begin{align}
(xd_id_j)(y)\beta_2\in Aut(G) \text{ or }  (xd_jd_i)(y)\beta_ 3\in Aut(G),
\end{align}
for some $\beta_2, \beta_3$ and $d_j\neq d_i\in V(G)$.  First, we show that $(xd_jd_i)(y)\beta_3\not\in Aut(G)$. Otherwise,
\begin{align}
(xd_jd_i)(y)\beta_3 \cdot (xd_i)(y)\beta_1 = (x) (y)(d_id_j\dots)\beta_3\beta_1\in Aut(G).
\end{align}
This contradicts that $\{x, y\}$ is a determining set by Proposition \ref{P4}. Next, suppose that $(xd_id_j)(y)\beta_2\in Aut(G)$. Then, $((xd_id_j)(y)\beta_2)^2 = (xd_jd_i)(y)\beta_2^2\in Aut(G)$. However, this contradicts $(xd_jd_i)(y)\beta_3 \not\in Aut(G)$, for any $\beta_3$.  \\
		
		\noindent (ii)
			Suppose to get a contradiction that $(xd_i)(y)\beta_1\in Aut(G)$ and 
			\begin{align}
				(yd_id_j)(x)\beta_4\in Aut(G) \text{ or }  (yd_jd_i)(x)\beta_ 5\in Aut(G),
			\end{align}
		for some $\beta_4, \beta_5$ and $d_j\neq d_i\in V(G)$. Now if $(yd_id_j)(x)\beta_4 \in Aut(G)$, then also $(yd_jd_i)(x)\beta_5 \in Aut(G)$, since
			\begin{align}
		((yd_id_j)(x)\beta_4)^2 = (yd_jd_i)(x)\beta_4^2\in Aut(G),
		\end{align}
		  where $\beta_5=\beta_4^2$. 
		    Consider, 
		    \begin{align}
		    (yd_jd_i)(x)\beta_5  \cdot (xd_i)(y)\beta_1\cdot (yd_id_j)(x)\beta_4  = (xy)(d_j) \beta_4 \beta_1 \beta_5 \in Aut(G).
		    \end{align}

			Equation (6) gives an extension of $(xy)$ that fixes $d_j$, whereas the automorphism $ (xy)(d_1d_2)(d_3d_4)\alpha  \in Aut (G)$ is an extension of $(xy)$ that moves $d_j$. These are distinct extensions of \((xy)\), contradicting Proposition \ref{P6}.
		  Next, if $(yd_jd_i)\beta_5 \in Aut(G)$, then $((yd_jd_i)(x)\beta_5)^2 = (yd_id_j)(x)\beta_5^2\in Aut(G)$. Letting, $\beta_4=\beta_5^2$, this gives a contradiction.
	\end{proof}	
	
	\begin{proposition}\label{P10}
		Let $G$ be a graph, $\{x, y\}\subseteq V(G)$ be a determining set of G and $(xy)(d_1d_2)\alpha \in Aut(G)$, for some $\alpha$. Then,\\
\vspace{-8mm}
        
		 $$(xd_1)(y) \beta_1 \in Aut(G), \  \text{for some} \ \beta_1 \iff (yd_2)(x) \beta_2 \in Aut(G), \ \text{for some} \ \beta_2.$$
	\end{proposition}
	
	\begin{proof}
		
		($\Rightarrow$): 	Suppose that $(xy)(d_1d_2)\alpha , (xd_1)(y)\beta_1\in Aut(G)$ for some $\alpha, \beta_1$. Then,
		\begin{align}
		(xy)(d_1d_2)\alpha \cdot (xd_1)(y) \beta_1  \cdot (xy)(d_1d_2)\alpha =(yd_2)(x) \alpha\beta_1\alpha \in Aut(G).
		\end{align}
		 	 Thus, let $\beta_2:= \alpha\beta_1\alpha$. The converse direction ($\Leftarrow$) is similar.
	\end{proof}	
	
	\begin{proposition}\label{P11}
		Let $G$ be a graph and let $\{x, y\} \subseteq V(G)$ be a determining set of $G$. Then,
        $$(xd_1)(y) \beta_1, (xd_2)(y)\beta_2 \in Aut(G) \implies (xd_1)(y)(d_2)\beta_3, (xd_2)(y)(d_1)\beta_4 \not\in Aut(G),$$

         \noindent for any $\beta_1,\  \beta_2,\beta_3,\beta_4$  and $d_1 \not= d_2$. 

	\end{proposition}
	
	\vspace{-3mm}
	\begin{proof}
		Assume, by way of contradiction that $(xd_1)(y) \beta_1, (xd_2)(y)\beta_2$ and	$(xd_1)(y)(d_2)\beta_3 \in Aut(G)$, for some $\beta_3$. Then we can write,
		\begin{align}
		(xd_2)(y)\beta_2 = (xd_2)(y)(d_1d_3\dots)\beta_5\in Aut(G),
		\end{align}
	 where $d_3 \not = d_2$ as $d_2$ is already the image of $x$ (however $d_3$ could be $d_1$). Consider,
		\begin{align}
			\begin{split}
				(xd_2)(y)(d_1d_3\dots)\beta_5 \cdot (xd_1)(y)(d_2)\beta_3 \cdot (xd_2)(y)(d_1d_3\dots)\beta_5 \\= (x)(y)(d_2d_3\dots) \in Aut(G).
			\end{split}
		\end{align}
		 Observe that, this automorphism is not equal to the identity automorphism as $d_2 \not = d_3$. This contradicts $\{x, y\}$ being a determining set by Proposition \ref{P4}. Similarly, if $(xd_2)(y)(d_1)\beta_4 \in Aut(G)$.
	\end{proof}

	\begin{lemma}\label{L4}
		
		Let $G$ be a graph with $D(G) = 2$. Assume that  $\{x, y\} \subseteq V(G)$ is a determining set of $G$  and  $(xy)(d_1d_2)(d_3d_4)\alpha \in Aut(G)$, for some $\alpha$, where $d_i\neq d_j$ for $i\neq j $ and $ i, j \in \{1,2,3,4\}$. 
		
		\begin{enumerate}
			\item[(i)] If $(xd_1)(y)\beta_1, \ (xd_3)(y)\beta_2  \in Aut(G)$, for some $\beta_1,\beta_2$, then the following holds:\\
			\begin{align*}
				(xd_1)(yd_3)\beta_3,  (xd_3)(yd_1)\beta_4& \not\in Aut(G), \text{ for any $\beta_3,\beta_4$} \implies \\
	 			\{x, y, d_1, d_3\} \text{ is a } &\text{distinguishing class of $G$.}
			\end{align*}
			\vspace{.1cm}
			 
			\item[(ii)]   If $(xd_1)(y)\beta_1, \ (yd_3)(x)\beta_2  \in Aut(G)$, for some $\beta_1,\beta_2$, then the following holds:
			
				\begin{align*}
			\begin{rcases*}
				(xd_1)(yd_3)\beta_3, &  $(xd_3)(yd_1)\beta_4$,  \\
				(xd_1)(y)(d_3)\beta_5, & $(yd_3)(x)(d_1)\beta_6$\\
			\end{rcases*}&\not \in Aut(G), \text{ for any $\beta_3,\beta_4,\beta_5$ and $\beta_6$ } \implies\\
		\{x,y,d_1,d_3\} &\text{ is a distinguishing class of $G$.}
			\end{align*}
		\end{enumerate}
	\end{lemma}
	
	\begin{proof}(i) Assume that  $(xd_1)(y)\beta_1, \ (xd_3)(y)\beta_2  \in Aut(G)$, for some $\beta_1,\beta_2$ and $(xd_1)(yd_3)\beta_3, \\  (xd_3)(yd_1)\beta_4 \not\in Aut(G)$, \text{for any} $\beta_3,\beta_4$. Let  $S= \{x, y, d_1, d_3\}$. This is a determining set since $\{x, y\}$ is a determining set. We will show that this forms a distinguishing class using Lemma \ref{L3}, by showing that if $\varphi(S) = S$, then $\varphi(s) = s$, for every $s \in S$. We will investigate all possibilities of permutations which leaves $S$ invariant but does not fix its elements, and show that such permutations cannot be in $Aut(G)$. These include, permutations which, \textsc{Case 1:} flips two elements in $S$ and leaves the other elements of $S$ fixed; \textsc{Case 2:}  includes two transpositions of the elements of $S$; \textsc{Case 3:} includes a 3-cycle and a single cycle of the elements of $S$; \textsc{Case 4:} includes a 4-cycle of the elements of $S$. We consider these cases next:\\
		
		\noindent \underline{\textsc{Case 1:}} In disjoint-cycle form, these permutations consist of;
		\begin{align}
			\begin{split}
			&(xy) (d_1)(d_3)\gamma_1, \qquad (d_1d_3)(x)(y) \gamma_2, \qquad (xd_1) (y)(d_3) \gamma_3,\\
			&(xd_3) (y)(d_1) \gamma_4, \qquad (yd_1) (x)(d_3) \gamma_5, \qquad (yd_3)(x)(d_1) \gamma_6.
			\end{split}
		\end{align}
		  Since  $(xy)(d_1d_2)(d_3d_4)\alpha \in Aut(G)$  by Proposition \ref{P6} we must have $(xy)(d_1)(d_3)\gamma_1 \not \in Aut(G)$. Also, $(d_1d_3)(x)(y) \gamma_2 \not\in Aut(G)$ by Proposition \ref{P4}. Additionally, $(xd_1) (y)(d_3) \gamma_3$ and $\ (xd_3) (y)(d_1) \gamma_4 \not\in Aut(G)$ by Proposition \ref{P11}. Lastly, $(yd_1)(x)(d_3) \gamma_5,(yd_3) (x)(d_1) \gamma_6 \not\in Aut(G)$ by Proposition \ref{P8}.\\
		
		\noindent  \underline{\textsc{Case 2:}} In disjoint-cycle form, the possibilities are;
		\begin{align}
			\begin{split}
				(xy)(d_1d_3)\gamma_1,\ (xd_1)(yd_3)\gamma_2,\ (xd_3)(yd_1)\gamma_3.
			\end{split}
		\end{align}
		The first type of permutation is not in $Aut (G)$ by Proposition \ref{P6} and the last two types of permutations are not in $Aut(G)$ by the assumed condition.\\
		
		\noindent \underline{\textsc{Case 3:}} In disjoint-cycle form, the possibilities are;
		\begin{align}
			\begin{split}
			(xyd_1)(d_3)\gamma_1, \ (xyd_3)(d_1)\gamma_2,\\
			(xd_1d_3)(y)\gamma_3, \ (yd_1d_3)(x)\gamma_4,
			\end{split}
		\end{align}
		
		\noindent and their inverses. 
		 % $(xd_1y)(d_3)\gamma_1^{-1}, \ (xd_3y)(d_1)\gamma_2^{-1}$, $(xd_3d_1)(y)\gamma_3^{-1},$ and $(yd_3d_1)(x)\gamma_4^{-1}$. 
		 Since  $\varphi \in Aut(G)$ iff $\varphi^{-1} \in Aut(G)$, it is enough to show that the first 4 types of permutations are not in $Aut(G)$. Now, $(xyd_1)(d_3)\gamma_1$ and $(xyd_3)(d_1)\gamma_2 \not \in Aut(G)$ by Proposition \ref{P7}, whilst $(xd_1d_3)(y)\gamma_3$ and $(yd_1d_3)(x)\gamma_4 \not\in Aut(G)$ by Proposition \ref{P9}.\\
		
		\noindent  \underline{\textsc{Case 4:}}  In disjoint-cycle form, the possibilities are;
		\begin{align}
			\begin{split}
			(xyd_1d_3)\gamma_1, \ (xyd_3d_1)\gamma_2, \ (xd_1yd_3)\gamma_3,\\
			(xd_3yd_1)\gamma_4, (xd_1d_3y)\gamma_5, (xd_3d_1y)\gamma_6.
			\end{split}
		\end{align}
		 If the first two permutation types are in $Aut(G)$, then squaring them implies that a permutation of the form $(xd_1)(yd_3)\beta$ or $(xd_3)(yd_1)\beta$ is in  $Aut(G)$, which is in contradiction with the assumed conditions. Squaring the third and fourth permutation types, implies that a permutation of the form $(xy)(d_1d_3)\beta \in Aut(G)$ but this contradicts Proposition \ref{P6} since $(xy)(d_1d_2)(d_3d_4)\alpha \in Aut(G)$. Lastly, squaring the fifth and sixth permutation types implies that  a permutation of the form $(xd_3)(yd_1)\beta$ or $(xd_1)(yd_3)\beta$ is in  $Aut(G)$ which is in contradiction with the assumed conditions.\\
		
		\noindent  (ii) Similar to (i), except that since we cannot apply Proposition \ref{P11} here, we have to also forbid $(xd_1)(y)(d_3)\beta_5$ and $(yd_3)(x)(d_1)\beta_6$ from being in $Aut(G)$, in the antecedent of the condition.
		\end{proof}
	
	\begin{lemma}\label{L5}
		
		Let $G$ be a graph with $D(G) = 2$ and $Det(G) = 2$. If $\{x, y\}$ is a determining set of $G$, then $(xy) \not\in Aut(G)$.
	\end{lemma}
	
	\begin{proof}
		
		Suppose to get a contradiction that $\{x, y\}$ is a determining set of such a graph $G$ and $(xy) \in Aut(G)$ (note that the remaining images of this permutation are single cycles).
		\begin{claim} \label{C1}
	 $(xv\dots )(y) \beta \not \in Aut(G)$ for any $\beta$ and $v \in V(G)\setminus\{x, y\}$.
\end{claim}

\begin{proof}
	
 Suppose, by way of contradiction that $(xv\dots )(y) \beta  \in Aut(G)$, for some $\beta$ and $v\in V(G)\setminus\{x,y\}$. Then $(xv \dots)(y) \beta \cdot (xy) \cdot ( (xv \dots)(y)\beta)^{-1} = (vy) \in Aut(G).$ Similarly, $(vx) \in Aut(G)$ using that fact that $(yv \dots)(x) \beta \in Aut(G)$ as $(xy)\cdot (xv \dots)(y) \beta \cdot (xy) = (yv \dots)(x) \beta \in Aut(G)$.

		However, $(vy), (vx)$ and $(xy)$ cannot all belong to $Aut(G)$  as this contradicts that there exists a 2 distinguishing coloring of $G$ since two colors cannot non-monochromatically color these 3 permutations at the same time. This ends the proof of the claim.
	\end{proof}

		Consider a coloring $k: V(G) \rightarrow \{\text{red}, \text{blue}\}$, which colors $y$ red and all other nodes of $G$ blue. Then, by Claim \ref{C1} and the fact  that $(x)(y)\gamma  \not\in Aut(G)\setminus \{e\}$ for any $\gamma$, by Proposition \ref{P4}, we get that no non-trivial permutation of $G$ is color-preserving.  Therefore, $k$ is a distinguishing coloring of $G$. Hence, $\{y\}$ is a distinguishing class and by Lemma \ref{L3}, this implies that $\{y\}$ is a determining set which contradicts that $Det(G) = 2$. Thus, $(xy) \not\in Aut(G)$. 
	\end{proof}
	
\noindent 	Now we are ready to prove our main theorem. Before doing so, we introduce the following terminology.
	
	\begin{definition}
			Let $v_1$, $v_2 \in V(G)$. We call a pair of nodes $(n_1, n_2)$, \textbf{\textit{a neighbour-non-neighbour pair of}}  $\{v_1, v_2\}$ iff either $\{n_1, v_1\}, \{n_2, v_2\} \in E(G)$ and $\{n_1, v_2\}, \{n_2, v_1\} \not\in E(G)$; or  $\{n_1, v_2\}, \{n_2, v_1\} \in E(G)$ and $\{n_1, v_1\}, \{n_2, v_2\} \not\in E(G)$.
	\end{definition}
	
	\begin{theorem}[Main Theorem]\label{main_thm}
		Let $G$ be a graph with $D(G) = 2$ and $Det(G) =2$. Then $2 \leq \rho(G) \leq 4$.
	\end{theorem}
	
	\begin{proof}
 		Suppose that $G$ is such a graph and that $\{x, y\}$ is a determining set of $G$. If $(xy)\alpha \not \in Aut(G)$, for any $\alpha$ extending $(xy)$, then $\varphi(\{x, y\}) = \{x, y\}$ implies that $\varphi = e$, the identity permutation, as  $(x)(y)\gamma \not\in Aut(G)\setminus\{e\}$ for any $\gamma$, by Proposition \ref{P4}. Therefore, by Lemma \ref{L3}, $\{x, y\}$ is a distinguishing class for $G$. Thus, coloring $x$ and $y$ red and the remaining nodes blue, is a distinguishing coloring of $G$ and $\rho(G) \leq 2$. Now, $\rho(G) \geq Det(G)$ since any distinguishing class of a distinguishing coloring is a determining set by Lemma \ref{L3}. Hence, $\rho(G)$ cannot be equal to 1 and $\rho(G) = 2$. If however,  $(xy)\alpha \in Aut(G)$, then by Proposition \ref{P5} (i), $\alpha$ consists of 2-cycles and single cycles. Furthermore, by Lemma \ref{L5}, $\alpha$ cannot consist of single cycles only. Therefore, $\alpha$ contains at least one 2-cycle.

        We observe that if there exists an automorphism of $G$ which flips $x$ and $y$ but does not flip any neighbour-non-neighbour pair of $\{x,y\}$ (only flips are possible as $\alpha$ can contain only 2-cycles and single cycles), then this implies that $G$ is invariant by just flipping $x$ and $y$, that is $(xy) \in Aut(G)$. This contradicts  Lemma \ref{L5}.   Therefore, there must exist, $d_1, \ d_2 \in V(G)$ such that $(d_1, d_2)$ is a neighbour-non-neighbour pair of $\{x, y\}$ and $\alpha$ contains $(d_1d_2)$.  Hence, 
        
\vspace{-6mm}
        
    \begin{align}
		(xy)(d_1 d_2)\alpha_2 \in Aut(G), 
		\end{align}

        \noindent for some $\alpha_2$, where $\alpha = (d_1 d_2)\alpha_2$. Also, without loss of generality we have that:

        \vspace{-3mm}
		
\begin{align}\label{xd_1}
			\{x, d_1\},  \{y, d_2\} \in E(G); \quad \text{and}\quad \{x, d_2\},  \{y, d_1\} \not \in E(G).
		\end{align}

        \vspace{2mm}
        
\noindent Whether there are edges between the remaining elements will be considered later.
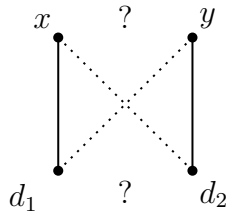
\begin{center}

\tikzset{every picture/.style={line width=0.75pt}} %set default line width to 0.75pt        

\begin{tikzpicture}[x=0.75pt,y=0.75pt,yscale=-1,xscale=1]

%Straight Lines [id:da9337374154058968] 
\draw    (79,51) -- (79,118) ;
\draw [shift={(79,118)}, rotate = 90] [color={rgb, 255:red, 0; green, 0; blue, 0 }  ][fill={rgb, 255:red, 0; green, 0; blue, 0 }  ][line width=0.75]      (0, 0) circle [x radius= 2.01, y radius= 2.01]   ;
\draw [shift={(79,51)}, rotate = 90] [color={rgb, 255:red, 0; green, 0; blue, 0 }  ][fill={rgb, 255:red, 0; green, 0; blue, 0 }  ][line width=0.75]      (0, 0) circle [x radius= 2.01, y radius= 2.01]   ;
%Straight Lines [id:da5623084137180538] 
\draw    (146,51) -- (146,118) ;
\draw [shift={(146,118)}, rotate = 90] [color={rgb, 255:red, 0; green, 0; blue, 0 }  ][fill={rgb, 255:red, 0; green, 0; blue, 0 }  ][line width=0.75]      (0, 0) circle [x radius= 2.01, y radius= 2.01]   ;
\draw [shift={(146,51)}, rotate = 90] [color={rgb, 255:red, 0; green, 0; blue, 0 }  ][fill={rgb, 255:red, 0; green, 0; blue, 0 }  ][line width=0.75]      (0, 0) circle [x radius= 2.01, y radius= 2.01]   ;
%Straight Lines [id:da9900399260164463] 
\draw  [dash pattern={on 0.84pt off 2.51pt}]  (79,118) -- (146,51) ;
%Straight Lines [id:da5504249033465496] 
\draw  [dash pattern={on 0.84pt off 2.51pt}]  (79,51) -- (146,118) ;

% Text Node
\draw (77,47.6) node [anchor=south east] [inner sep=0.75pt]    {$x$};
% Text Node
\draw (148,47.6) node [anchor=south west] [inner sep=0.75pt]    {$y$};
% Text Node
\draw (77,121.4) node [anchor=north east] [inner sep=0.75pt]    {$ \begin{array}{l}
d_{1}\\
\end{array}$};
% Text Node
\draw (148,121.4) node [anchor=north west][inner sep=0.75pt]    {$d_{2}$};
% Text Node
\draw (112.5,47.6) node [anchor=south] [inner sep=0.75pt]    {$?$};
% Text Node
\draw (112.5,121.4) node [anchor=north] [inner sep=0.75pt]    {$?$};

\end{tikzpicture}

\end{center}

		\begin{claim}\label{claim_rho<3}
	If $(xd_i)(y)\beta \not\in Aut(G)$ for any $i \in \{1, 2\}$ and any $\beta$, then $\rho(G) \leq 3$.
		\end{claim}		
 \begin{proof}[Proof of the claim:]
 	We will show that $\{x, y, d_1\}$ forms a distinguishing class of $G$. It is a determining set of $G$ as it contains $\{x, y\}$.  By Lemma \ref{L3}, this means that we need to show that no non-identity automorphism fixes $\{x, y, d_1\}$. Now,  $(xd_i)(y)\beta\not\in Aut(G)$ for any $i \in \{1, 2\}$  implies that  $(yd_j)(x)\gamma \not\in Aut(G)$ for any $j \in \{1, 2\}$ and extension  $\gamma$, by Proposition \ref{P10}. Therefore, it is left to show that there is no 3-cycle of $\{x, y, d_1\}$ as part of an automorphism of $Aut(G)$. The only possibilities are $(xyd_1)\gamma_1$ and $(xd_1y)\gamma_2$. Thus, we need to see that $(xyd_1)\gamma_1,(xd_1y)\gamma_2\not\in Aut(G)$ for any extension $\gamma_1$ and $ \gamma_2$. However, since $((xyd_1)\gamma_1)^2 = (xd_1y)\gamma_1^2$, where $\gamma_1^2$ can be taken to be $\gamma_2$, it is enough to show that $(xd_1y)\gamma_2\not\in Aut(G)$, for any extension $\gamma_2$.
		
		By way of contradiction, assume that $(xd_1y)\gamma_2 \in Aut(G)$, for some $\gamma_2$. Now, the  action of $(xd_1y)\gamma_2$ on the edge $\{x, d_1\}$, gives  $(xd_1y)\gamma_2 (\{x, d_1\}) = \{d_1, y\} \not \in E(G)$, which contradicts that automorphisms preserve edges. Therefore, only the identity automorphism preserves $\{x, y, d_1\}$ and so this forms a distinguishing class of $G$. Thus, $\rho(G) \leq 3$. 
		 \end{proof}
		\noindent By the claim above, we may henceforth assume  that,
		\begin{align}\label{(xd_iy)}
			(xd_i)(y)\beta \in Aut(G),
		\end{align}
	 for some $i \in \{1, 2\}$. This situation breaks down into the structure below:\\

		\noindent \textsc{Case 1:} There is \textit{exactly} one neighbour-non-neighbour pair of $\{x, y\}$, $(d_1, d_2)$  and \\ $(xy)(d_1 d_2)\alpha_2 \in Aut(G)$, where $\alpha = (d_1 d_2)\alpha_2$.\ This implies that, 
		\begin{align}\label{N(x)=N(y)}
			N(x) \setminus \{x, y, d_1, d_2\} = N(y) \setminus \{x, y, d_1, d_2\}
		\end{align}
		 and breaks into the subcases below:\\
		\vspace{-3mm}
		
		\begin{enumerate}
			\item[] \textsc{Case 1}a:  $(xy) (d_1 d_2) \in Aut(G)$. That is, $\alpha_2$ consists of single cycles only.\\
		
		\item[] \textsc{Case 1}b:  $(xy)(d_1d_2) \not\in Aut(G)$. Therefore, there must exist a neighbour-non-neighbour pair of $\{d_1, d_2\}$, say $(d_3, d_4)$. Else, $N(d_1) \setminus \{x, y, d_2\} = N(d_2) \setminus \{x, y, d_1\}$. Since $x$ and $y$ have the same neighbours  in $V(G)\setminus\{x, y\}$ apart from $d_1$ and $d_2$, this implies\textit{} that $(xy)(d_1 d_2) \in Aut(G)$ which contradicts the case assumption. Note that without loss of generality, we may assume here that $(xy)(d_1d_2)(d_3d_4)\alpha_3 \in Aut(G)$ (since for some neighbour-non-neighbour pair of $\{d_1, d_2\}$, it must be contained in  $\alpha$, which is the unique extension of $(xy)$).  This further breaks down into the cases:
			\end{enumerate}
	
		\begin{enumerate}
		\item[]
		\begin{enumerate}
			\item[] \textit{Case 1b (i)}: $(xd_i)(y) \beta_i \not \in Aut(G)$, for any $i \in \{3, 4\}$ and any extension $\beta_i$.
		\item[] \textit{Case 1b (ii)}: $(xd_i)(y)\beta_i \in Aut(G)$, for some $i \in \{3, 4\}$ and  extension $\beta_i$.
			\end{enumerate}
				\end{enumerate}
		
			\item[] 	\textsc{Case 2}: There is at least one other  neighbour-non-neighbour pair of $\{x, y\}$ other than $(d_1, d_2)$, say  $ (d_3, d_4)$, and $(xy) (d_1 d_2)(d_3 d_4)\alpha_3 \in Aut(G)$, where $\alpha = (d_1 d_2)(d_3 d_4) \alpha_3$. This is split into the cases:
		\begin{enumerate}
			\item[]  \textsc{Case  2}a: $(xd_i)(y)\beta_i \not \in Aut(G)$ for $ i \in \{3, 4\}$ and any extension $\beta_i$.
			\item[] \textsc{Case 2}b: $(xd_i)(y)\beta_i \in Aut(G)$ for some $ i \in \{3,4\}$ and extension $\beta_i$.\\
		\end{enumerate}
		We will show that  \textsc{Case 1}a is not possible, while  $\rho(G)\leq 3$ or $\rho(G)\leq 4$, for the remaining cases of \textit{Case 1b (i)}, \textit{Case 1b (ii)}, \textsc{Case 2}a and \textsc{Case 2}b.\\
		
		\noindent\underline{\textsc{Case 1}a:} For this case $(xy)(d_1d_2) \in Aut(G)$, with $(xd_i)(y)\beta_i $ for some $i \in \{1, 2\}$ by (\ref{(xd_iy)}).  Therefore, this splits into the cases (i) $(xd_1)(y)\beta_1 \in Aut(G)$ and (ii) $(xd_2)(y) \beta_2 \in Aut(G)$. These  are not exactly the same because $\{x, d_1\} \in E(G)$ but $\{x, d_2\} \not\in E(G)$.  \\
		
	\noindent	(i) Consider  $(xd_1)(y)\beta_1 \in Aut(G)$.  Now, there are two cases for what the image of $d_2$ is in $\beta_1$. Either $(xd_1)(y)(d_2) \beta_4$ or $(xd_1)(y)(d_2d_3)\beta_5$, where $d_2 \not = d_3$, for some $d_3\in V(G)$.  Assume the former, that is
		\begin{align}
			(xd_1)(y)(d_2) \beta_4 \in Aut(G).
		\end{align}
		\noindent  Then,
        \vspace{-5mm}
        \begin{align}
(xy)(d_1d_2) \cdot (xd_1)(y)(d_2) \beta_4 \cdot (xy)(d_1d_2) = (yd_2)(x)(d_1)\beta_4 \in Aut(G).
		  \end{align} Now by Proposition \ref{P5} (ii), $\beta_4$ consists of 2-cycles and single cycles. Therefore, 
	  \begin{align}
 (xd_1)(y)(d_2) \beta_4 \cdot (yd_2)(x)(d_1)\beta_4 = (xd_1)(yd_2) \beta_4^2 = (xd_1)(yd_2) \in Aut(G).	  
\end{align}
 It follows that
  $N(d_1) \setminus \{x, y, d_1, d_2\} = N(x) \setminus \{x, y, d_1, d_2\}$ and $N(d_2) \setminus \{x, y, d_1, d_2\} = N(y) \setminus \{x, y, d_1, d_2\}$. This and  (\ref{N(x)=N(y)}) which holds under \textsc{Case 1}, implies $N(x) \setminus \{x, y, d_1, d_2\} = N(y) \setminus \{x, y, d_1, d_2\} = N(d_1) \setminus \{x, y, d_1, d_2\} = N(d_2) \setminus \{x, y, d_1, d_2\}$. Thus, by Lemma \ref{L2}, a 2-distinguishing coloring of $G$  restricted to $\{x, y, d_1, d_2\}$ must be a 2-distinguishing coloring of the induced subgraph of those nodes. Let us discover what this induced subgraph is: Since the action, $ (xd_1)(y)\beta_1(\{y, d_1\}) = \{x, y\}$ and $\{y, d_1\} \not \in E(G)$, then $\{x, y\} \not\in E(G)$ as automorphisms preserve non-edges. Similarly, $(xd_1)(yd_2) (\{x, y\}) = \{d_1, d_2\}$, and so $\{d_1, d_2\} \not \in E(G)$.  Therefore, the induced subgraph of $\{x, y, d_1, d_2\}$ consists  of 2 disjoint edges:
  \begin{center}

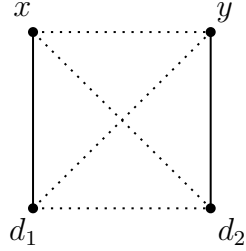
\tikzset{every picture/.style={line width=0.75pt}} %set default line width to 0.75pt        

\begin{tikzpicture}[x=0.75pt,y=0.75pt,yscale=-1,xscale=1]
	
	%Straight Lines [id:da2644301350415974] 
	\draw    (205,84.67) -- (205,173.33) ;
	\draw [shift={(205,173.33)}, rotate = 90] [color={rgb, 255:red, 0; green, 0; blue, 0 }  ][fill={rgb, 255:red, 0; green, 0; blue, 0 }  ][line width=0.75]      (0, 0) circle [x radius= 2.01, y radius= 2.01]   ;
	\draw [shift={(205,84.67)}, rotate = 90] [color={rgb, 255:red, 0; green, 0; blue, 0 }  ][fill={rgb, 255:red, 0; green, 0; blue, 0 }  ][line width=0.75]      (0, 0) circle [x radius= 2.01, y radius= 2.01]   ;
	%Straight Lines [id:da7018673546337553] 
	\draw    (293.67,84.67) -- (293.67,173.33) ;
	\draw [shift={(293.67,173.33)}, rotate = 90] [color={rgb, 255:red, 0; green, 0; blue, 0 }  ][fill={rgb, 255:red, 0; green, 0; blue, 0 }  ][line width=0.75]      (0, 0) circle [x radius= 2.01, y radius= 2.01]   ;
	\draw [shift={(293.67,84.67)}, rotate = 90] [color={rgb, 255:red, 0; green, 0; blue, 0 }  ][fill={rgb, 255:red, 0; green, 0; blue, 0 }  ][line width=0.75]      (0, 0) circle [x radius= 2.01, y radius= 2.01]   ;
	%Straight Lines [id:da16178146984010744] 
	\draw  [dash pattern={on 0.84pt off 2.51pt}]  (205,173.33) -- (293.67,173.33) ;
	%Straight Lines [id:da8495471600698603] 
	\draw  [dash pattern={on 0.84pt off 2.51pt}]  (205,84.67) -- (293.67,84.67) ;
	%Straight Lines [id:da3063503089601576] 
	\draw  [dash pattern={on 0.84pt off 2.51pt}]  (205,173.33) -- (293.67,84.67) ;
	%Straight Lines [id:da5076193516857332] 
	\draw  [dash pattern={on 0.84pt off 2.51pt}]  (205,84.67) -- (293.67,173.33) ;
	
	% Text Node
	\draw (193.67,68.4) node [anchor=north west][inner sep=0.75pt]    {$x$};
	% Text Node
	\draw (295,68.4) node [anchor=north west][inner sep=0.75pt]    {$y$};
	% Text Node
	\draw (295.67,176.73) node [anchor=north west][inner sep=0.75pt]    {$d_{2}$};
	% Text Node
	\draw (191.67,176.73) node [anchor=north west][inner sep=0.75pt]    {$d_{1}$};

\end{tikzpicture}

  \end{center}
  
 However, 2 disjoint edges cannot be distinguished with 2-colors and so we have a contradiction for the existence of a distinguishing 2-coloring for $G$  by Lemma \ref{L2}. Thus, this case is not possible as $D(G) = 2$.\\
		
		For the latter case,

\vspace{-10mm}
		\begin{align}
		(xd_1)(y)(d_2d_3)\beta_5 \in Aut(G).
		\end{align}
		 Then, 
         \vspace{-5mm}
		 \begin{align}
		 	(xy)(d_1d_2) \cdot (xd_1)(y)(d_2d_3) \beta_5 \cdot (xy)(d_1d_2) = (yd_2)(x)(d_1d_3)\beta_5 \in Aut(G).
		 \end{align}
		  Now by Proposition \ref{P5} (ii), $\beta_5$ consists of 2-cycles and single cycles. Therefore,
		  \begin{align}
		  	(xd_1)(y)(d_2d_3) \beta_5 \cdot (yd_2)(x)(d_1d_3)\beta_5 = (xd_1d_2yd_3) \beta_5^2 = (xd_1d_2yd_3) \in Aut(G).
		  \end{align}
		   This implies that each power of this cycle is in $Aut(G)$. Thus, for any  $i, j \in \{x, y, d_1, d_2, d_3\}$ there is a power of this cycle, $\varphi_{ij}$ say, which is such that $\varphi_{ij}(i) = j$. It follows that $N(i) \setminus \{ x, y, d_1, d_2, d_3\} = N(j) \setminus \{ x, y, d_1, d_2, d_3\}$, for any $i, j \in \{x, y, d_1, d_2, d_3\}$. By Lemma \ref{L2}, we know that a 2-distinguishing coloring of $G$ when restricted to $\{x, y, d_1, d_2, d_3\}$ must be a 2-distinguishing coloring of the induced subgraph.   Since the action $ (xd_1)(y)\beta_1(\{y, d_1\}) = \{x, y\}$ and $\{y, d_1\} \not \in E(G)$, then $\{x, y\} \not\in 
		   E(G)$. Also, $(xd_1d_2yd_3)$ acting on edges  $\{x, d_1\}$ and $\{y, d_2\}$, takes them to $\{d_1, d_2\}$ and $\{y, d_3\}$ and so these must be edges. It also takes the non-edges $\{y, d_1\}$ and $\{x, y\}$ to $\{d_2, d_3\}$ and $\{d_1, d_3\}$ respectively. Moreover, $(xd_1d_2yd_3)^2 = (xd_2d_3d_1y)$ takes the edge $\{y, d_2\}$  to $\{x, d_3\}$. Overall, the induced graph of $\{x, y, d_1, d_2, d_3\}$ is a 5-cycle, $C_5$.

		    \vspace{-5mm} 
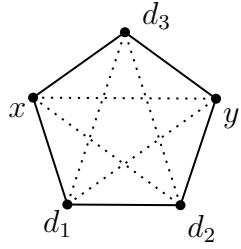
		   \begin{center}

		   	\tikzset{every picture/.style={line width=0.75pt}} %set default line width to 0.75pt        
		   	
		   	\begin{tikzpicture}[x=0.75pt,y=0.75pt,yscale=-1,xscale=1]
		   		
		   		%Shape: Regular Polygon [id:dp8000869193195916] 
		   		\draw   (318.74,107.45) -- (300.98,161) -- (244.55,160.66) -- (227.44,106.89) -- (273.29,74) -- cycle ;
		   		%Straight Lines [id:da46076389217983005] 
		   		\draw    (273.29,74) ;
		   		\draw [shift={(273.29,74)}, rotate = 0] [color={rgb, 255:red, 0; green, 0; blue, 0 }  ][fill={rgb, 255:red, 0; green, 0; blue, 0 }  ][line width=0.75]      (0, 0) circle [x radius= 2.01, y radius= 2.01]   ;
		   		%Straight Lines [id:da29203258422446843] 
		   		\draw    (227.44,106.89) ;
		   		\draw [shift={(227.44,106.89)}, rotate = 0] [color={rgb, 255:red, 0; green, 0; blue, 0 }  ][fill={rgb, 255:red, 0; green, 0; blue, 0 }  ][line width=0.75]      (0, 0) circle [x radius= 2.01, y radius= 2.01]   ;
		   		%Straight Lines [id:da10230865773526054] 
		   		\draw    (318.74,107.45) ;
		   		\draw [shift={(318.74,107.45)}, rotate = 0] [color={rgb, 255:red, 0; green, 0; blue, 0 }  ][fill={rgb, 255:red, 0; green, 0; blue, 0 }  ][line width=0.75]      (0, 0) circle [x radius= 2.01, y radius= 2.01]   ;
		   		%Straight Lines [id:da9348890261483245] 
		   		\draw    (244.55,160.66) ;
		   		\draw [shift={(244.55,160.66)}, rotate = 0] [color={rgb, 255:red, 0; green, 0; blue, 0 }  ][fill={rgb, 255:red, 0; green, 0; blue, 0 }  ][line width=0.75]      (0, 0) circle [x radius= 2.01, y radius= 2.01]   ;
		   		%Straight Lines [id:da267419300796758] 
		   		\draw    (300.98,161) ;
		   		\draw [shift={(300.98,161)}, rotate = 0] [color={rgb, 255:red, 0; green, 0; blue, 0 }  ][fill={rgb, 255:red, 0; green, 0; blue, 0 }  ][line width=0.75]      (0, 0) circle [x radius= 2.01, y radius= 2.01]   ;
		   		%Straight Lines [id:da17647639757310962] 
		   		\draw  [dash pattern={on 0.84pt off 2.51pt}]  (227.44,106.89) -- (318.74,107.45) ;
		   		%Straight Lines [id:da6932702917301574] 
		   		\draw  [dash pattern={on 0.84pt off 2.51pt}]  (244.55,160.66) -- (318.74,107.45) ;
		   		%Straight Lines [id:da10941261648743694] 
		   		\draw  [dash pattern={on 0.84pt off 2.51pt}]  (300.98,161) -- (273.29,74) ;
		   		%Straight Lines [id:da38096619476012594] 
		   		\draw  [dash pattern={on 0.84pt off 2.51pt}]  (244.55,160.66) -- (273.29,74) ;
		   		%Straight Lines [id:da00841672155331974] 
		   		\draw  [dash pattern={on 0.84pt off 2.51pt}]  (227.44,106.89) -- (300.98,161) ;
		   		
		   		% Text Node
		   		\draw (213.67,108.07) node [anchor=north west][inner sep=0.75pt]    {$x$};
		   		% Text Node
		   		\draw (231,162.07) node [anchor=north west][inner sep=0.75pt]    {$d_{1}$};
		   		% Text Node
		   		\draw (302.98,164.4) node [anchor=north west][inner sep=0.75pt]    {$d_{2}$};
		   		% Text Node
		   		\draw (320.74,110.85) node [anchor=north west][inner sep=0.75pt]    {$y$};
		   		% Text Node
		   		\draw (273.67,55.4) node [anchor=north west][inner sep=0.75pt]    {$ \begin{array}{l}
		   				d_{3}\\
		   			\end{array}$};

		   	\end{tikzpicture}
		   	
		   \end{center}

		    \vspace{-5mm}
		   
		   However, $D(C_5)=3$, which contradicts  Lemma \ref{L2} and the existence of a distinguishing 2-coloring for $G$. Thus, this case  is also not possible, as $D(G)=2$.\\

		(ii) Consider  $(xd_2)(y)\beta_1 \in Aut(G)$.  Now, there are two cases for what the image of $d_1$ is in $\beta_2$. That is, either $(xd_2)(y)(d_1) \beta_4$ or $(xd_2)(y)(d_1d_3)\beta_5$, where $d_1 \not = d_3$. 
		
	\noindent	Assume the former case of
		\begin{align}
		(xd_2)(y)(d_1) \beta_4 \in Aut(G).
		\end{align}
		 Then,

         \vspace{-8mm}
		 \begin{align}
		 	(xy)(d_1d_2) \cdot (xd_2)(y)(d_1) \beta_4 \cdot (xy)(d_1d_2) = (yd_1)(x)(d_2)\beta_4 \in Aut(G).
		 \end{align}
		  Now by Proposition \ref{P5} (ii), $\beta_4$ consists of 2-cycles and single cycles. Therefore,
		  \begin{align}
(xd_2)(y)(d_1) \beta_4 \cdot (yd_1)(x)(d_2)\beta_4 = (xd_2)(yd_1) \beta_4^2 = (xd_2)(yd_1) \in Aut(G).
		  \end{align}
		  It follows that $N(d_2) \setminus \{x, y, d_1, d_2\} = N(x) \setminus \{x, y, d_1, d_2\}$ and $N(d_1) \setminus \{x, y, d_1, d_2\} = N(y) \setminus \{x, y, d_1, d_2\}$. Together with (\ref{N(x)=N(y)}), this implies that  $N(x) \setminus \{x, y, d_1, d_2\} = N(y) \setminus \{x, y, d_1, d_2\} = N(d_1) \setminus \{x, y, d_1, d_2\} = N(d_2) \setminus \{x, y, d_1, d_2\}$. By Lemma \ref{L2}, the induced subgraph of $\{x, y, d_1, d_2\}$ must have a 2-distinguishing coloring. Now, as the action $ (xd_2)(y)\beta_1(\{y, d_2\}) = \{x, y\}$ and $\{y, d_2\}  \in E(G)$, then $\{x, y\} \in E(G)$. Also, $(xd_1)(yd_2) (\{x, y\}) = \{d_1, d_2\}$, and so $\{d_1, d_2\} \in E(G)$. Therefore, the induced subgraph of $\{x, y, d_1, d_2\}$ is $C_4$:
		  \begin{center}

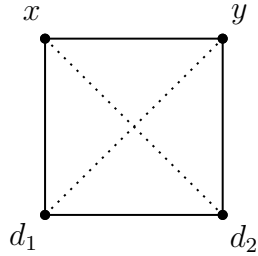
		  	\tikzset{every picture/.style={line width=0.75pt}} %set default line width to 0.75pt        
		  	
		  	\begin{tikzpicture}[x=0.75pt,y=0.75pt,yscale=-1,xscale=1]
		  		%uncomment if require: \path (0,386); %set diagram left start at 0, and has height of 386
		  		
		  		%Straight Lines [id:da3063503089601576] 
		  		\draw  [dash pattern={on 0.84pt off 2.51pt}]  (201.67,296) -- (290.33,207.33) ;
		  		%Straight Lines [id:da5076193516857332] 
		  		\draw  [dash pattern={on 0.84pt off 2.51pt}]  (201.67,207.33) -- (290.33,296) ;
		  		%Shape: Square [id:dp6416173780407373] 
		  		\draw   (201.67,207.33) -- (290.33,207.33) -- (290.33,296) -- (201.67,296) -- cycle ;
		  		%Straight Lines [id:da8807784599638624] 
		  		\draw    (201.67,207.33) ;
		  		\draw [shift={(201.67,207.33)}, rotate = 0] [color={rgb, 255:red, 0; green, 0; blue, 0 }  ][fill={rgb, 255:red, 0; green, 0; blue, 0 }  ][line width=0.75]      (0, 0) circle [x radius= 2.01, y radius= 2.01]   ;
		  		\draw [shift={(201.67,207.33)}, rotate = 0] [color={rgb, 255:red, 0; green, 0; blue, 0 }  ][fill={rgb, 255:red, 0; green, 0; blue, 0 }  ][line width=0.75]      (0, 0) circle [x radius= 2.01, y radius= 2.01]   ;
		  		%Straight Lines [id:da64150520434865] 
		  		\draw    (290.33,207.33) ;
		  		\draw [shift={(290.33,207.33)}, rotate = 0] [color={rgb, 255:red, 0; green, 0; blue, 0 }  ][fill={rgb, 255:red, 0; green, 0; blue, 0 }  ][line width=0.75]      (0, 0) circle [x radius= 2.01, y radius= 2.01]   ;
		  		\draw [shift={(290.33,207.33)}, rotate = 0] [color={rgb, 255:red, 0; green, 0; blue, 0 }  ][fill={rgb, 255:red, 0; green, 0; blue, 0 }  ][line width=0.75]      (0, 0) circle [x radius= 2.01, y radius= 2.01]   ;
		  		%Straight Lines [id:da3501371660268888] 
		  		\draw    (201.67,296) ;
		  		\draw [shift={(201.67,296)}, rotate = 0] [color={rgb, 255:red, 0; green, 0; blue, 0 }  ][fill={rgb, 255:red, 0; green, 0; blue, 0 }  ][line width=0.75]      (0, 0) circle [x radius= 2.01, y radius= 2.01]   ;
		  		\draw [shift={(201.67,296)}, rotate = 0] [color={rgb, 255:red, 0; green, 0; blue, 0 }  ][fill={rgb, 255:red, 0; green, 0; blue, 0 }  ][line width=0.75]      (0, 0) circle [x radius= 2.01, y radius= 2.01]   ;
		  		%Straight Lines [id:da7790418325556563] 
		  		\draw    (290.33,296) ;
		  		\draw [shift={(290.33,296)}, rotate = 0] [color={rgb, 255:red, 0; green, 0; blue, 0 }  ][fill={rgb, 255:red, 0; green, 0; blue, 0 }  ][line width=0.75]      (0, 0) circle [x radius= 2.01, y radius= 2.01]   ;
		  		\draw [shift={(290.33,296)}, rotate = 0] [color={rgb, 255:red, 0; green, 0; blue, 0 }  ][fill={rgb, 255:red, 0; green, 0; blue, 0 }  ][line width=0.75]      (0, 0) circle [x radius= 2.01, y radius= 2.01]   ;
		  		
		  		% Text Node
		  		\draw (189,189.73) node [anchor=north west][inner sep=0.75pt]    {$x$};
		  		% Text Node
		  		\draw (293,188.4) node [anchor=north west][inner sep=0.75pt]    {$y$};
		  		% Text Node
		  		\draw (292.33,299.4) node [anchor=north west][inner sep=0.75pt]    {$d_{2}$};
		  		% Text Node
		  		\draw (182.33,298.73) node [anchor=north west][inner sep=0.75pt]    {$d_{1}$};
		  				  		
		  	\end{tikzpicture}
		  	
		  \end{center} 
		   However, $C_4$ cannot be distinguished with 2-colors, which contradicts Lemma \ref{L2} and the 2-coloring for $G$. Thus, this case is not possible as $D(G) = 2$.\\
		
		For the latter case, 
		 \vspace{-5mm}
		\begin{align}
		(xd_2)(y)(d_1d_3)\beta_5 \in Aut(G).
		\end{align}
		Then, 
        \vspace{-5mm}
		\begin{align}
(xy)(d_1d_2) \cdot (xd_2)(y)(d_1d_3) \beta_5 \cdot (xy)(d_1d_2) = (yd_1)(x)(d_2d_3)\beta_5 \in Aut(G).
		\end{align}
	  Now by Proposition \ref{P5} (ii), $\beta_5$ consists of 2-cycles and single cycles. Therefore, 
	  \begin{align}
	  	(xd_2)(y)(d_1d_3) \beta_5 \cdot (yd_1)(x)(d_2d_3)\beta_5 = (xd_2d_1yd_3) \beta_5^2 = (xd_2d_1yd_3) \in Aut(G), 
	  \end{align}
	  
	   \noindent and each power of this cycle is in $Aut(G)$. Hence, for any  $i, j \in \{x, y, d_1, d_2, d_3\}$ there is a power of this cycle, $\varphi_{ij}$ say, which is such that $\varphi_{ij}(i) = j$. This implies that $N(i) \setminus \{ x, y, d_1, d_2, d_3\} = N(j) \setminus \{ x, y, d_1, d_2, d_3\}$, for any $i, j \in \{x, y, d_1, d_2, d_3\}$. By Lemma \ref{L2}, the induced subgraph of $\{x, y, d_1, d_2, d_3\}$ must have a 2-distinguishing coloring.  Now, as $ (xd_2)(y)\beta_1(\{y, d_2\}) = \{x, y\}$ and $\{y, d_2\}  \in E(G)$, then $\{x, y\} \in E(G)$. Also, $(xd_2d_1yd_3)$ acting on the edges $\{x, y\}$ and $\{y, d_2\}$, takes them to $\{d_2, d_3\}$ and $\{d_1, d_3\}$ and so these must be edges. It also takes the non-edges $\{x, d_2\}$ and $\{y, d_1\}$ to $\{d_1, d_2\}$ and $\{y, d_3\}$ respectively. Moreover, $(xd_2d_1yd_3)^2 = (xd_1d_3d_2y)$ takes the non-edge $\{y, d_1\}$ to $\{x, d_3\}$. Overall, the induced graph of $\{x, y, d_1, d_2, d_3\}$ is a 5-cycle:
 \vspace{-5mm}

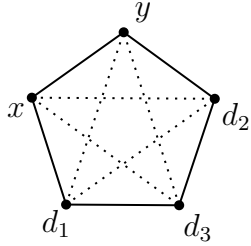
	   \begin{center}

	   	\tikzset{every picture/.style={line width=0.75pt}} %set default line width to 0.75pt        
	   	
	   	\begin{tikzpicture}[x=0.75pt,y=0.75pt,yscale=-1,xscale=1]
	   		
	   		%Shape: Regular Polygon [id:dp8000869193195916] 
	   		\draw   (318.74,107.45) -- (300.98,161) -- (244.55,160.66) -- (227.44,106.89) -- (273.29,74) -- cycle ;
	   		%Straight Lines [id:da46076389217983005] 
	   		\draw    (273.29,74) ;
	   		\draw [shift={(273.29,74)}, rotate = 0] [color={rgb, 255:red, 0; green, 0; blue, 0 }  ][fill={rgb, 255:red, 0; green, 0; blue, 0 }  ][line width=0.75]      (0, 0) circle [x radius= 2.01, y radius= 2.01]   ;
	   		%Straight Lines [id:da29203258422446843] 
	   		\draw    (227.44,106.89) ;
	   		\draw [shift={(227.44,106.89)}, rotate = 0] [color={rgb, 255:red, 0; green, 0; blue, 0 }  ][fill={rgb, 255:red, 0; green, 0; blue, 0 }  ][line width=0.75]      (0, 0) circle [x radius= 2.01, y radius= 2.01]   ;
	   		%Straight Lines [id:da10230865773526054] 
	   		\draw    (318.74,107.45) ;
	   		\draw [shift={(318.74,107.45)}, rotate = 0] [color={rgb, 255:red, 0; green, 0; blue, 0 }  ][fill={rgb, 255:red, 0; green, 0; blue, 0 }  ][line width=0.75]      (0, 0) circle [x radius= 2.01, y radius= 2.01]   ;
	   		%Straight Lines [id:da9348890261483245] 
	   		\draw    (244.55,160.66) ;
	   		\draw [shift={(244.55,160.66)}, rotate = 0] [color={rgb, 255:red, 0; green, 0; blue, 0 }  ][fill={rgb, 255:red, 0; green, 0; blue, 0 }  ][line width=0.75]      (0, 0) circle [x radius= 2.01, y radius= 2.01]   ;
	   		%Straight Lines [id:da267419300796758] 
	   		\draw    (300.98,161) ;
	   		\draw [shift={(300.98,161)}, rotate = 0] [color={rgb, 255:red, 0; green, 0; blue, 0 }  ][fill={rgb, 255:red, 0; green, 0; blue, 0 }  ][line width=0.75]      (0, 0) circle [x radius= 2.01, y radius= 2.01]   ;
	   		%Straight Lines [id:da17647639757310962] 
	   		\draw  [dash pattern={on 0.84pt off 2.51pt}]  (227.44,106.89) -- (318.74,107.45) ;
	   		%Straight Lines [id:da6932702917301574] 
	   		\draw  [dash pattern={on 0.84pt off 2.51pt}]  (244.55,160.66) -- (318.74,107.45) ;
	   		%Straight Lines [id:da10941261648743694] 
	   		\draw  [dash pattern={on 0.84pt off 2.51pt}]  (300.98,161) -- (273.29,74) ;
	   		%Straight Lines [id:da38096619476012594] 
	   		\draw  [dash pattern={on 0.84pt off 2.51pt}]  (244.55,160.66) -- (273.29,74) ;
	   		%Straight Lines [id:da00841672155331974] 
	   		\draw  [dash pattern={on 0.84pt off 2.51pt}]  (227.44,106.89) -- (300.98,161) ;
	   		
	   		% Text Node
	   		\draw (213.67,108.07) node [anchor=north west][inner sep=0.75pt]    {$x$};
	   		% Text Node
	   		\draw (231,162.07) node [anchor=north west][inner sep=0.75pt]    {$d_{1}$};
	   		% Text Node
	   		\draw (320.74,108.07) node [anchor=north west][inner sep=0.75pt]    {$d_{2}$};
	   		% Text Node
	   		\draw (277.41,57.51) node [anchor=north west][inner sep=0.75pt]    {$y$};
	   		% Text Node
	   		\draw (295,162.07) node [anchor=north west][inner sep=0.75pt]    {$ \begin{array}{l}
	   				d_{3}\\
	   			\end{array}$};

	   	\end{tikzpicture}
	   	
	   \end{center}

	    \vspace{-5mm}
	   
	 \noindent  However, $C_5$ cannot be distinguished with 2 colors and so this contradicts that $D(G)=2$ by Lemma \ref{L2}.

		\hfill\\ 
	\noindent \underline{\textsc{Case 1}b:}   Here, $(xy)(d_1d_2)(d_3d_4)\alpha_3 \in Aut(G)$, where $(d_3, d_4)$ is a neighbour-non-neighbour pair of $\{d_1, d_2\}$. Suppose, without loss of generality, that 
	 \begin{align}\label{d_1d_3}
	 \{d_1, d_3\}, \ \{d_2, d_4\} \in E(G) \quad \{d_3, d_2\}, \ \{d_4, d_1\} \not \in E(G).
	 \end{align}
     Thus we have a pattern as follows:
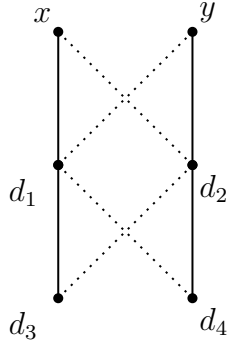
\begin{center}

\tikzset{every picture/.style={line width=0.75pt}} %set default line width to 0.75pt        

\begin{tikzpicture}[x=0.75pt,y=0.75pt,yscale=-1,xscale=1]

%Straight Lines [id:da9337374154058968] 
\draw    (79,51) -- (79,118) ;
\draw [shift={(79,118)}, rotate = 90] [color={rgb, 255:red, 0; green, 0; blue, 0 }  ][fill={rgb, 255:red, 0; green, 0; blue, 0 }  ][line width=0.75]      (0, 0) circle [x radius= 2.01, y radius= 2.01]   ;
\draw [shift={(79,51)}, rotate = 90] [color={rgb, 255:red, 0; green, 0; blue, 0 }  ][fill={rgb, 255:red, 0; green, 0; blue, 0 }  ][line width=0.75]      (0, 0) circle [x radius= 2.01, y radius= 2.01]   ;
%Straight Lines [id:da5623084137180538] 
\draw    (146,51) -- (146,118) ;
\draw [shift={(146,118)}, rotate = 90] [color={rgb, 255:red, 0; green, 0; blue, 0 }  ][fill={rgb, 255:red, 0; green, 0; blue, 0 }  ][line width=0.75]      (0, 0) circle [x radius= 2.01, y radius= 2.01]   ;
\draw [shift={(146,51)}, rotate = 90] [color={rgb, 255:red, 0; green, 0; blue, 0 }  ][fill={rgb, 255:red, 0; green, 0; blue, 0 }  ][line width=0.75]      (0, 0) circle [x radius= 2.01, y radius= 2.01]   ;
%Straight Lines [id:da9900399260164463] 
\draw  [dash pattern={on 0.84pt off 2.51pt}]  (79,118) -- (146,51) ;
%Straight Lines [id:da5504249033465496] 
\draw  [dash pattern={on 0.84pt off 2.51pt}]  (79,51) -- (146,118) ;
%Straight Lines [id:da1376378383062169] 
\draw    (79,118) -- (79,185) ;
\draw [shift={(79,185)}, rotate = 90] [color={rgb, 255:red, 0; green, 0; blue, 0 }  ][fill={rgb, 255:red, 0; green, 0; blue, 0 }  ][line width=0.75]      (0, 0) circle [x radius= 2.01, y radius= 2.01]   ;
\draw [shift={(79,118)}, rotate = 90] [color={rgb, 255:red, 0; green, 0; blue, 0 }  ][fill={rgb, 255:red, 0; green, 0; blue, 0 }  ][line width=0.75]      (0, 0) circle [x radius= 2.01, y radius= 2.01]   ;
%Straight Lines [id:da5080012657093291] 
\draw    (146,118) -- (146,185) ;
\draw [shift={(146,185)}, rotate = 90] [color={rgb, 255:red, 0; green, 0; blue, 0 }  ][fill={rgb, 255:red, 0; green, 0; blue, 0 }  ][line width=0.75]      (0, 0) circle [x radius= 2.01, y radius= 2.01]   ;
\draw [shift={(146,118)}, rotate = 90] [color={rgb, 255:red, 0; green, 0; blue, 0 }  ][fill={rgb, 255:red, 0; green, 0; blue, 0 }  ][line width=0.75]      (0, 0) circle [x radius= 2.01, y radius= 2.01]   ;
%Straight Lines [id:da13091617855307147] 
\draw  [dash pattern={on 0.84pt off 2.51pt}]  (79,185) -- (146,118) ;
%Straight Lines [id:da0647188067066462] 
\draw  [dash pattern={on 0.84pt off 2.51pt}]  (79,118) -- (146,185) ;

% Text Node
\draw (77,47.6) node [anchor=south east] [inner sep=0.75pt]    {$x$};
% Text Node
\draw (148,47.6) node [anchor=south west] [inner sep=0.75pt]    {$y$};
% Text Node
\draw (77,121.4) node [anchor=north east] [inner sep=0.75pt]    {$ \begin{array}{l}
d_{1}\\
\end{array}$};
% Text Node
\draw (148,121.4) node [anchor=north west][inner sep=0.75pt]    {$d_{2}$};
% Text Node
\draw (77,188.4) node [anchor=north east] [inner sep=0.75pt]    {$ \begin{array}{l}
d_{3}\\
\end{array}$};
% Text Node
\draw (148,188.4) node [anchor=north west][inner sep=0.75pt]    {$d_{4}$};

\end{tikzpicture}

\end{center}
     \noindent \underline{\textit{Case 1b (i)}:} By (\ref{(xd_iy)}) we have $(xd_i)(y) \beta_i \in Aut(G)$, for some $i \in \{1, 2\}$, and $(xd_j)(y) \beta_j \not\in Aut(G)$, for all $j \in \{3, 4\}$, and extension $\beta_j$.
	We will show that $\{x, y, d_3\}$ forms a distinguishing class of $G$ using Lemma \ref{L3}. This is a determining set as it contains $\{x, y\}$.  Now, by case assumption $(xd_j)(y)\beta_j \not\in Aut(G)$ for any $j \in \{3, 4\}$ and extension $\beta_j$. This implies that  $(yd_j)(x)\beta_{j} \not\in Aut(G)$ for any $j \in \{3, 4\}$ and extension $\beta_j$, by Proposition \ref{P10}. Therefore, it is left to show that there is no 3-cycle of $\{x, y, d_3\}$ as part of a permutation of $Aut(G)$. The only possibilities are $(xyd_3)\gamma_1$ and $(xd_3y)\gamma_2$, where $\gamma_1,\gamma_2$ are permutations of $V(G) \setminus\{x, y, d_3\}$. However, as $((xyd_3)\gamma_1)^2 = (xd_3y)\gamma_1^2$, where $\gamma_1^2$ can be taken to be $\gamma_2$, it is enough to show that $(xd_3y)\gamma_2$ is not in $Aut(G).$  Next, to get a contradiction, assume that $(xd_3y)\gamma_2 \in Aut(G)$, for some $\gamma_2$.
		
		\begin{claim}\label{claim_paths}
		In $G$, $d_3 $ has one more 2-edge-path connecting it to $x$ than it has connecting it to $y$.
	\end{claim}

	\begin{proof}[Proof of the claim:]
	
 Now, $d_3vx$ is the form of a 2-edge-path (2-path) from $d_3$ to $x$, where $v$ is a variable for a node which can take values from $V(G) \setminus \{d_3, x\}$. If $v \in V(G) \setminus \{x, y, d_1, d_2\}$, then $d_3vx$ is a path in $G$ iff $d_3vy$ is a path in $G$, since  $\{v, x\} \in E(G)$ iff $\{v, y\} \in E(G)$ as $N(x) \setminus \{x, y, d_1, d_2\} = N(y) \setminus \{x, y, d_1, d_2\}$ by (\ref{N(x)=N(y)}). We are left with considering $v$ from $\{x, y, d_1, d_2\}$.  Next, observe that $(xd_3y)\gamma_2(\{d_3, x\}) = \{y, d_3\}$, therefore $\{d_3, x\} \in E(G)$ iff $\{ d_3, y\} \in E(G)$. This covers the case of $v=x$ or $v=y$. Hence, $d_3xy$ is a 2-path  of $G$ iff $d_3yx$ is a 2-path of $G$. For $v=d_2$, $d_3d_2x$ is not a 2-path in $G$ as $\{d_2, x\} \not\in E(G)$ by (\ref{xd_1}), and also $d_3d_2y$ is also not a 2-path as  $\{d_3, d_2\} \not\in E(G)$ by (\ref{d_1d_3}).  For $v =d_1$ however, we have that $d_3d_1x$ is a 2-path of $G$ by (\ref{d_1d_3}), but $d_3d_1y$ is not a 2-path since $\{d_1, y\} \not \in E(G)$ by (\ref{xd_1}). Overall: $d_3vx$ is a 2-path of $G$ iff $d_3vy$ is a 2-path of $G$ if $v \in V(G) \setminus \{ x, y, d_1\}$; $d_3yx$ is a 2-path iff $d_3xy$ is a 2-path but; $d_3d_1x$ is a 2-path whilst $d_3d_1y$ is not a 2-path. Therefore, $d_3$ has one more 2-path to $x$ than it does to $y$. 	\end{proof}
		
		Notice that the action of $(xd_3y)\gamma_2$ on the pair $(d_3, x)$ takes it to $(y, d_3)$. However, automorphisms preserve the number of 2-paths between nodes (note that number of 2-paths from $d_3$ to $y$ equals number of 2-paths from $y$ to $d_3$) and so this contradicts Claim \ref{claim_paths}. Therefore, $(xd_3y)\gamma_2 \not\in Aut(G), \ \{x, y, d_3\}$ forms a distinguishing class of $G$ and $\rho (G) \leq 3$.\\
		
		\noindent \underline{\textit{Case 1b (ii)}:} By (\ref{xd_1}) we have that $(xd_i)(y) \beta_i \in Aut(G)$, for some $i \in \{1, 2\}$ and $(xd_j)(y) \beta_j \in Aut(G)$, for some $j \in \{3, 4\}$, by assumption. Also, $(xy)(d_1d_2)(d_3d_4)\alpha_3 \in Aut(G)$, where $(d_3, d_4)$ is a neighbour-non-neighbour pair of $\{d_1, d_2\}$. Without loss of generality, we can suppose that (\ref{d_1d_3}) is the case. This situation breaks down into the following cases:\\
		
		(A)(i) $(xd_1)(y)\beta_1, (xd_3)(y)\beta_3 \in Aut(G)$ and (ii) $(xd_1)(y)\beta_1, (xd_4)(y)\beta_4 \in Aut(G)$,\\
		
		(B)(i) $(xd_2)(y)\beta_2, (xd_3)(y)\beta_3 \in Aut(G)$ and (ii) $(xd_2)(y)\beta_2, (xd_4)(y)\beta_4 \in Aut(G)$,\\
		
		\noindent for some $\beta_1,\beta_2,\beta_3$ and $\beta_4$. Also,  $\{x, y, d_1, d_3\}$ is a determing set of $G$ as it contains $\{x, y\}$. We now show that it forms a distinguishing class  of $G$  using Lemma \ref{L3} and so $\rho(G) \leq 4$ for (A). \\
		
		For (A): Here, $(xd_1)(y)\beta_1 \in Aut(G)$. Now, $(xd_1)(y)\beta_1(\{y, d_1\}) = \{x, y\}$, so we get that
	\begin{align}\label{not_xy}
		\{x, y\} \not\in E(G),
	\end{align}
		as $\{y, d_1\} \not \in E(G)$. Next, we will show that,
		\begin{align}\label{not_xd_1}
		(xd_1)(yd_3) \gamma_1, (xd_3)(yd_1)\gamma_2 \not\in Aut(G),
		\end{align}
		 for any $\gamma_1,\gamma_2$. Suppose to get a contradiction that $(xd_3)(yd_1) \gamma_2 \in Aut(G)$, for some $\gamma_2$. Then, the action $(xd_3)(yd_1)\gamma_2(\{ d_1, d_3\}) = \{x, y\}$, but $\{d_1, d_3\} \in E(G)$ whilst $\{x, y\} \not\in E(G)$ and this is a contradiction. Similarly, assume to get a contradiction that $(xd_1)(yd_3) \gamma_1 \\ \in Aut(G)$. Then, $(xd_1)(yd_3)\gamma_1(\{ d_1, d_3\})  = \{x, y\}$ but $\{d_1, d_3\} \in E(G)$ whilst $\{x, y\} \not \in E(G)$ and this is a contradiction. \\
		
		For (A)(i): Since, $(xd_1)(y)\beta_1, (xd_3)(y)\beta_3 \in Aut(G)$ and $(xd_1)(yd_3) \gamma_1, (xd_3)(yd_1)\gamma_2 \not\in Aut(G)$, we can apply Lemma \ref{L4}(i), to get that $\{x, y, d_1, d_3\}$ is a distinguishing class of $G$ and so $\rho(G) \leq 4$.\\
		
		For (A)(ii): Here, $(xd_1)(y)\beta_1, (xd_4)(y)\beta_4 \in Aut(G)$. 
		Since $(xy)(d_1d_2)(d_3d_4)\alpha_3$ is equal to $(xy) (d_3d_4)(d_1d_2)\alpha_3$, 
		 using Proposition \ref{P10} on $(xd_4)(y)\beta_4 \in Aut(G)$, we have $(yd_3)(x)\beta_5 \in Aut (G)$. Here, we will apply Lemma \ref{L4} (ii) on $(xd_1)(y)\beta_1, (yd_3)(x)\beta_5 \in Aut(G)$. Therefore, we need to further show that,
		 \begin{align}\label{excluded}
		 (xd_1)(y)(d_3)\beta_6, \ (yd_3)(x)(d_1)\beta_7\not\in Aut(G).
		 \end{align}

	\noindent	Assume to get a contradiction that $(xd_1)(y)(d_3)\beta_6 \in Aut(G)$ for some $\beta_6$. Let
		\begin{align}
\varphi = (yd_3)(x) \beta_5 \cdot (xd_1)(y) (d_3)\beta_6\in Aut(G).
		\end{align}
		  Now, $\varphi(d_1) = x$ and $\varphi(d_3) = y$. Hence, the action $\varphi(\{d_1, d_3\}) = \{x, y\}$ but $\{d_1, d_3\} \in E(G)$ and $\{x, y\} \not\in E(G)$, which contradicts (\ref{d_1d_3}) and (\ref{not_xy}). Similarly, suppose to get a contradiction, that $(yd_3)(x)(d_1)\beta_7 \in Aut(G)$. Let $\varphi = (xd_1)(y)\beta_1 \cdot (yd_3)(x)(d_1)\beta_7$. Then, $\varphi(d_1) = x$ and $\varphi(d_3) = y$. Hence, $\varphi(\{d_1, d_3\}) = \{x, y\}$ but $\{d_1, d_3\} \in E(G)$ and $\{x, y\} \not\in E(G)$ which is a contradiction. Thus (\ref{not_xd_1}) and (\ref{excluded}) hold and we can apply Lemma \ref{L4} (ii) to get that $\{x, y, d_1, d_3\}$ is a distinguishing class of $G$ and so $\rho(G) \leq 4$.\\
		
		\noindent  Next for (B),  notice that $\{x, y, d_2, d_3\}$ is a determining set of $G$ as it contains $\{x, y\}$. We now show that it forms a distinguishing class of $G$  using Lemma \ref{L3} and so $\rho(G) \leq 4$. \\
		
		For (B): Here, $(xd_2)(y)\beta_2 \in Aut(G)$. Now, $(xd_2)(y)\beta_2(\{y, d_2\}) = \{x, y\}$, so we get that
		\begin{align}\label{xy}
			\{x, y\} \in E(G),
		\end{align}
		  as $\{y, d_2\}  \in E(G)$. Similarly to (\ref{not_xd_1}), we will show that 
		  \begin{align}\label{52}
(xd_2)(yd_3) \gamma_1, (xd_3)(yd_2)\gamma_2 \not\in Aut(G),
		  \end{align}
		  for all $\gamma_1,\gamma_2$. Suppose to get a contradiction that $(xd_2)(yd_3) \gamma_1 \in Aut(G)$. Then
		  \begin{align}
		  	(xd_2)(yd_3)\gamma_1(\{ d_2, d_3\}) = \{x, y\},
		  \end{align}
		  but $\{d_2, d_3\} \not\in E(G)$, which contradicts (\ref{xy}). Similarly, assume to get a contradiction that $(xd_3)(yd_2) \gamma_2 \in Aut(G)$. Then, $(xd_3)(yd_2)\gamma_2(\{ d_2, d_3\})  = \{x, y\}$, but $\{d_2, d_3\} \not\in E(G)$.\\
		
		 For (B)(i): Since, $(xd_2)(y)\beta_2, (xd_3)(y)\beta_3 \in Aut(G)$ and $(xd_2)(yd_3) \gamma_1, (xd_3)(yd_2)\gamma_2 \not\in Aut(G)$, we can apply Lemma \ref{L4}(i), to get that $\{x, y, d_2, d_3\}$ is a distinguishing class of $G$ and so $\rho(G) \leq 4$.\\
		
		For (B)(ii): Here, $(xd_2)(y)\beta_2, (xd_4)(y)\beta_4 \in Aut(G)$. Since $(xy)(d_1d_2)(d_3d_4)\alpha_3$ is equal to $(xy) (d_3d_4)(d_1d_2)\alpha_3 \in Aut(G)$, using Proposition \ref{P10} on $(xd_4)(y)\beta_4$, we have $(yd_3)(x)\beta_5 \in Aut (G)$. In this case, we apply Lemma \ref{L4} (ii) on $(xd_2)(y)\beta_1, (yd_3)(x)\beta_5 \in Aut(G)$. Hence, we need to further show that,
		\begin{align}\label{excluded2}
		(xd_2)(y)(d_3)\beta_6, \ (yd_3)(x)(d_2)\beta_7\not\in Aut(G).
		\end{align}

		Assume to get a contradiction that $(xd_2)(y)(d_3)\beta_6 \in Aut(G)$, for some $\beta_6$. Let
		\begin{align}
		\varphi = (yd_3)(x) \beta_5 \cdot (xd_2)(y) (d_3)\beta_6\in Aut(G)
.		\end{align}
	 Then, $\varphi(d_2) = x$ and $\varphi(d_3) = y$. Hence, $\varphi(\{d_2, d_3\}) = \{x, y\}$ but $\{d_2, d_3\} \not\in E(G)$ and $\{x, y\} \in E(G)$, which is a contradiction. Similarly, suppose to get a contradiction, that $(yd_3)(x)(d_2)\beta_7 \in Aut(G)$. Let $\varphi = (xd_2)(y)\beta_2 \cdot (yd_3)(x)(d_2)\beta_7$. Then, $\varphi(d_2) = x$ and $\varphi(d_3) = y$. Hence, $\varphi(\{d_2, d_3\}) = \{x, y\}$ but $\{d_2, d_3\} \not\in E(G)$ and $\{x, y\} \in E(G)$ which is a contradiction. Thus (\ref{52}) and (\ref{excluded2}) hold, so we can apply Lemma \ref{L4} (ii) to get that $\{x, y, d_2, d_3\}$ is a distinguishing class of $G$ and so $\rho(G) \leq 4$.\\

		\noindent \underline{\textsc{Case 2}(a):} $(xd_i)(y)\beta_i \not \in Aut(G)$ for $ i \in \{3, 4\}$. This is shown in Claim \ref{claim_rho<3}, using $(d_3, d_4)$ as the neighbour-non-neighbour pair of $\{x, y\}$ instead of $(d_1, d_2)$. \\
		
		\noindent\underline{\textsc{Case} 2(b):} $(xd_i)(y) \beta_i \in Aut(G)$, for some $i \in \{1, 2\}$ and $(xd_j)(y) \beta_j \in Aut(G)$, for some $j \in \{3, 4\}$. We also have that $(xy)(d_1d_2)(d_3d_4)\alpha_3 \in Aut(G)$ and $(d_3, d_4)$ is also a neighbour-non-neighbour pair of $\{x, y\}$. Suppose without loss of generality, that
		\begin{align}\label{xd_3}
			\{x, d_3\}, \{y, d_4\} \in E(G) \ \text{and} \  \{x, d_4\}, \{y, d_3\} \not\in E(G).
		\end{align}
		 The situation also breaks down into the following cases:\\

\vspace{-2mm}
        
		(A)(i) $(xd_1)(y)\beta_1, (xd_3)(y)\beta_3 \in Aut(G)$ and (ii) $(xd_1)(y)\beta_1, (xd_4)(y)\beta_4 \in Aut(G)$.
		
		(B)(i) $(xd_2)(y)\beta_2, (xd_3)(y)\beta_3 \in Aut(G)$ and (ii) $(xd_2)(y)\beta_2, (xd_4)(y)\beta_4 \in Aut(G)$.\\
		
		   For (A)(ii): We have that   $(xd_1)(y)\beta_1, (xd_4)(y)\beta_4 \in Aut(G)$.  Consider the action, \\ $(xd_1)(y)\beta_1 (\{d_1, y\}) = \{x, y\}$. Since $\{d_1, y\} \not \in E(G)$, by (\ref{xd_1}), then $\{x, y\} \not \in E(G)$. Now, it is also the case that $(xd_4)(y)\beta_4 (\{y, d_4\}) $ $= \{x, y\}$. However, $\{y, d_4\} \in E(G)$ by (\ref{xd_3}) and so this gives $\{x, y\} \in E(G)$. Therefore, it cannot be the case that both $(xd_1)(y)\beta_1 \in Aut(G)$ and $(xd_4)(y)\beta_4 \in Aut(G)$.
		
		Similarly, for (B)(i): We have that $(xd_2)(y)\beta_2, (xd_3)(y)\beta_3 \in Aut(G)$. Consider the image, $(xd_2)(y)\beta_2 (\{d_2, y\})$ $ = \{x, y\}$. As $\{d_2, y\}  \in E(G)$, then $\{x, y\}  \in E(G)$. Now, it is also the case that $(xd_3)(y)\beta_3 (\{y, d_3\})= \{x, y\}$. However, $\{y, d_3\} \not\in E(G)$ by (\ref{xd_3}) and so this gives $\{x, y\} \not\in E(G)$. Therefore, it cannot be the case that both $(xd_2)(y)\beta_2 \in Aut(G)$ and $(xd_3)(y)\beta_3 \in Aut(G)$.\\

		 For (A)(i): Here, $(xd_1)(y)\beta_1, (xd_3)(y)\beta_3 \in Aut(G)$. Just as in (\ref{not_xd_1}), we will show that
		\begin{align}
		(xd_1)(yd_3) \gamma_1, (xd_3)(yd_1)\gamma_2 \not\in Aut(G).
		\end{align}
		 Suppose the  contrary, that $(xd_1)(yd_3)\gamma_1 \in Aut(G)$. Consider  $(xd_1)(yd_3) \gamma_1(\{x, d_3\}) = \{d_1, y\}$. However, $\{x, d_3 \} \in E(G)$ but $\{d_1, y\} \not\in E(G)$ by (\ref{xd_3}) and (\ref{xd_1}), which is a contradiction. Similarly, suppose that $(xd_3)(yd_1)\gamma_2 \in Aut(G)$. Consider, $(xd_3)(yd_1)\gamma_2 (\{y, d_3\}) = \{d_1, x\}$. However, $\{y, d_3\} \not \in E(G)$ but $\{d_1, x\} \in E(G)$,  by (\ref{xd_3}) and (\ref{xd_1}) which is a contradiction. Therefore,  $(xd_1)(yd_3) \gamma_1, (xd_3)(yd_1)\gamma_2 \not\in Aut(G)$  and we can use Lemma \ref{L4} (i), to give us that $\{x, y, d_1, d_3\}$ is a distinguishing class of $G$ and $\rho(G) \leq 4$. \\   
		
		  For (B)(ii):  Here, $(xd_2)(y)\beta_2, (xd_4)(y)\beta_4 \in Aut(G)$. We will show that,
		\begin{align}
		(xd_2)(yd_4) \gamma_1, (xd_4)(yd_2)\gamma_2 \not\in Aut(G).
		\end{align}
		 Suppose the contrary, that $(xd_2)(yd_4)\gamma_1 \in Aut(G)$. Consider  $(xd_2)(yd_4) \gamma_1(\{x, d_4\}) = \{d_2, y\}$. However, $\{x, d_4 \} \not\in E(G)$ but $\{d_2, y\} \in E(G)$,  by (\ref{xd_3}) and (\ref{xd_1}) which is a contradiction.  Similarly, suppose that $(xd_4)(yd_2)\gamma_2 \in Aut(G)$. Consider, $(xd_4)(yd_2)\gamma_2 (\{y, d_4\}) = \{d_2, x\}$. However, $\{y, d_4\}  \in E(G)$ but $\{d_2, x\} \not\in E(G)$,  by (\ref{xd_3}) and (\ref{xd_1}), which is a contradiction. Therefore,  $(xd_2)(yd_4) \gamma_1, (xd_4)(yd_2)\gamma_2 \not\in Aut(G)$  and we can use Lemma \ref{L4} (i), to give us that $\{x, y, d_2, d_4\}$ is a distinguishing class of $G$ and $\rho(G) \leq 4$.
	\end{proof}

       \section {Bounds for graphs with Determining number greater than 2}

        Despite the main result,  that the costs of graphs with determining number 2 cannot be arbitrarily high, it is not known for graphs with higher determining numbers if the costs are bounded. \textit{If} these costs are bounded, then the following class of examples gives lower bounds on what some of these bounds can be:

\begin{example} \label{E4.1}

Consider  cliques $K_{n}$ with paths of $\lceil log_2{n} \rceil$ nodes attached separately to each node which is the minimum number of nodes needed to have $n$ different 0-1 strings. Replace each node of the clique with a pair of non-adjacent nodes, called clones, each having the same neighbours as the node that it replaces.  Call the resulting graphs, $H_n$ (see $H_8$ in Figure \ref{Fig_3}).

 \vspace{-5mm}

\begin{figure}[H]
\centering

\tikzset{every picture/.style={line width=0.75pt}} %set default line width to 0.75pt        

\begin{tikzpicture}[x=0.75pt,y=0.75pt,yscale=-0.85,xscale=0.85]
%uncomment if require: \path (0,580); %set diagram left start at 0, and has height of 580

%Shape: Rectangle [id:dp2854129207075776] 
\draw   (181.65,175.68) -- (294.86,175.68) -- (294.86,285.5) -- (181.65,285.5) -- cycle ;
%Shape: Rectangle [id:dp6805770101872163] 
\draw   (238.25,152.93) -- (318.28,230.59) -- (238.25,308.24) -- (158.23,230.59) -- cycle ;
%Straight Lines [id:da27538762361561875] 
\draw    (119.15,114.94) ;
\draw [shift={(119.15,114.94)}, rotate = 0] [color={rgb, 255:red, 0; green, 0; blue, 0 }  ][fill={rgb, 255:red, 0; green, 0; blue, 0 }  ][line width=0.75]      (0, 0) circle [x radius= 2.01, y radius= 2.01]   ;
%Straight Lines [id:da7812795649906844] 
\draw [line width=0.75]    (158.21,230.59) -- (318.3,230.59) ;
%Straight Lines [id:da6679103884349563] 
\draw [line width=0.75]    (181.65,285.5) -- (294.86,175.68) ;
%Straight Lines [id:da5827818389970423] 
\draw [line width=0.75]    (181.65,175.68) -- (294.86,285.5) ;
%Straight Lines [id:da12159777637418157] 
\draw [line width=0.75]    (238.25,308.24) -- (238.25,224.58) -- (238.25,152.93) ;
%Straight Lines [id:da02275174629554444] 
\draw [line width=0.75]    (158.21,230.59) -- (294.86,285.5) ;
%Straight Lines [id:da5521616083554041] 
\draw [line width=0.75]    (181.65,175.68) -- (318.3,230.59) ;
%Straight Lines [id:da7027982950135784] 
\draw [line width=0.75]    (158.21,230.59) -- (294.86,175.68) ;
%Straight Lines [id:da9725357114251331] 
\draw [line width=0.75]    (181.65,285.5) -- (318.3,230.59) ;
%Straight Lines [id:da08068916537562809] 
\draw [line width=0.75]    (181.65,175.68) -- (238.25,308.24) ;
%Straight Lines [id:da7903143745891038] 
\draw [line width=0.75]    (238.25,308.24) -- (294.86,175.68) ;
%Straight Lines [id:da6103684562824998] 
\draw [line width=0.75]    (181.65,285.5) -- (238.25,152.93) ;
%Straight Lines [id:da8927643753886191] 
\draw [line width=0.75]    (238.25,308.24) -- (294.86,285.5) ;
%Straight Lines [id:da3582661318449972] 
\draw [line width=0.75]    (294.86,285.5) -- (318.3,230.59) ;
%Straight Lines [id:da30275499482232315] 
\draw [line width=0.75]    (318.3,230.59) -- (294.86,175.68) ;
%Straight Lines [id:da3485305213087573] 
\draw    (238.25,152.93) -- (294.86,175.68) ;
%Straight Lines [id:da7252857904275093] 
\draw    (181.65,175.68) -- (238.25,152.93) ;
%Straight Lines [id:da36620006213110656] 
\draw    (158.21,230.59) -- (181.65,175.68) ;
%Straight Lines [id:da4159528344549246] 
\draw [line width=0.75]    (158.21,230.59) -- (181.65,285.5) ;
%Straight Lines [id:da6246576014432694] 
\draw [line width=0.75]    (181.65,285.5) -- (238.25,308.24) ;
%Shape: Ellipse [id:dp5303819519896006] 
\draw  [fill={rgb, 255:red, 255; green, 255; blue, 255 }  ,fill opacity=1 ] (138.54,230.59) .. controls (138.54,224.22) and (147.35,219.05) .. (158.21,219.05) .. controls (169.07,219.05) and (177.88,224.22) .. (177.88,230.59) .. controls (177.88,236.96) and (169.07,242.13) .. (158.21,242.13) .. controls (147.35,242.13) and (138.54,236.96) .. (138.54,230.59) -- cycle ;
%Straight Lines [id:da9388086244983933] 
\draw    (148.38,228.42) ;
\draw [shift={(148.38,228.42)}, rotate = 0] [color={rgb, 255:red, 0; green, 0; blue, 0 }  ][fill={rgb, 255:red, 0; green, 0; blue, 0 }  ][line width=0.75]      (0, 0) circle [x radius= 2.01, y radius= 2.01]   ;
%Straight Lines [id:da4510128716224492] 
\draw    (166.57,228.35) ;
\draw [shift={(166.57,228.35)}, rotate = 0] [color={rgb, 255:red, 0; green, 0; blue, 0 }  ][fill={rgb, 255:red, 0; green, 0; blue, 0 }  ][line width=0.75]      (0, 0) circle [x radius= 2.01, y radius= 2.01]   ;
%Shape: Ellipse [id:dp645603349040602] 
\draw  [fill={rgb, 255:red, 255; green, 255; blue, 255 }  ,fill opacity=1 ] (161.99,285.5) .. controls (161.99,279.13) and (170.79,273.96) .. (181.65,273.96) .. controls (192.52,273.96) and (201.32,279.13) .. (201.32,285.5) .. controls (201.32,291.87) and (192.52,297.04) .. (181.65,297.04) .. controls (170.79,297.04) and (161.99,291.87) .. (161.99,285.5) -- cycle ;
%Shape: Ellipse [id:dp042707132643018286] 
\draw  [fill={rgb, 255:red, 255; green, 255; blue, 255 }  ,fill opacity=1 ] (218.59,308.24) .. controls (218.59,301.87) and (227.39,296.7) .. (238.25,296.7) .. controls (249.12,296.7) and (257.92,301.87) .. (257.92,308.24) .. controls (257.92,314.62) and (249.12,319.79) .. (238.25,319.79) .. controls (227.39,319.79) and (218.59,314.62) .. (218.59,308.24) -- cycle ;
%Shape: Ellipse [id:dp5600339259131218] 
\draw  [fill={rgb, 255:red, 255; green, 255; blue, 255 }  ,fill opacity=1 ] (275.19,285.5) .. controls (275.19,279.13) and (283.99,273.96) .. (294.86,273.96) .. controls (305.72,273.96) and (314.52,279.13) .. (314.52,285.5) .. controls (314.52,291.87) and (305.72,297.04) .. (294.86,297.04) .. controls (283.99,297.04) and (275.19,291.87) .. (275.19,285.5) -- cycle ;
%Shape: Ellipse [id:dp10509332979951991] 
\draw  [fill={rgb, 255:red, 255; green, 255; blue, 255 }  ,fill opacity=1 ] (298.63,230.59) .. controls (298.63,224.22) and (307.44,219.05) .. (318.3,219.05) .. controls (329.16,219.05) and (337.97,224.22) .. (337.97,230.59) .. controls (337.97,236.96) and (329.16,242.13) .. (318.3,242.13) .. controls (307.44,242.13) and (298.63,236.96) .. (298.63,230.59) -- cycle ;
%Shape: Ellipse [id:dp09759377858706775] 
\draw  [fill={rgb, 255:red, 255; green, 255; blue, 255 }  ,fill opacity=1 ] (275.19,175.68) .. controls (275.19,169.3) and (283.99,164.14) .. (294.86,164.14) .. controls (305.72,164.14) and (314.52,169.3) .. (314.52,175.68) .. controls (314.52,182.05) and (305.72,187.22) .. (294.86,187.22) .. controls (283.99,187.22) and (275.19,182.05) .. (275.19,175.68) -- cycle ;
%Shape: Ellipse [id:dp39543734354845494] 
\draw  [fill={rgb, 255:red, 255; green, 255; blue, 255 }  ,fill opacity=1 ] (218.59,152.93) .. controls (218.59,146.56) and (227.39,141.39) .. (238.25,141.39) .. controls (249.12,141.39) and (257.92,146.56) .. (257.92,152.93) .. controls (257.92,159.31) and (249.12,164.47) .. (238.25,164.47) .. controls (227.39,164.47) and (218.59,159.31) .. (218.59,152.93) -- cycle ;
%Shape: Ellipse [id:dp26011444032985886] 
\draw  [fill={rgb, 255:red, 255; green, 255; blue, 255 }  ,fill opacity=1 ] (161.99,175.68) .. controls (161.99,169.3) and (170.79,164.14) .. (181.65,164.14) .. controls (192.52,164.14) and (201.32,169.3) .. (201.32,175.68) .. controls (201.32,182.05) and (192.52,187.22) .. (181.65,187.22) .. controls (170.79,187.22) and (161.99,182.05) .. (161.99,175.68) -- cycle ;
%Straight Lines [id:da7746598215744226] 
\draw    (173.65,282.9) ;
\draw [shift={(173.65,282.9)}, rotate = 0] [color={rgb, 255:red, 0; green, 0; blue, 0 }  ][fill={rgb, 255:red, 0; green, 0; blue, 0 }  ][line width=0.75]      (0, 0) circle [x radius= 2.01, y radius= 2.01]   ;
%Straight Lines [id:da05706585610336745] 
\draw    (191.46,282.8) ;
\draw [shift={(191.46,282.8)}, rotate = 0] [color={rgb, 255:red, 0; green, 0; blue, 0 }  ][fill={rgb, 255:red, 0; green, 0; blue, 0 }  ][line width=0.75]      (0, 0) circle [x radius= 2.01, y radius= 2.01]   ;
%Straight Lines [id:da364546679032585] 
\draw    (228.58,305.52) ;
\draw [shift={(228.58,305.52)}, rotate = 0] [color={rgb, 255:red, 0; green, 0; blue, 0 }  ][fill={rgb, 255:red, 0; green, 0; blue, 0 }  ][line width=0.75]      (0, 0) circle [x radius= 2.01, y radius= 2.01]   ;
%Straight Lines [id:da04437920939122886] 
\draw    (249.15,305.84) ;
\draw [shift={(249.15,305.84)}, rotate = 0] [color={rgb, 255:red, 0; green, 0; blue, 0 }  ][fill={rgb, 255:red, 0; green, 0; blue, 0 }  ][line width=0.75]      (0, 0) circle [x radius= 2.01, y radius= 2.01]   ;
%Straight Lines [id:da13408071288702172] 
\draw    (173.2,173.08) ;
\draw [shift={(173.2,173.08)}, rotate = 0] [color={rgb, 255:red, 0; green, 0; blue, 0 }  ][fill={rgb, 255:red, 0; green, 0; blue, 0 }  ][line width=0.75]      (0, 0) circle [x radius= 2.01, y radius= 2.01]   ;
%Straight Lines [id:da7791170663916364] 
\draw    (192.19,173.01) ;
\draw [shift={(192.19,173.01)}, rotate = 0] [color={rgb, 255:red, 0; green, 0; blue, 0 }  ][fill={rgb, 255:red, 0; green, 0; blue, 0 }  ][line width=0.75]      (0, 0) circle [x radius= 2.01, y radius= 2.01]   ;
%Straight Lines [id:da2933232037590814] 
\draw    (285.5,283.22) ;
\draw [shift={(285.5,283.22)}, rotate = 0] [color={rgb, 255:red, 0; green, 0; blue, 0 }  ][fill={rgb, 255:red, 0; green, 0; blue, 0 }  ][line width=0.75]      (0, 0) circle [x radius= 2.01, y radius= 2.01]   ;
%Straight Lines [id:da10665603254415723] 
\draw    (305.28,283.54) ;
\draw [shift={(305.28,283.54)}, rotate = 0] [color={rgb, 255:red, 0; green, 0; blue, 0 }  ][fill={rgb, 255:red, 0; green, 0; blue, 0 }  ][line width=0.75]      (0, 0) circle [x radius= 2.01, y radius= 2.01]   ;
%Straight Lines [id:da521027083546853] 
\draw    (310.98,228.29) ;
\draw [shift={(310.98,228.29)}, rotate = 0] [color={rgb, 255:red, 0; green, 0; blue, 0 }  ][fill={rgb, 255:red, 0; green, 0; blue, 0 }  ][line width=0.75]      (0, 0) circle [x radius= 2.01, y radius= 2.01]   ;
%Straight Lines [id:da7634410271936941] 
\draw    (328.78,228.21) ;
\draw [shift={(328.78,228.21)}, rotate = 0] [color={rgb, 255:red, 0; green, 0; blue, 0 }  ][fill={rgb, 255:red, 0; green, 0; blue, 0 }  ][line width=0.75]      (0, 0) circle [x radius= 2.01, y radius= 2.01]   ;
%Straight Lines [id:da829332863125407] 
\draw    (286.01,173.19) ;
\draw [shift={(286.01,173.19)}, rotate = 0] [color={rgb, 255:red, 0; green, 0; blue, 0 }  ][fill={rgb, 255:red, 0; green, 0; blue, 0 }  ][line width=0.75]      (0, 0) circle [x radius= 2.01, y radius= 2.01]   ;
%Straight Lines [id:da8684249665530791] 
\draw    (303.8,173.12) ;
\draw [shift={(303.8,173.12)}, rotate = 0] [color={rgb, 255:red, 0; green, 0; blue, 0 }  ][fill={rgb, 255:red, 0; green, 0; blue, 0 }  ][line width=0.75]      (0, 0) circle [x radius= 2.01, y radius= 2.01]   ;
%Straight Lines [id:da00696626706244341] 
\draw    (230.8,149.73) ;
\draw [shift={(230.8,149.73)}, rotate = 0] [color={rgb, 255:red, 0; green, 0; blue, 0 }  ][fill={rgb, 255:red, 0; green, 0; blue, 0 }  ][line width=0.75]      (0, 0) circle [x radius= 2.01, y radius= 2.01]   ;
%Straight Lines [id:da6902427054579507] 
\draw    (248.39,149.65) ;
\draw [shift={(248.39,149.65)}, rotate = 0] [color={rgb, 255:red, 0; green, 0; blue, 0 }  ][fill={rgb, 255:red, 0; green, 0; blue, 0 }  ][line width=0.75]      (0, 0) circle [x radius= 2.01, y radius= 2.01]   ;
%Straight Lines [id:da6701010037195477] 
\draw [line width=0.75]    (171.94,165.96) -- (168.87,163.02) -- (145.54,140.45) ;
%Straight Lines [id:da3662187036843393] 
\draw    (92.75,89.42) ;
\draw [shift={(92.75,89.42)}, rotate = 0] [color={rgb, 255:red, 0; green, 0; blue, 0 }  ][fill={rgb, 255:red, 0; green, 0; blue, 0 }  ][line width=0.75]      (0, 0) circle [x radius= 2.01, y radius= 2.01]   ;
%Straight Lines [id:da24734408705166178] 
\draw    (145.54,140.45) ;
\draw [shift={(145.54,140.45)}, rotate = 0] [color={rgb, 255:red, 0; green, 0; blue, 0 }  ][fill={rgb, 255:red, 0; green, 0; blue, 0 }  ][line width=0.75]      (0, 0) circle [x radius= 2.01, y radius= 2.01]   ;
%Straight Lines [id:da6827179975052922] 
\draw    (144.54,320.04) ;
\draw [shift={(144.54,320.04)}, rotate = 0] [color={rgb, 255:red, 0; green, 0; blue, 0 }  ][fill={rgb, 255:red, 0; green, 0; blue, 0 }  ][line width=0.75]      (0, 0) circle [x radius= 2.01, y radius= 2.01]   ;
%Straight Lines [id:da34678637271284507] 
\draw    (357.28,346.76) ;
\draw [shift={(357.28,346.76)}, rotate = 0] [color={rgb, 255:red, 0; green, 0; blue, 0 }  ][fill={rgb, 255:red, 0; green, 0; blue, 0 }  ][line width=0.75]      (0, 0) circle [x radius= 2.01, y radius= 2.01]   ;
%Straight Lines [id:da35149132581818243] 
\draw    (383.68,372.27) ;
\draw [shift={(383.68,372.27)}, rotate = 0] [color={rgb, 255:red, 0; green, 0; blue, 0 }  ][fill={rgb, 255:red, 0; green, 0; blue, 0 }  ][line width=0.75]      (0, 0) circle [x radius= 2.01, y radius= 2.01]   ;
%Straight Lines [id:da8612346828557742] 
\draw    (330.89,321.25) ;
\draw [shift={(330.89,321.25)}, rotate = 0] [color={rgb, 255:red, 0; green, 0; blue, 0 }  ][fill={rgb, 255:red, 0; green, 0; blue, 0 }  ][line width=0.75]      (0, 0) circle [x radius= 2.01, y radius= 2.01]   ;
%Straight Lines [id:da6795310376459227] 
\draw [line width=0.75]    (238.25,319.79) -- (238.25,337.93) -- (238.25,354.89) ;
%Straight Lines [id:da46042301431829424] 
\draw [line width=0.75]    (168.52,220.82) -- (227.75,163.15) ;
%Straight Lines [id:da4207730495676699] 
\draw [line width=0.75]    (248.75,163.11) -- (307.81,220.41) ;
%Straight Lines [id:da7645967605480637] 
\draw [line width=0.75]    (248.76,298.38) -- (307.78,241.12) ;
%Straight Lines [id:da358750290734681] 
\draw [line width=0.75]    (227.45,298.45) -- (168.51,241.27) ;
%Straight Lines [id:da4663709317988425] 
\draw [line width=0.75]    (294.86,273.96) -- (294.86,187.49) ;
%Straight Lines [id:da9466431974449285] 
\draw [line width=0.75]    (181.65,273.96) -- (181.65,187.49) ;
%Straight Lines [id:da11829115724449746] 
\draw [line width=0.75]    (162.82,219.15) -- (176.27,187.31) ;
%Straight Lines [id:da07121424140272203] 
\draw [line width=0.75]    (198.4,169.34) -- (221.53,159.35) ;
%Straight Lines [id:da69419228073856] 
\draw [line width=0.75]    (254.69,159.29) -- (278.01,169.18) ;
%Straight Lines [id:da5977448647115483] 
\draw    (238.25,354.89) ;
\draw [shift={(238.25,354.89)}, rotate = 0] [color={rgb, 255:red, 0; green, 0; blue, 0 }  ][fill={rgb, 255:red, 0; green, 0; blue, 0 }  ][line width=0.75]      (0, 0) circle [x radius= 2.01, y radius= 2.01]   ;
%Straight Lines [id:da38894640854651774] 
\draw    (238.25,425.11) ;
\draw [shift={(238.25,425.11)}, rotate = 0] [color={rgb, 255:red, 0; green, 0; blue, 0 }  ][fill={rgb, 255:red, 0; green, 0; blue, 0 }  ][line width=0.75]      (0, 0) circle [x radius= 2.01, y radius= 2.01]   ;
%Shape: Ellipse [id:dp3155961992667978] 
\draw   (500.77,98.08) .. controls (500.77,94.1) and (506.5,90.88) .. (513.58,90.88) .. controls (520.65,90.88) and (526.39,94.1) .. (526.39,98.08) .. controls (526.39,102.06) and (520.65,105.28) .. (513.58,105.28) .. controls (506.5,105.28) and (500.77,102.06) .. (500.77,98.08) -- cycle ;
%Straight Lines [id:da8246323174168798] 
\draw [color={rgb, 255:red, 74; green, 144; blue, 226 }  ,draw opacity=1 ]   (508.19,97.84) ;
\draw [shift={(508.19,97.84)}, rotate = 0] [color={rgb, 255:red, 74; green, 144; blue, 226 }  ,draw opacity=1 ][fill={rgb, 255:red, 74; green, 144; blue, 226 }  ,fill opacity=1 ][line width=0.75]      (0, 0) circle [x radius= 2.01, y radius= 2.01]   ;
%Straight Lines [id:da4866467029578485] 
\draw [color={rgb, 255:red, 74; green, 144; blue, 226 }  ,draw opacity=1 ]   (519.43,97.52) ;
\draw [shift={(519.43,97.52)}, rotate = 0] [color={rgb, 255:red, 74; green, 144; blue, 226 }  ,draw opacity=1 ][fill={rgb, 255:red, 74; green, 144; blue, 226 }  ,fill opacity=1 ][line width=0.75]      (0, 0) circle [x radius= 2.01, y radius= 2.01]   ;
%Straight Lines [id:da5826064297355191] 
\draw [line width=0.75]    (526.1,98.08) -- (545.35,98.08) ;
%Shape: Ellipse [id:dp4916563006425245] 
\draw   (545.35,98.08) .. controls (545.35,94.1) and (551.09,90.88) .. (558.16,90.88) .. controls (565.24,90.88) and (570.97,94.1) .. (570.97,98.08) .. controls (570.97,102.06) and (565.24,105.28) .. (558.16,105.28) .. controls (551.09,105.28) and (545.35,102.06) .. (545.35,98.08) -- cycle ;
%Straight Lines [id:da4889572967545258] 
\draw    (552.25,97.52) ;
\draw [shift={(552.25,97.52)}, rotate = 0] [color={rgb, 255:red, 0; green, 0; blue, 0 }  ][fill={rgb, 255:red, 0; green, 0; blue, 0 }  ][line width=0.75]      (0, 0) circle [x radius= 2.01, y radius= 2.01]   ;
%Straight Lines [id:da16894345200749494] 
\draw    (562.9,97.52) ;
\draw [shift={(562.9,97.52)}, rotate = 0] [color={rgb, 255:red, 0; green, 0; blue, 0 }  ][fill={rgb, 255:red, 0; green, 0; blue, 0 }  ][line width=0.75]      (0, 0) circle [x radius= 2.01, y radius= 2.01]   ;
%Shape: Ellipse [id:dp6688679377669448] 
\draw   (545.32,62.62) .. controls (545.32,58.64) and (551.06,55.42) .. (558.13,55.42) .. controls (565.21,55.42) and (570.94,58.64) .. (570.94,62.62) .. controls (570.94,66.59) and (565.21,69.82) .. (558.13,69.82) .. controls (551.06,69.82) and (545.32,66.59) .. (545.32,62.62) -- cycle ;
%Straight Lines [id:da6224604302774439] 
\draw    (552.78,62.06) ;
\draw [shift={(552.78,62.06)}, rotate = 0] [color={rgb, 255:red, 0; green, 0; blue, 0 }  ][fill={rgb, 255:red, 0; green, 0; blue, 0 }  ][line width=0.75]      (0, 0) circle [x radius= 2.01, y radius= 2.01]   ;
%Straight Lines [id:da6978792254045544] 
\draw    (563.99,62.06) ;
\draw [shift={(563.99,62.06)}, rotate = 0] [color={rgb, 255:red, 0; green, 0; blue, 0 }  ][fill={rgb, 255:red, 0; green, 0; blue, 0 }  ][line width=0.75]      (0, 0) circle [x radius= 2.01, y radius= 2.01]   ;
%Shape: Ellipse [id:dp19610544162562682] 
\draw   (504.34,134.74) .. controls (504.34,130.77) and (510.07,127.54) .. (517.15,127.54) .. controls (524.22,127.54) and (529.96,130.77) .. (529.96,134.74) .. controls (529.96,138.72) and (524.22,141.94) .. (517.15,141.94) .. controls (510.07,141.94) and (504.34,138.72) .. (504.34,134.74) -- cycle ;
%Straight Lines [id:da7718668010764029] 
\draw [color={rgb, 255:red, 74; green, 144; blue, 226 }  ,draw opacity=1 ]   (511.79,134.18) ;
\draw [shift={(511.79,134.18)}, rotate = 0] [color={rgb, 255:red, 74; green, 144; blue, 226 }  ,draw opacity=1 ][fill={rgb, 255:red, 74; green, 144; blue, 226 }  ,fill opacity=1 ][line width=0.75]      (0, 0) circle [x radius= 2.01, y radius= 2.01]   ;
%Straight Lines [id:da45423992120004686] 
\draw [color={rgb, 255:red, 74; green, 144; blue, 226 }  ,draw opacity=1 ]   (523,134.18) ;
\draw [shift={(523,134.18)}, rotate = 0] [color={rgb, 255:red, 74; green, 144; blue, 226 }  ,draw opacity=1 ][fill={rgb, 255:red, 74; green, 144; blue, 226 }  ,fill opacity=1 ][line width=0.75]      (0, 0) circle [x radius= 2.01, y radius= 2.01]   ;
%Straight Lines [id:da5979295969625684] 
\draw [line width=0.75]    (529.96,134.74) -- (557.59,134.74) ;
\draw [shift={(557.59,134.74)}, rotate = 0] [color={rgb, 255:red, 0; green, 0; blue, 0 }  ][fill={rgb, 255:red, 0; green, 0; blue, 0 }  ][line width=0.75]      (0, 0) circle [x radius= 2.01, y radius= 2.01]   ;
%Straight Lines [id:da8086645848918403] 
\draw [line width=0.75]    (655.11,132.14) -- (633.3,141.78) ;
\draw [shift={(655.11,132.14)}, rotate = 156.16] [color={rgb, 255:red, 0; green, 0; blue, 0 }  ][fill={rgb, 255:red, 0; green, 0; blue, 0 }  ][line width=0.75]      (0, 0) circle [x radius= 1.34, y radius= 1.34]   ;
%Straight Lines [id:da24178166911418197] 
\draw [line width=0.75]    (655.11,132.14) -- (633.3,124.98) ;
\draw [shift={(655.11,132.14)}, rotate = 198.18] [color={rgb, 255:red, 0; green, 0; blue, 0 }  ][fill={rgb, 255:red, 0; green, 0; blue, 0 }  ][line width=0.75]      (0, 0) circle [x radius= 2.01, y radius= 2.01]   ;
%Straight Lines [id:da5120505124715823] 
\draw    (683.4,106.12) ;
\draw [shift={(683.4,106.12)}, rotate = 0] [color={rgb, 255:red, 0; green, 0; blue, 0 }  ][fill={rgb, 255:red, 0; green, 0; blue, 0 }  ][line width=0.75]      (0, 0) circle [x radius= 2.01, y radius= 2.01]   ;
%Straight Lines [id:da1184456333538435] 
\draw    (238.25,390) ;
\draw [shift={(238.25,390)}, rotate = 0] [color={rgb, 255:red, 0; green, 0; blue, 0 }  ][fill={rgb, 255:red, 0; green, 0; blue, 0 }  ][line width=0.75]      (0, 0) circle [x radius= 2.01, y radius= 2.01]   ;
%Straight Lines [id:da23795096600038246] 
\draw    (635.19,89.32) -- (682.9,89.32) ;
%Straight Lines [id:da40756102057463794] 
\draw    (635.19,106.12) -- (682.9,106.12) ;
%Straight Lines [id:da8589099670659953] 
\draw    (635.19,89.32) -- (683.4,106.12) ;
%Straight Lines [id:da31940308641491977] 
\draw    (635.19,106.12) -- (682.9,89.32) ;
\draw [shift={(682.9,89.32)}, rotate = 340.6] [color={rgb, 255:red, 0; green, 0; blue, 0 }  ][fill={rgb, 255:red, 0; green, 0; blue, 0 }  ][line width=0.75]      (0, 0) circle [x radius= 2.01, y radius= 2.01]   ;
%Straight Lines [id:da8432805583525247] 
\draw    (635.22,61.11) ;
\draw [shift={(635.22,61.11)}, rotate = 0] [color={rgb, 255:red, 0; green, 0; blue, 0 }  ][fill={rgb, 255:red, 0; green, 0; blue, 0 }  ][line width=0.75]      (0, 0) circle [x radius= 2.01, y radius= 2.01]   ;
%Straight Lines [id:da5385853945930996] 
\draw    (649.22,61.47) ;
\draw [shift={(649.22,61.47)}, rotate = 0] [color={rgb, 255:red, 0; green, 0; blue, 0 }  ][fill={rgb, 255:red, 0; green, 0; blue, 0 }  ][line width=0.75]      (0, 0) circle [x radius= 2.01, y radius= 2.01]   ;
%Straight Lines [id:da15591019921249816] 
\draw [line width=0.75]    (238.25,354.89) -- (238.25,390) ;
%Straight Lines [id:da12861653039150833] 
\draw [line width=0.75]    (238.25,390) -- (238.25,425.11) ;
%Straight Lines [id:da12655121071018127] 
\draw [line width=0.75]    (144.54,320.04) -- (118.51,345.02) ;
%Straight Lines [id:da21747892137328217] 
\draw [line width=0.75]    (118.51,345.02) -- (92.48,370) ;
%Straight Lines [id:da6629454464485709] 
\draw    (92.48,370) ;
\draw [shift={(92.48,370)}, rotate = 0] [color={rgb, 255:red, 0; green, 0; blue, 0 }  ][fill={rgb, 255:red, 0; green, 0; blue, 0 }  ][line width=0.75]      (0, 0) circle [x radius= 2.01, y radius= 2.01]   ;
%Straight Lines [id:da2636077063952388] 
\draw    (118.51,345.02) ;
\draw [shift={(118.51,345.02)}, rotate = 0] [color={rgb, 255:red, 0; green, 0; blue, 0 }  ][fill={rgb, 255:red, 0; green, 0; blue, 0 }  ][line width=0.75]      (0, 0) circle [x radius= 2.01, y radius= 2.01]   ;
%Straight Lines [id:da960441732809249] 
\draw [line width=0.75]    (330.83,141) -- (304.8,165.98) ;
%Straight Lines [id:da17901248630655509] 
\draw [line width=0.75]    (170.57,295.06) -- (144.54,320.04) ;
%Straight Lines [id:da3211883166560696] 
\draw [line width=0.75]    (356.86,116.02) -- (330.83,141) ;
%Straight Lines [id:da271785320667985] 
\draw [line width=0.75]    (382.89,91.03) -- (356.86,116.02) ;
%Straight Lines [id:da02389824305382382] 
\draw    (356.86,116.02) ;
\draw [shift={(356.86,116.02)}, rotate = 0] [color={rgb, 255:red, 0; green, 0; blue, 0 }  ][fill={rgb, 255:red, 0; green, 0; blue, 0 }  ][line width=0.75]      (0, 0) circle [x radius= 2.01, y radius= 2.01]   ;
%Straight Lines [id:da5062827958494216] 
\draw    (382.89,91.03) ;
\draw [shift={(382.89,91.03)}, rotate = 0] [color={rgb, 255:red, 0; green, 0; blue, 0 }  ][fill={rgb, 255:red, 0; green, 0; blue, 0 }  ][line width=0.75]      (0, 0) circle [x radius= 2.01, y radius= 2.01]   ;
%Straight Lines [id:da3031013828664564] 
\draw    (330.83,141) ;
\draw [shift={(330.83,141)}, rotate = 0] [color={rgb, 255:red, 0; green, 0; blue, 0 }  ][fill={rgb, 255:red, 0; green, 0; blue, 0 }  ][line width=0.75]      (0, 0) circle [x radius= 2.01, y radius= 2.01]   ;
%Straight Lines [id:da06396218159458544] 
\draw    (337.97,230.59) -- (377.38,230.59) ;
%Straight Lines [id:da9256822311632562] 
\draw    (377.38,230.59) -- (416.79,230.59) ;
%Straight Lines [id:da9420714504560149] 
\draw    (416.79,230.59) -- (456.21,230.59) ;
%Straight Lines [id:da1261795230491034] 
\draw    (377.38,230.59) ;
\draw [shift={(377.38,230.59)}, rotate = 0] [color={rgb, 255:red, 0; green, 0; blue, 0 }  ][fill={rgb, 255:red, 0; green, 0; blue, 0 }  ][line width=0.75]      (0, 0) circle [x radius= 2.01, y radius= 2.01]   ;
%Straight Lines [id:da20248218592957046] 
\draw    (456.21,230.59) ;
\draw [shift={(456.21,230.59)}, rotate = 0] [color={rgb, 255:red, 0; green, 0; blue, 0 }  ][fill={rgb, 255:red, 0; green, 0; blue, 0 }  ][line width=0.75]      (0, 0) circle [x radius= 2.01, y radius= 2.01]   ;
%Straight Lines [id:da6554837041466355] 
\draw    (416.79,230.59) ;
\draw [shift={(416.79,230.59)}, rotate = 0] [color={rgb, 255:red, 0; green, 0; blue, 0 }  ][fill={rgb, 255:red, 0; green, 0; blue, 0 }  ][line width=0.75]      (0, 0) circle [x radius= 2.01, y radius= 2.01]   ;
%Straight Lines [id:da25871797346945047] 
\draw    (99.13,230.59) -- (138.54,230.59) ;
%Straight Lines [id:da6650137826497997] 
\draw    (59.72,230.59) -- (99.13,230.59) ;
%Straight Lines [id:da7476669318962343] 
\draw    (20.3,230.59) -- (59.72,230.59) ;
%Straight Lines [id:da16894250185958692] 
\draw    (99.13,230.59) ;
\draw [shift={(99.13,230.59)}, rotate = 0] [color={rgb, 255:red, 0; green, 0; blue, 0 }  ][fill={rgb, 255:red, 0; green, 0; blue, 0 }  ][line width=0.75]      (0, 0) circle [x radius= 2.01, y radius= 2.01]   ;
%Straight Lines [id:da2924250526309773] 
\draw    (20.3,230.59) ;
\draw [shift={(20.3,230.59)}, rotate = 0] [color={rgb, 255:red, 0; green, 0; blue, 0 }  ][fill={rgb, 255:red, 0; green, 0; blue, 0 }  ][line width=0.75]      (0, 0) circle [x radius= 2.01, y radius= 2.01]   ;
%Straight Lines [id:da9082534880536912] 
\draw    (59.72,230.59) ;
\draw [shift={(59.72,230.59)}, rotate = 0] [color={rgb, 255:red, 0; green, 0; blue, 0 }  ][fill={rgb, 255:red, 0; green, 0; blue, 0 }  ][line width=0.75]      (0, 0) circle [x radius= 2.01, y radius= 2.01]   ;
%Straight Lines [id:da5200043203605025] 
\draw [line width=0.75]    (238.25,106.28) -- (238.25,124.43) -- (238.25,141.39) ;
%Straight Lines [id:da29915378959711836] 
\draw [line width=0.75]    (238.25,71.17) -- (238.25,89.32) -- (238.25,106.28) ;
%Straight Lines [id:da5575915537433072] 
\draw [line width=0.75]    (238.25,36.07) -- (238.25,54.21) -- (238.25,71.17) ;
%Straight Lines [id:da17419500905099583] 
\draw    (238.25,36.07) ;
\draw [shift={(238.25,36.07)}, rotate = 0] [color={rgb, 255:red, 0; green, 0; blue, 0 }  ][fill={rgb, 255:red, 0; green, 0; blue, 0 }  ][line width=0.75]      (0, 0) circle [x radius= 2.01, y radius= 2.01]   ;
%Straight Lines [id:da8114566270597321] 
\draw    (238.25,71.17) ;
\draw [shift={(238.25,71.17)}, rotate = 0] [color={rgb, 255:red, 0; green, 0; blue, 0 }  ][fill={rgb, 255:red, 0; green, 0; blue, 0 }  ][line width=0.75]      (0, 0) circle [x radius= 2.01, y radius= 2.01]   ;
%Straight Lines [id:da01741745919057347] 
\draw    (238.25,106.28) ;
\draw [shift={(238.25,106.28)}, rotate = 0] [color={rgb, 255:red, 0; green, 0; blue, 0 }  ][fill={rgb, 255:red, 0; green, 0; blue, 0 }  ][line width=0.75]      (0, 0) circle [x radius= 2.01, y radius= 2.01]   ;
%Straight Lines [id:da6151058927139037] 
\draw [line width=0.75]    (119.15,114.94) -- (116.08,112) -- (92.75,89.42) ;
%Straight Lines [id:da5438430495557162] 
\draw [line width=0.75]    (145.54,140.45) -- (142.48,137.51) -- (119.15,114.94) ;
%Straight Lines [id:da6861336944131907] 
\draw [line width=0.75]    (330.89,321.25) -- (327.82,318.31) -- (304.49,295.73) ;
%Straight Lines [id:da270737472520048] 
\draw [line width=0.75]    (357.28,346.76) -- (354.22,343.82) -- (330.89,321.25) ;
%Straight Lines [id:da29014232679721974] 
\draw [line width=0.75]    (383.76,372.17) -- (380.69,369.23) -- (357.36,346.65) ;
%Straight Lines [id:da2665156202388961] 
\draw [color={rgb, 255:red, 74; green, 144; blue, 226 }  ,draw opacity=1 ]   (633.3,141.78) ;
\draw [shift={(633.3,141.78)}, rotate = 0] [color={rgb, 255:red, 74; green, 144; blue, 226 }  ,draw opacity=1 ][fill={rgb, 255:red, 74; green, 144; blue, 226 }  ,fill opacity=1 ][line width=0.75]      (0, 0) circle [x radius= 2.01, y radius= 2.01]   ;
%Straight Lines [id:da1286647788680363] 
\draw [color={rgb, 255:red, 74; green, 144; blue, 226 }  ,draw opacity=1 ]   (633.3,124.98) ;
\draw [shift={(633.3,124.98)}, rotate = 0] [color={rgb, 255:red, 74; green, 144; blue, 226 }  ,draw opacity=1 ][fill={rgb, 255:red, 74; green, 144; blue, 226 }  ,fill opacity=1 ][line width=0.75]      (0, 0) circle [x radius= 2.01, y radius= 2.01]   ;
%Rounded Rect [id:dp9914554986060775] 
\draw   (492.1,65.97) .. controls (492.1,53.97) and (501.83,44.24) .. (513.83,44.24) -- (669.07,44.24) .. controls (681.07,44.24) and (690.8,53.97) .. (690.8,65.97) -- (690.8,131.17) .. controls (690.8,143.18) and (681.07,152.91) .. (669.07,152.91) -- (513.83,152.91) .. controls (501.83,152.91) and (492.1,143.18) .. (492.1,131.17) -- cycle ;
%Straight Lines [id:da2959548020210274] 
\draw [color={rgb, 255:red, 74; green, 144; blue, 226 }  ,draw opacity=1 ]   (635.19,106.12) ;
\draw [shift={(635.19,106.12)}, rotate = 0] [color={rgb, 255:red, 74; green, 144; blue, 226 }  ,draw opacity=1 ][fill={rgb, 255:red, 74; green, 144; blue, 226 }  ,fill opacity=1 ][line width=0.75]      (0, 0) circle [x radius= 2.01, y radius= 2.01]   ;
%Straight Lines [id:da3924071779344177] 
\draw [color={rgb, 255:red, 74; green, 144; blue, 226 }  ,draw opacity=1 ]   (635.19,89.32) ;
\draw [shift={(635.19,89.32)}, rotate = 0] [color={rgb, 255:red, 74; green, 144; blue, 226 }  ,draw opacity=1 ][fill={rgb, 255:red, 74; green, 144; blue, 226 }  ,fill opacity=1 ][line width=0.75]      (0, 0) circle [x radius= 2.01, y radius= 2.01]   ;

% Text Node
\draw (148.38,231.82) node [anchor=north] [inner sep=0.75pt]  [font=\tiny]  {$1$};
% Text Node
\draw (166.57,231.75) node [anchor=north] [inner sep=0.75pt]  [font=\tiny]  {$0$};
% Text Node
\draw (173.65,286.3) node [anchor=north] [inner sep=0.75pt]  [font=\tiny]  {$1$};
% Text Node
\draw (191.46,286.2) node [anchor=north] [inner sep=0.75pt]  [font=\tiny]  {$0$};
% Text Node
\draw (228.58,308.92) node [anchor=north] [inner sep=0.75pt]  [font=\tiny]  {$1$};
% Text Node
\draw (249.15,309.24) node [anchor=north] [inner sep=0.75pt]  [font=\tiny]  {$0$};
% Text Node
\draw (173.2,176.29) node [anchor=north] [inner sep=0.75pt]  [font=\tiny]  {$1$};
% Text Node
\draw (192.19,176.41) node [anchor=north] [inner sep=0.75pt]  [font=\tiny]  {$0$};
% Text Node
\draw (285.5,286.62) node [anchor=north] [inner sep=0.75pt]  [font=\tiny]  {$1$};
% Text Node
\draw (305.28,286.94) node [anchor=north] [inner sep=0.75pt]  [font=\tiny]  {$0$};
% Text Node
\draw (310.98,231.69) node [anchor=north] [inner sep=0.75pt]  [font=\tiny]  {$1$};
% Text Node
\draw (328.78,231.61) node [anchor=north] [inner sep=0.75pt]  [font=\tiny]  {$0$};
% Text Node
\draw (286.01,176.59) node [anchor=north] [inner sep=0.75pt]  [font=\tiny]  {$1$};
% Text Node
\draw (303.8,176.52) node [anchor=north] [inner sep=0.75pt]  [font=\tiny]  {$0$};
% Text Node
\draw (230.8,152.93) node [anchor=north] [inner sep=0.75pt]  [font=\tiny]  {$1$};
% Text Node
\draw (248.39,153.05) node [anchor=north] [inner sep=0.75pt]  [font=\tiny]  {$0$};
% Text Node
\draw (117.15,118.34) node [anchor=north east] [inner sep=0.75pt]  [font=\tiny]  {$0$};
% Text Node
\draw (90.75,92.82) node [anchor=north east] [inner sep=0.75pt]  [font=\tiny]  {$0$};
% Text Node
\draw (143.54,143.85) node [anchor=north east] [inner sep=0.75pt]  [font=\tiny]  {$0$};
% Text Node
\draw (359.28,343.36) node [anchor=south west] [inner sep=0.75pt]  [font=\tiny]  {$1$};
% Text Node
\draw (332.89,317.85) node [anchor=south west] [inner sep=0.75pt]  [font=\tiny]  {$0$};
% Text Node
\draw (240.25,425.11) node [anchor=west] [inner sep=0.75pt]  [font=\tiny]  {$1$};
% Text Node
\draw (576.77,56.19) node [anchor=north west][inner sep=0.75pt]   [align=left] {means};
% Text Node
\draw (575.21,90.91) node [anchor=north west][inner sep=0.75pt]   [align=left] {means};
% Text Node
\draw (573.77,126.51) node [anchor=north west][inner sep=0.75pt]   [align=left] {means};
% Text Node
\draw (120.51,348.42) node [anchor=north west][inner sep=0.75pt]  [font=\tiny]  {$1$};
% Text Node
\draw (516.17,21.91) node [anchor=north west][inner sep=0.75pt]   [align=left] {Key};
% Text Node
\draw (20.3,233.99) node [anchor=north] [inner sep=0.75pt]  [font=\tiny]  {$0$};
% Text Node
\draw (59.72,233.99) node [anchor=north] [inner sep=0.75pt]  [font=\tiny]  {$0$};
% Text Node
\draw (99.13,233.99) node [anchor=north] [inner sep=0.75pt]  [font=\tiny]  {$1$};
% Text Node
\draw (94.48,373.4) node [anchor=north west][inner sep=0.75pt]  [font=\tiny]  {$0$};
% Text Node
\draw (146.54,323.44) node [anchor=north west][inner sep=0.75pt]  [font=\tiny]  {$0$};
% Text Node
\draw (240.25,390) node [anchor=west] [inner sep=0.75pt]  [font=\tiny]  {$0$};
% Text Node
\draw (240.25,354.89) node [anchor=west] [inner sep=0.75pt]  [font=\tiny]  {$0$};
% Text Node
\draw (385.68,368.87) node [anchor=south west] [inner sep=0.75pt]  [font=\tiny]  {$1$};
% Text Node
\draw (377.38,233.99) node [anchor=north] [inner sep=0.75pt]  [font=\tiny]  {$1$};
% Text Node
\draw (456.21,233.99) node [anchor=north] [inner sep=0.75pt]  [font=\tiny]  {$1$};
% Text Node
\draw (416.79,233.99) node [anchor=north] [inner sep=0.75pt]  [font=\tiny]  {$0$};
% Text Node
\draw (384.89,94.43) node [anchor=north west][inner sep=0.75pt]  [font=\tiny]  {$0$};
% Text Node
\draw (332.83,144.4) node [anchor=north west][inner sep=0.75pt]  [font=\tiny]  {$1$};
% Text Node
\draw (358.86,119.42) node [anchor=north west][inner sep=0.75pt]  [font=\tiny]  {$1$};
% Text Node
\draw (240.25,106.28) node [anchor=west] [inner sep=0.75pt]  [font=\tiny]  {$1$};
% Text Node
\draw (240.25,71.17) node [anchor=west] [inner sep=0.75pt]  [font=\tiny]  {$1$};
% Text Node
\draw (240.25,36.07) node [anchor=west] [inner sep=0.75pt]  [font=\tiny]  {$1$};

\end{tikzpicture}

\caption{Figure 3. The graph $H_8$. $ \ \ \ \ \ \ \ \ \ \ \ \  \ \ \ \ \ \ \ \ \ \ \ \ \ \ \ \ \ \ \ \ \  $}
    \label{Fig_3}
\end{figure}

Every determining set must contain at least one node from each clone-pair, since the two nodes in a clone-pair can be interchanged in an automorphism while all other nodes are fixed. Conversely, choosing one vertex from every clone-pair gives a determining set. Hence, $Det(H_n) = n$. 
		
		Let us consider coloring this graph and without loss of generality we use 1s as the minimum color in a coloring. Now in each clone-pair, one node must be colored 1 and the other 0 as they share the same neighbourhood (i.e., no other nodes can be used to distinguish them). This contributes $n$ to the cost. Moreover, to distinguish the  different pairs of clones
		 from each other we have to use the attached paths.   For two pairs of clones, if their attached paths have the same labeling, observe that there is a color-preserving automorphism which swaps  those pairs and their attached paths and leaves the other nodes invariant. Therefore any distinguishing coloring would use different 0-1 $k$ length strings.  To get a coloring using the minimum number of 1s, let $f(n)$ be the minimum total number of 1s that is among $n$ different ${\lceil log_2{n} \rceil}$-length strings. Then, choose $n$ distinct binary strings of length $k$ whose total number of 1s is $f(n)$, and use these strings to color the attached paths. The resulting distinguishing coloring is minimal in the number of 1's used and $\rho(H_n)= n + f(n)$.

		 To illustrate this, consider $H_{2^k}.$ Since there are $2^k$, $k$-length $0$-$1$ sequences, this can just be done. This colors $\frac{1}{2} \cdot 2^k \cdot k$ nodes color 0 and   $\frac{1}{2}  \cdot 2^k \cdot k$ nodes color 1. As this tightly achieves different names for the clone-pairs of the clique subgraph, and no other type of 2-coloring will do this (apart from permutations of the colorings of the attached paths), therefore $\rho( H_{2^k}) = k2^{k-1} + 2^k$ (i.e. any 2-distinguishing coloring uses all $2^k$ strings on the attached paths in some order). For example, if $k = 3$, then this is a graph with $Det(H_8) = 8$ and $\rho(H_8) = 20$, see Figure \ref{Fig_3}. This class of examples shows that in general, if the cost of graphs with $Det(G) = n $ is bounded, then this bound is $\geq n + f(n) $.

\end{example}		 

\noindent  Next, we conjecture that these examples are the extremal cases.

	\begin{conjecture}
{\em Let $G$ be a 2-distinguishable graph with $Det(G) = n$. Let $f(n)$ denote the minimum total number of 1s among any $n$  distinct binary strings of length $\lceil log_2{n} \rceil$.

\noindent Then,\\
 \vspace{-10mm}

 $$\rho(G) \leq n + f(n).$$ 

  %\vspace{-1.5mm}
 
\noindent  Moreover, this bound is sharp.} 
\end{conjecture}

\noindent If this conjecture is true it follows from it that: if $Det(G)=1$ then $\rho(G) = 1$ (known); if $Det(G)=2$ then $\rho(G) \leq 3$ (indeed while we showed here that $\rho(G) \leq 4$ we could not find an example with $\rho(G) = 4$); if $Det(G)=3$ then $\rho(G) \leq 5$ and so on.

	\vspace{-1mm}

	\section{Open Problems and Future Work}
    \vspace{2mm}
		
		\noindent 	1. For $Det(G) = 3$, show that $\rho(G) \leq 5$ or find a graph a counterexample graph.\\	
        
		\noindent   2. We proved that for a graph $G$ with $D(G) = 2$ and $Det(G) = 2$, that $\rho(G) = 2, 3$ or $4$. It is easy to give examples for $\rho(G)$ being $2$ or $3$,  see the graphs $G_1$ and $G_2$ below. The blue circles represent (minimal) determining sets, while the reds represent  cost nodes.

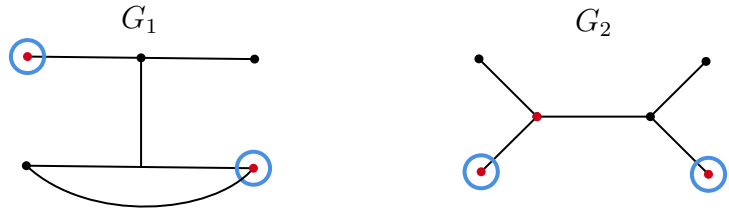
\begin{center}

\tikzset{every picture/.style={line width=0.75pt}} %set default line width to 0.75pt        

\begin{tikzpicture}[x=0.65pt,y=0.65pt,yscale=-1,xscale=1]

%Straight Lines [id:da7651186705992465] 
\draw    (368,131.33) -- (433.33,131.33) ;
\draw [shift={(433.33,131.33)}, rotate = 0] [color={rgb, 255:red, 0; green, 0; blue, 0 }  ][fill={rgb, 255:red, 0; green, 0; blue, 0 }  ][line width=0.75]      (0, 0) circle [x radius= 2.01, y radius= 2.01]   ;
\draw [shift={(368,131.33)}, rotate = 0] [color={rgb, 255:red, 0; green, 0; blue, 0 }  ][fill={rgb, 255:red, 0; green, 0; blue, 0 }  ][line width=0.75]      (0, 0) circle [x radius= 2.01, y radius= 2.01]   ;
%Shape: Right Angle [id:dp5309109421824032] 
\draw   (334.59,97.81) -- (368,131.33) -- (335.89,163.33) ;
%Straight Lines [id:da1602927512836314] 
\draw    (334.59,97.81) ;
\draw [shift={(334.59,97.81)}, rotate = 0] [color={rgb, 255:red, 0; green, 0; blue, 0 }  ][fill={rgb, 255:red, 0; green, 0; blue, 0 }  ][line width=0.75]      (0, 0) circle [x radius= 2.01, y radius= 2.01]   ;
%Straight Lines [id:da9530679990430042] 
\draw [color={rgb, 255:red, 208; green, 2; blue, 27 }  ,draw opacity=1 ]   (335.89,163.33) ;
\draw [shift={(335.89,163.33)}, rotate = 0] [color={rgb, 255:red, 208; green, 2; blue, 27 }  ,draw opacity=1 ][fill={rgb, 255:red, 208; green, 2; blue, 27 }  ,fill opacity=1 ][line width=0.75]      (0, 0) circle [x radius= 2.01, y radius= 2.01]   ;
%Straight Lines [id:da4867832229309761] 
\draw    (465.39,99.28) ;
\draw [shift={(465.39,99.28)}, rotate = 0] [color={rgb, 255:red, 0; green, 0; blue, 0 }  ][fill={rgb, 255:red, 0; green, 0; blue, 0 }  ][line width=0.75]      (0, 0) circle [x radius= 2.01, y radius= 2.01]   ;
%Shape: Right Angle [id:dp9538766642461778] 
\draw   (466.8,164.8) -- (433.33,131.33) -- (465.39,99.28) ;
%Straight Lines [id:da2053614049297341] 
\draw    (74.59,96.47) -- (205.39,97.94) ;
%Straight Lines [id:da5242070057682577] 
\draw    (73.92,159.81) -- (204.72,161.28) ;
%Straight Lines [id:da7289635410955839] 
\draw [color={rgb, 255:red, 208; green, 2; blue, 27 }  ,draw opacity=1 ]   (74.59,96.47) ;
\draw [shift={(74.59,96.47)}, rotate = 0] [color={rgb, 255:red, 208; green, 2; blue, 27 }  ,draw opacity=1 ][fill={rgb, 255:red, 208; green, 2; blue, 27 }  ,fill opacity=1 ][line width=0.75]      (0, 0) circle [x radius= 2.01, y radius= 2.01]   ;
%Straight Lines [id:da22213332160538046] 
\draw    (205.39,97.94) ;
\draw [shift={(205.39,97.94)}, rotate = 0] [color={rgb, 255:red, 0; green, 0; blue, 0 }  ][fill={rgb, 255:red, 0; green, 0; blue, 0 }  ][line width=0.75]      (0, 0) circle [x radius= 2.01, y radius= 2.01]   ;
%Straight Lines [id:da3347928879231019] 
\draw    (139.99,97.21) ;
\draw [shift={(139.99,97.21)}, rotate = 0] [color={rgb, 255:red, 0; green, 0; blue, 0 }  ][fill={rgb, 255:red, 0; green, 0; blue, 0 }  ][line width=0.75]      (0, 0) circle [x radius= 2.01, y radius= 2.01]   ;
%Straight Lines [id:da7533908752957996] 
\draw    (139.99,97.21) -- (139.99,160.54) ;
%Straight Lines [id:da6228468717486293] 
\draw    (73.92,159.81) ;
\draw [shift={(73.92,159.81)}, rotate = 0] [color={rgb, 255:red, 0; green, 0; blue, 0 }  ][fill={rgb, 255:red, 0; green, 0; blue, 0 }  ][line width=0.75]      (0, 0) circle [x radius= 2.01, y radius= 2.01]   ;
%Shape: Circle [id:dp44383885757981956] 
\draw  [color={rgb, 255:red, 74; green, 144; blue, 226 }  ,draw opacity=1 ][line width=1.5]  (65.02,96.47) .. controls (65.02,91.19) and (69.3,86.9) .. (74.59,86.9) .. controls (79.87,86.9) and (84.16,91.19) .. (84.16,96.47) .. controls (84.16,101.76) and (79.87,106.04) .. (74.59,106.04) .. controls (69.3,106.04) and (65.02,101.76) .. (65.02,96.47) -- cycle ;
%Shape: Circle [id:dp16081885380007743] 
\draw  [color={rgb, 255:red, 74; green, 144; blue, 226 }  ,draw opacity=1 ][line width=1.5]  (195.15,161.28) .. controls (195.15,155.99) and (199.44,151.71) .. (204.72,151.71) .. controls (210.01,151.71) and (214.29,155.99) .. (214.29,161.28) .. controls (214.29,166.56) and (210.01,170.85) .. (204.72,170.85) .. controls (199.44,170.85) and (195.15,166.56) .. (195.15,161.28) -- cycle ;
%Shape: Circle [id:dp3798379288234268] 
\draw  [color={rgb, 255:red, 74; green, 144; blue, 226 }  ,draw opacity=1 ][line width=1.5]  (326.32,163.33) .. controls (326.32,158.05) and (330.6,153.76) .. (335.89,153.76) .. controls (341.17,153.76) and (345.46,158.05) .. (345.46,163.33) .. controls (345.46,168.62) and (341.17,172.9) .. (335.89,172.9) .. controls (330.6,172.9) and (326.32,168.62) .. (326.32,163.33) -- cycle ;
%Shape: Circle [id:dp41646745707333577] 
\draw  [color={rgb, 255:red, 74; green, 144; blue, 226 }  ,draw opacity=1 ][line width=1.5]  (457.23,164.8) .. controls (457.23,159.52) and (461.52,155.23) .. (466.8,155.23) .. controls (472.09,155.23) and (476.37,159.52) .. (476.37,164.8) .. controls (476.37,170.09) and (472.09,174.37) .. (466.8,174.37) .. controls (461.52,174.37) and (457.23,170.09) .. (457.23,164.8) -- cycle ;
%Straight Lines [id:da6629303995219997] 
\draw [color={rgb, 255:red, 208; green, 2; blue, 27 }  ,draw opacity=1 ]   (466.8,164.8) ;
\draw [shift={(466.8,164.8)}, rotate = 0] [color={rgb, 255:red, 208; green, 2; blue, 27 }  ,draw opacity=1 ][fill={rgb, 255:red, 208; green, 2; blue, 27 }  ,fill opacity=1 ][line width=0.75]      (0, 0) circle [x radius= 2.01, y radius= 2.01]   ;
%Straight Lines [id:da835820167193631] 
\draw [color={rgb, 255:red, 208; green, 2; blue, 27 }  ,draw opacity=1 ]   (368,131.33) ;
\draw [shift={(368,131.33)}, rotate = 0] [color={rgb, 255:red, 208; green, 2; blue, 27 }  ,draw opacity=1 ][fill={rgb, 255:red, 208; green, 2; blue, 27 }  ,fill opacity=1 ][line width=0.75]      (0, 0) circle [x radius= 2.01, y radius= 2.01]   ;
%Curve Lines [id:da21956368292767126] 
\draw    (73.92,159.81) .. controls (109.83,193.5) and (181.17,188.83) .. (204.72,161.28) ;
%Straight Lines [id:da11708201038972854] 
\draw [color={rgb, 255:red, 208; green, 2; blue, 27 }  ,draw opacity=1 ]   (204.72,161.28) ;
\draw [shift={(204.72,161.28)}, rotate = 0] [color={rgb, 255:red, 208; green, 2; blue, 27 }  ,draw opacity=1 ][fill={rgb, 255:red, 208; green, 2; blue, 27 }  ,fill opacity=1 ][line width=0.75]      (0, 0) circle [x radius= 2.01, y radius= 2.01]   ;

% Text Node
\draw (139.32,84.47) node [anchor=south] [inner sep=0.75pt]    {$G_{1}$};
% Text Node
\draw (400.83,86.14) node [anchor=south] [inner sep=0.75pt]    {$G_{2}$};

\end{tikzpicture}

\end{center}
        
However, we could not find an example where $\rho(G) = 4$. Find such an example or show that one does not exist (this would be consistent with the conjecture). \\
		
		\noindent	3. Find and classify sets of  distinguishably equivalent graphs. Graphs which satisfy  $Aut(G, V(G)_l)  = \{e\} $, where $e =(1)\dots (n)$ and $l: V(G) \rightarrow \{1, \dots n\}$ for example,   are the asymmetric graphs on $n$ nodes. Moreover,  while a graph and its complement are  distinguishably equivalent, many more possibilities exist. See Example \ref{E1}, for  two  distinguishably equivalent graphs which are not asymmetric or complements of each other. What are some other cases? 

        \vspace{-4mm}

\bibliographystyle{cost_paper}
\bibliography{cost_paper}

\vspace{6mm}		
		\noindent Alexa Gopaulsingh: Department of Logic, Eötvös Loránd University, Budapest, Hungary\\
		Email: alexa279e@gmail.com \\

		\vspace{-2mm}

		\noindent Zalán Molnár: Department of Logic, Eötvös Loránd University, Budapest, Hungary\\
		Email: mozaag@gmail.com\\

			\vspace{-2mm}

        \noindent Amitayu Banerjee: Department of Logic, Eötvös Loránd University, Budapest, Hungary\\
        Email: banerjee.amitayu@gmail.com
        
		%\noindent 4. Which group of permutations are equivalent to the automorphism representation group of some graph? For eg. it can be shown that $\{e, (12)(34), (13)(24), (14)(23)\}$ does not correspond to the automorphisms of any graph on 4 nodes labelled $\{1, 2, 3, 4\}$. \\

\end{document}